\documentclass[11pt,a4paper,dvipsnames,eqrno]{article}
\usepackage[utf8]{inputenc}

\title{\textbf{The Best Ways to Slice a Polytope}}
\author{Marie-Charlotte Brandenburg,  Jes\'us A. De Loera, and Chiara Meroni}

\date{}

\usepackage{a4wide}
\usepackage{amssymb}
\usepackage{amsmath}
\usepackage{euscript}
\usepackage{amsthm}
\usepackage{thmtools}
\usepackage{mathtools}
\usepackage{amsopn}
\usepackage{stackrel}
\usepackage[all]{xy}
\usepackage{setspace}
\usepackage{float}
\usepackage{multirow}
\usepackage{blkarray}
\usepackage[dvipsnames]{xcolor}
\usepackage{etoolbox}
\usepackage{enumerate}
\usepackage{enumitem}
\usepackage{verbatim}
\usepackage{comment}
\usepackage{array}
\definecolor{cb-black}      {RGB}{  0,   0,   0}
\definecolor{cb-blue-green} {RGB}{  0,  073,  073}
\definecolor{cb-green-sea}  {RGB}{  0, 146, 146}
\definecolor{cb-rose}       {RGB}{255, 109, 182}
\definecolor{cb-salmon-pink}{RGB}{255, 182, 119}
\definecolor{cb-purple}     {RGB}{ 73,   0, 146}
\definecolor{cb-blue}       {RGB}{ 0, 109, 219}
\definecolor{cb-lilac}      {RGB}{182, 109, 255}
\definecolor{cb-blue-sky}   {RGB}{109, 182, 255}
\definecolor{cb-blue-light} {RGB}{182, 219, 255}
\definecolor{cb-burgundy}   {RGB}{146,   0,   0}
\definecolor{cb-brown}      {RGB}{146,  73,   0}
\definecolor{cb-clay}       {RGB}{219, 209,   0}
\definecolor{cb-green-lime} {RGB}{ 36, 255,  36}
\definecolor{cb-yellow}     {RGB}{255, 255, 109}

\usepackage{relsize}
\usepackage[pdftex, colorlinks, linkcolor=blue, citecolor=blue, urlcolor=blue]{hyperref}
\hypersetup{linkcolor=cb-green-sea, citecolor=cb-green-sea, urlcolor=cb-green-sea}
\usepackage{pgfplots}
\usepackage[margin=1cm]{caption}
\usepackage{subcaption}
\pgfplotsset{width=7cm, compat=1.10}
\usepgfplotslibrary{fillbetween}
\usepackage[ruled,vlined]{algorithm2e}
\usetikzlibrary{matrix}
\usepackage{mathdots}
\usepackage[nameinlink]{cleveref}
\usepackage{soul}
\usepackage{algorithmic}

\usepackage{todonotes}

\newcommand{\ma}{\begin{pmatrix}}
\newcommand{\trix}{\end{pmatrix}}
\newcommand{\sma}{\left(\begin{smallmatrix}}
\newcommand{\strix}{\end{smallmatrix}\right)}

\newcommand{\bu}{\mathbf{u}}
\newcommand{\bx}{\mathbf{x}}
\newcommand{\bt}{\mathbf{t}}
\newcommand{\bv}{\mathbf{v}}
\newcommand{\ba}{\mathbf{a}}
\newcommand{\bb}{\mathbf{b}}
\newcommand{\bs}{\mathbf{s}}
\newcommand{\bp}{\mathbf{p}}
\newcommand{\bk}{\mathbf{k}}
\newcommand{\balpha}{\boldsymbol{\alpha}}

\newcommand{\R}{\mathbb{R}}
\newcommand{\Z}{\mathbb{Z}}

\newcommand{\Q}{\mathbb{Q}}

\newcommand{\dd}{\mathrm{d}}

\DeclareMathOperator{\vol}{vol}

\DeclareMathOperator{\conv}{conv}
\newcommand{\dx}{\ \dd \bx}

\newcommand{\regtra}{\mathcal R_{\text{\rotatebox{-15}{$\scriptscriptstyle\uparrow$}}}\text{\hspace*{-.07em}}}
\newcommand{\regrad}{\mathcal R_\circlearrowleft}
\newcommand{\chamtra}{\mathcal C_{\text{\rotatebox{-15}{$\scriptscriptstyle\uparrow$}}}^\bu\text{\hspace*{-.07em}}}
\newcommand{\trarrow}{
    \text{\rotatebox{-15}{$\scriptscriptstyle\uparrow$}}
    \text{\hspace*{-.07em}}
}
\newcommand{\chamtraprime}{\mathcal C_{\text{\rotatebox{-15}{$\scriptscriptstyle\uparrow$}}}^{\bu'}\text{\hspace*{-.07em}}}
\newcommand{\chamrad}{\mathcal C_\circlearrowleft}

\newtheorem{theorem}{Theorem}[section]

\newtheorem{lemma}[theorem]{Lemma}
\newtheorem{prop}[theorem]{Proposition}

\newtheorem*{conjs*}{Conjectures}

\Crefname{prop}{Proposition}{Propositions}
\Crefname{subsection}{Subsection}{Subsections}

\theoremstyle{definition}

\declaretheoremstyle[
  qed=$\diamond$,
  sibling=theorem
]{mythmstyle}
\declaretheorem[style=mythmstyle]{example}

\theoremstyle{remark}
\newtheorem{remark}[theorem]{Remark}

\usepackage[backend=bibtex,style=alphabetic]{biblatex}
\AtBeginBibliography{
    \hypersetup{urlcolor=gray} \small %\scriptsize %\footnotesize
}

\allowdisplaybreaks

\makeatletter
\def\keywords{\xdef\@thefnmark{}\@footnotetext}
\makeatother

\makeatletter
\def\mscclasses{\xdef\@thefnmark{}\@footnotetext}
\makeatother

\begin{document}

\maketitle

\mscclasses{\hspace*{-2.3em} MSC classes:
52B55, %Computational aspects related to convexity (projections, polytopes, etc.)
52C35, %Arrangements of points, flats, hyperplanes (aspects of discrete geometry)$
52A38, %Length, area, volume and convex sets (aspects of convex geometry)
52A40, %Inequalities and extremum problems
52B11, %n-dimensional polytopes
90C27, %Combinatorial optimization
52C45, %Combinatorial Geometry
14P10. %Semialgberaic sets and related spaces
}
\keywords{\hspace*{-2.3em} Keywords: polytopes, hyperplane sections, hyperplane arrangements, optimal slices, volume, integration over polyhedral regions, combinatorial types of polytopes, extremal problems on polytopes.}

{\vspace*{-0.5em}\centerline{ 
\emph{Dedicated to G\"unter M. Ziegler on the occasion of his $\mathit{60}^{\text{th}}$ birthday.}}

\begin{abstract}
    We study the structure of the set of all possible affine hyperplane sections of a convex polytope. We present two different cell decompositions of this set, induced by hyperplane arrangements. Using our decomposition, we bound the number of possible combinatorial  types of sections and craft algorithms that compute optimal sections of the polytope according to various combinatorial and metric criteria, including sections that maximize the number of $k$-dimensional faces, maximize the volume, and maximize the integral of a polynomial. Our optimization algorithms run in polynomial  time in fixed dimension, but the same problems show hardness otherwise. Our tools can be extended to intersection with halfspaces and projections onto hyperplanes. Finally, we present several experiments illustrating our 
    theorems and algorithms on famous polytopes.
\end{abstract}

\section{Introduction} 
What is the best way to slice a $3$-dimensional permutahedron? \Cref{fig:permu} shows three of many possible ``best slices''. In this article we give a finite description of all affine hyperplane sections of an arbitrary polytope, called \emph{slices}, and 
explain how to compute the optimal ones. Our methods apply also to the investigation of the number of combinatorial types of slices. \\
\begin{figure}[!h]
    \centering
    \begin{subfigure}[t]{0.32\textwidth}
        \includegraphics[width=0.9\textwidth]{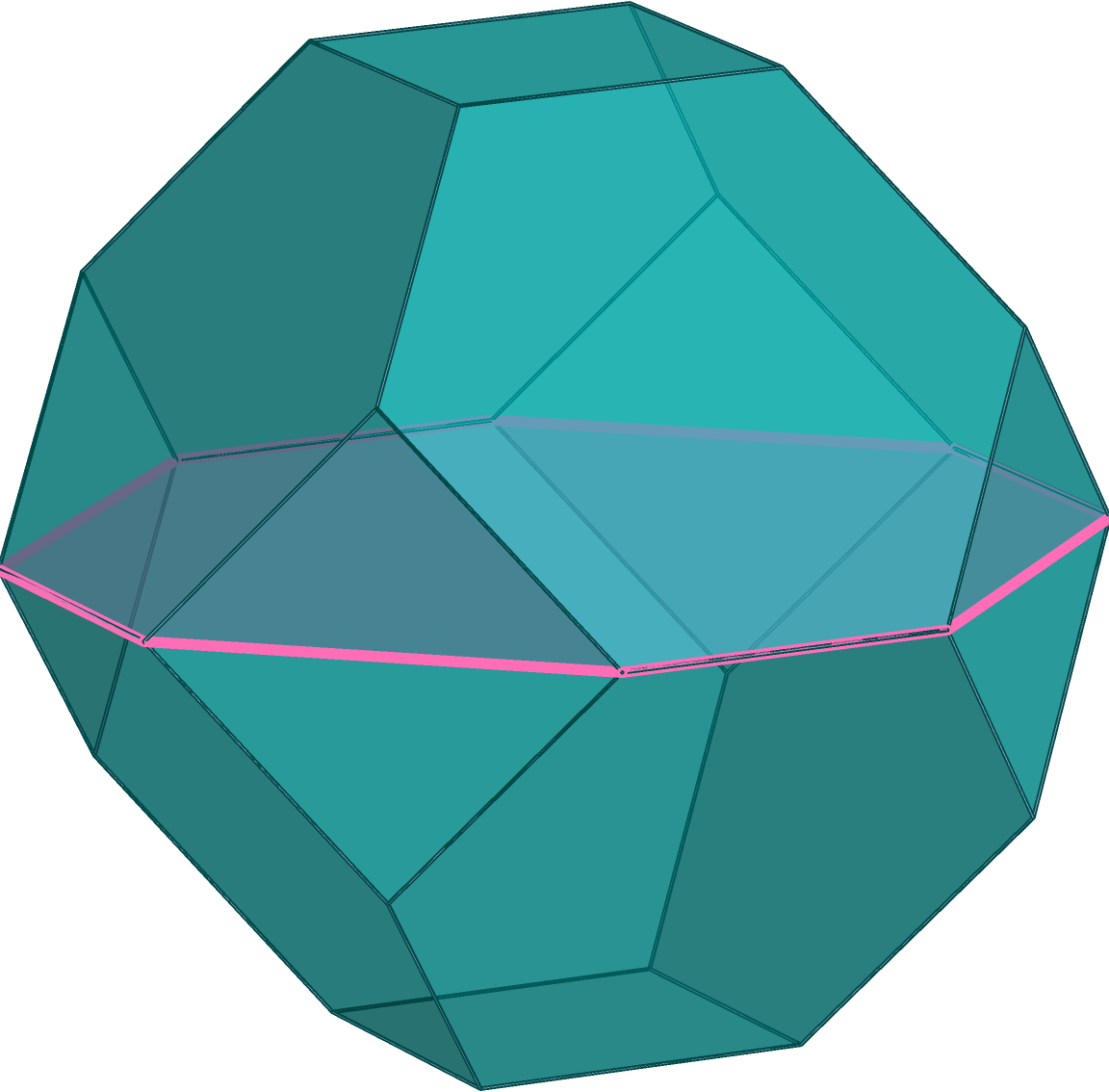}
        \caption{affine section of maximal\\ volume (unique up to symmetry)}
        \label{fig:permu-vol-max}
    \end{subfigure}
    \;
    \begin{subfigure}[t]{0.32\textwidth}
        \includegraphics[width=0.9\textwidth]{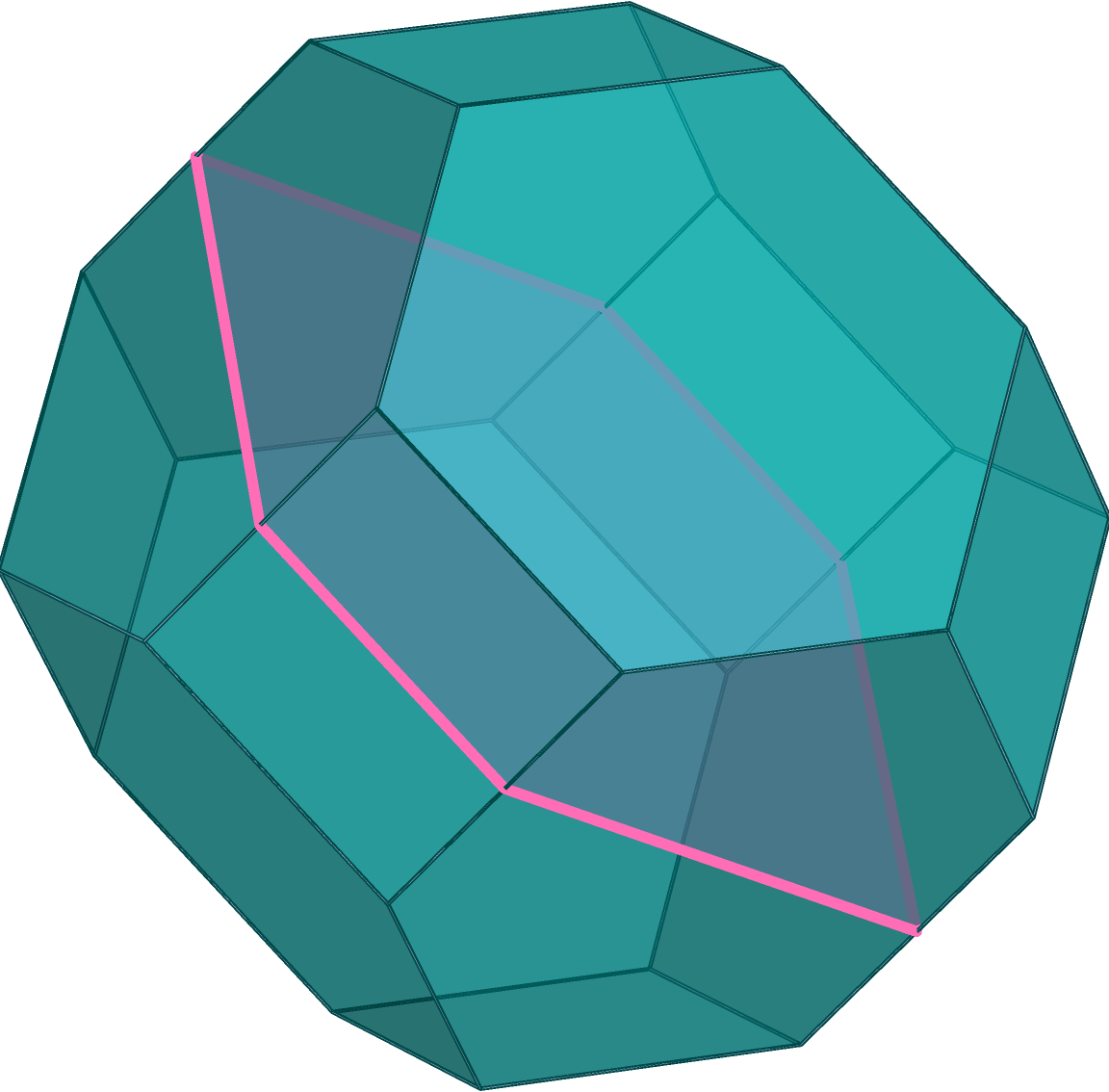}
        \caption{central section of minimal volume (unique up to symmetry)}
        \label{fig:permu-vol-min}
    \end{subfigure}
    \;
    \begin{subfigure}[t]{0.3\textwidth}
    \includegraphics[width=0.9\textwidth]{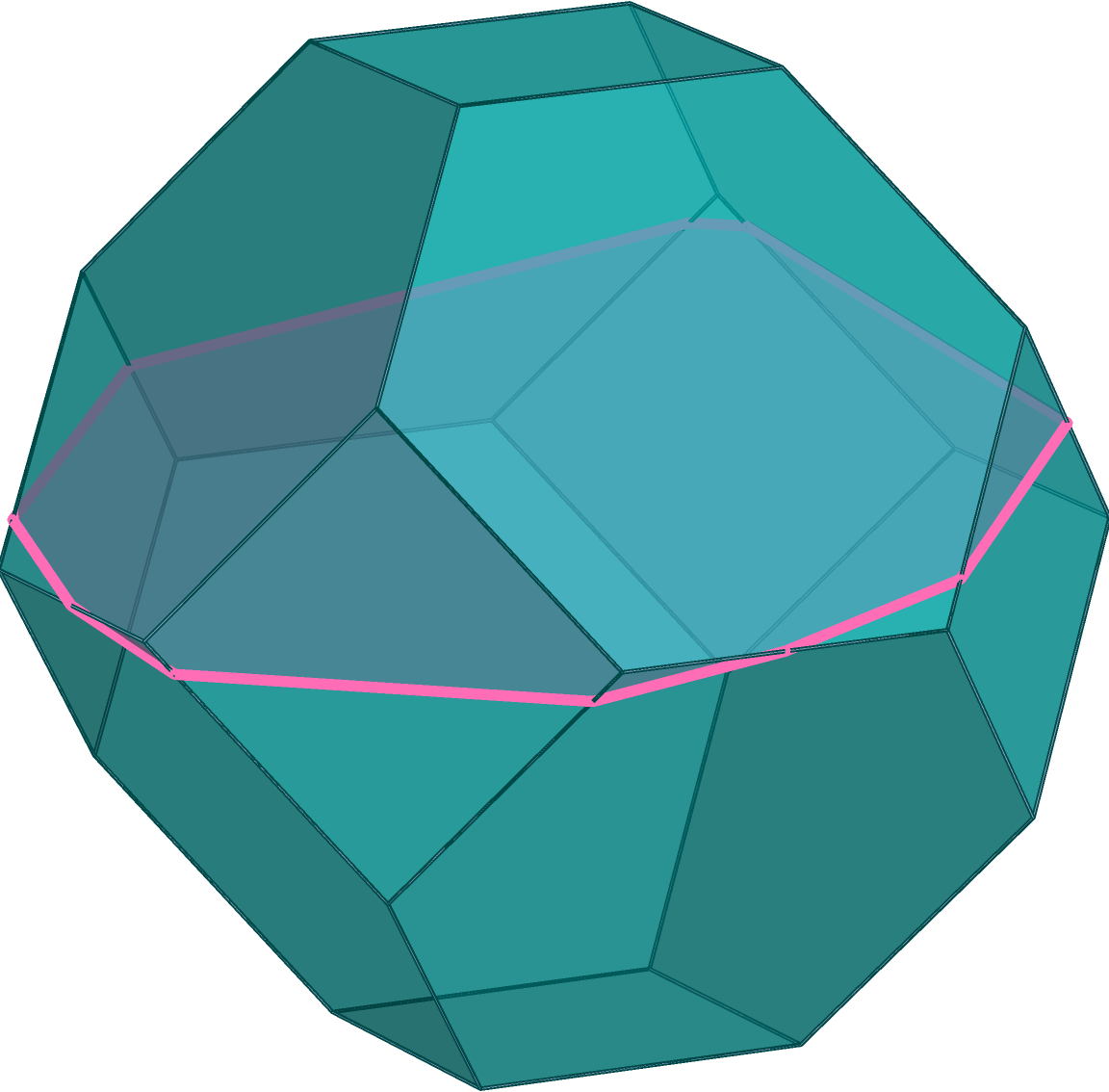}
    \caption{affine section with maximal number of vertices}
    \label{fig:permu-vcs}
    \end{subfigure}
    \caption{Three optimal sections of the permutahedron, the convex hull of permutations of $(1,2,3,4)$, shown in pink. See \Cref{ex:permu}.}
    \label{fig:permu}
\end{figure}

Sections or slices of convex bodies have been the focus of many researchers. To mention just a few highlights, a variety of problems from number theory can be rephrased as questions about volumes of sections or projections of convex bodies. For example, the sharp version of \emph{Siegel's Lemma} is equivalent to an estimate for the minimum volume of a central slice (of arbitrary dimension) of a unit cube \cite{vaaler79,Koldobskybook2005}. 
Similarly, the search for maximal volume slices of cubes, cross-polytopes, 
simplices and other convex bodies comes from applications in functional analysis and probability \cite{Kball1,Kball2,MeyerPajor88,webb,Koenig2021, pournin-scubesections}. 
Moreover, the research around the famous \emph{Busemann-Petty problem} and 
\emph{Bourgain's slicing conjecture} relates the volume of a convex body  to the volume of its sections by hyperplanes through the origin \cite{Gardnerbook,KlaMil22:SlicingProblem,Klartag:LogBoundBourgain,klartagLehec2022bourgains,GiannopoulosKoldobskyZvavitch2023,nayar+tkocz-2022extremalsecproj}.
Slices also play an important role in geometric tomography and inverse moment problems for convex sets  (see \cite[Chapter 8]{Gardnerbook} and \cite{Gravinetal2012-moments,KousholtSchulte2021moments,Walkup68:Simplex5D}). 
From a discrete point of view, the combinatorial analysis of hyperplane sections of polytopes, specially for polyhedral norm balls, has applications in algebraic and enumerative combinatorics \cite{Ardilaetal2021slicingpermutahedron,ChakerianLogothetti91,Khovanskii-slices2006,fukudaMutoGoto97-5cubeslices,Lawrence79,PadrolPfeifle2015}. 

\vskip .1cm
\noindent {\bf Our contributions.} 
The present paper addresses the 
fundamental question in computational convexity of finding optimal or extremal hyperplane sections of a polytope. Most work has concentrated on 
maximizing the volume of central hyperplane sections, while 
affine sections have been less studied and mostly for norm balls \cite{Moodyetal2013,Liu+Tkocz,Koenig2021,nayar+tkocz-2022extremalsecproj,pournin-scubesections}. Our results apply to both cases and seek to do research for more polytopes. 

Our first main contribution describes two different ways to parametrize all affine hyperplane sections of a polytope.
\begin{theorem}\label{thm:intro_structure}
Given a polytope $P \subset \R^d$, there exist two different parametric decompositions of the space 
of all affine hyperplanes in $\R^d$ into finitely many cells, called slicing chambers. Each decomposition is organized by a pair $(\mathcal R, \mathcal C)$ of arrangements (cf. \Cref{table:arrangements-overview}), where 
each region of $\mathcal R$ defines a parametric hyperplane arrangement 
$\mathcal C$. The following holds for slicing chambers in both decompositions:

\begin{enumerate}[label=\textup{(}\roman*\textup{)}]
\item Two affine hyperplanes $H_1,H_2$ belong to the same slicing chamber if and only if $H_1,H_2$ intersect the same set of edges of $P$. In particular, $P\cap H_1$ is combinatorially equivalent to $P \cap H_2$ and they admit the same triangulations.\label{item:comb-info}
\item For fixed dimension $d$, the number of slicing chambers is bounded by a polynomial in the number of vertices of $P$.\label{item:polynomial-size}
\item Restricted to a fixed slicing chamber, the integral $\int_{P\cap H} f(\bx) \dx$ of any polynomial is a rational function, which solely depends on the combinatorial information described in \labelcref{item:comb-info}. In particular, this holds for the volume of $P\cap H$.\label{item:integral-rational}
\end{enumerate}
\end{theorem}

\begin{table}[ht]
\centering
 {\def\arraystretch{1.5}
\begin{tabular}{c|c|c|c|c}
 & \textbf{Hyperplane Arrangement} & \textbf{Notation} & \textbf{Proofs in}  & \textbf{Reference Object} \\ \hline 
 \multirow{2}{*}{$\circlearrowleft$} &  central arrangement   & $\chamrad$             & \Cref{sec:rotational_slices}                   & intersection body \\
 & cocircuit arrangement & $\regrad$ & \Cref{sec:translating-the-rotation} & oriented matroid \\
 \hline
  \multirow{2}{*}{\rotatebox{-15}{$\small\uparrow$}} &  parallel arrangement   & $\chamtra$        & \Cref{sec:translational_slices}  & fiber polytope   \\
  & sweep arrangement  & $\regtra$ & \Cref{sec:rotating-the-translation} & sweep polytope         \\
\end{tabular}
}
\caption{An overview of the two pairs of hyperplane arrangements in this paper.}
\label{table:arrangements-overview}
\end{table}

We note the relevant prior work in the direction of \Cref{thm:intro_structure}. 
The central hyperplane arrangement in \Cref{table:arrangements-overview} already appeared in the 1970's in  \cite{Johnson+Preparata1978}, for an algorithm which finds a halfspace containing the maximum number of the points placed on the unit-sphere. Here, the halfspaces are bounded by central hyperplanes containing the origin.
Later, again in a central hyperplane setting, Filliman presented a decomposition of the Grassmannian of hyperplanes into cells, where rational function formulas for volumes of sections of central hyperplanes hold \cite{Filliman92}. However, his formulas are not suitable for general affine sections 
of polytopes.
More recently, \cite{BBMS:IntersectionBodiesPolytopes} used one of the decompositions we present in this paper to compute intersection bodies of polytopes.
In contrast, our techniques extend earlier work and are valid for \emph{all} affine sections. 
Using the structure of \Cref{thm:intro_structure}, we obtain formulas for the integral of polynomials (and hence the volume) over hyperplane sections or halfspace sections that are nonsingular rational functions in each slicing chamber. Our formulas rely on the integration formulas from \cite{baldoni11_howintegratepolynomial,lasserre01_multidimensionalversion}.

The slicing chambers are nicely organized. Each row of \Cref{table:arrangements-overview} points to the section which explains how to build them.
Our first slicing chamber decomposition is organized in terms of the \emph{cocircuit arrangement} generated by all hyperplanes spanned by sets of $d-1$ vertices of $P$. For each region of the cocircuit arrangement we identify a vector $\bt$ that we use to translate $P$, and from it we obtain a new central hyperplane for each vertex $\bv+\bt$ of the translated polytope $P+\bt$. This first approach was used in the study of intersection bodies of polytopes \cite{BraMer:IntersectionBodies}.
The second slicing chamber decomposition is organized differently.
This time we consider the regions of the \emph{sweep arrangement}, dual to the \emph{sweep polytope} \cite{PadrolPhilippe2021}. Each point $\bu$ in a region of the sweep arrangement identifies a direction. We then decompose $P$ into blocks defined by translations of $\bu^\perp$, while maintaining the combinatorics. These slabs are our second type of slicing chambers. They 
are the same pieces used to compute the 
\emph{monotone path polytope}
of $P$ \cite{BilleraKapranovSturmfelsMPPs1994}, a special instance of fiber polytopes \cite{BilStu:FiberPolytopes}. 

There are many applications of \Cref{thm:intro_structure}. Regarding combinatorial 
applications, we can use slicing chambers to bound the number of combinatorial types of slices. This is an interesting but hard problem. For instance, we do not even know all combinatorial types of sections for regular cubes of dimension greater than five \cite{fukudaMutoGoto97-5cubeslices,Lawrence79}. 
For each slicing chamber of both our arrangements the combinatorial type of hyperplane sections is fixed.
Thus, as a second main contribution, we recover an upper bound on the number of different combinatorial types of sections by affine hyperplanes of the polytope $P$.

\begin{theorem} \label{thm:intro-combinatorialtypes}
For $d$-dimensional polytopes with $n$ vertices, an upper bound on the 
number of combinatorial types of hyperplane sections is $O(n^{2d+1} 2^{d})$. 
\end{theorem}

We stress that using \Cref{thm:intro_structure} we can compute not only a 
bound for the exact number of combinatorial types of slices, but we also get an algorithm to list them for specific polytopes.
Our third main contribution, which relies again on \Cref{thm:intro_structure}, is a family of algorithms to find optimal affine hyperplane sections, halfspace sections, and projections of polytopes. 

\begin{theorem}\label{thm:intro_algorithm}
Let $P \subset \R^d$ be a polytope and $f(\bx)$ a polynomial in $\mathbb{Q}[x_1,\dots,x_d]$. Denote by $f_k(P)$ the number of $k$-dimensional faces of $P$ and let $w_{k+1}$ be a weight function defined on all $(k+1)$-dimensional faces $F$ of $P$. Let $H$ be an affine hyperplane, $H^+_0$ denote a halfspace defined by a central hyperplane, and $\pi_H$ denote the projection of $P$ in the direction orthogonal to $H$. 
We give algorithms to find an optimal solution for the following problems:
\begin{enumerate}[label=\textup{(}\roman*\textup{)}] 
\item\label{mainthm:section_vol_int} (section of maximum volume/integral) \hspace{0.6em} $\displaystyle \max_{H \subset \R^d} \ \vol(P \cap H), \qquad \ \max_{H \subset \R^d} \ \int_{P \cap H} f(\bx) \dx$.
\vspace{0.5em}
\item (optimal number of $k$-dimensional faces) $ \displaystyle \max_{H \subset \R^d} \ f_k(P \cap H),$ \hspace{1.6em} $ \displaystyle\max_{H^+_0 \subset \R^d} \ f_k(P \cap H^+_0).$ \label{mainthm:sect_k_faces}
\vspace{0.3em}
\item\label{mainthm:weighted_k_faces} (optimal weighted $k$-dimensional faces) \hspace{0.4em} 
$ \displaystyle \max_{H \subset \R^d} \sum_{\substack{F \subset P \\ F\cap H \neq \emptyset}} w_{k+1}(F), \;\: \min_{H \subset \R^d} \sum_{\substack{F \subset P \\ F\cap H \neq \emptyset}} w_{k+1}(F)$.

\item (central halfspace of optimal integral)\hspace{1.3em} $\displaystyle \max_{H^+_0 \subset \R^d} \ \int_{P \cap H^+_0} \!f(\bx) \dx,  \,  \min_{H^+_0 \subset \R^d} \ \int_{P \cap H^+_0} \!f(\bx) \dx.$ \label{mainthm:halfspaces_int}

\item (projections of optimal integral) \hspace{3.5em} $\ \displaystyle \max_{H \subset \R^d}  \int_{\pi_H (P)} f(\bx) \dx, \quad \, \min_{H \subset \R^d}  \int_{\pi_H (P)} f(\bx) \dx.$  \label{mainthm:projection}
\end{enumerate}
If $P$ is a rational polytope and the dimension $d$ is fixed, then all these problems can be solved in polynomial time.
\end{theorem}

As the number of items in \Cref{thm:intro_algorithm} suggests, there are many possible applications of \Cref{thm:intro_structure}, both for combinatorial as well as metric criteria.
Maximal combinatorial slices of polytopes are of interest in the context of algebraic and topological 
combinatorics. For instance, a variation of the \emph{upper bound theorem}
for polytopes is to find an upper bound for $f$-vectors of slices of a polytope. Khovanskii investigated this problem and asked to compare the $h$-vector of a section of the polytope by a generic affine plane of  dimension $l$ and the $h$-vector of the original 
polytope \cite[Section 7]{Khovanskii-slices2006}.
The volume of special slices of the permutahedron fixed by a permutation was analyzed in \cite{Ardilaetal2021slicingpermutahedron}, and it turns out to agree with the slice depicted in \Cref{fig:permu-vol-min} for certain cases.

Optimal combinatorial halfspace sections in computational geometry have been studied for
the maximization of the number of vertices on the sphere \cite{Johnson+Preparata1978}. 
Moreover, there are interesting applications regarding optimization of volumes and integrals of hyperplane sections (e.g., moments \cite{Gravinetal2012-moments,KousholtSchulte2021moments}). It is very difficult to find the slices which have maximal or minimal (through a given point) volume, even for most basic polytopes. 
For instance, the affine hyperplane section of maximum volume has been identified for the $d$-dimensional cube \cite{Kball1} and the cross-polytope \cite{MeyerPajor88,Koldobskybook2005}; an analogous result for the simplex concerns hyperplanes through the centroid \cite{webb}. We hope our algorithms will add new information. 

Our fourth main contribution is experimental.
As indicated in \Cref{thm:intro_algorithm}, our algorithms work well for 
many optimality criteria, both concerning combinatorial and metric properties.
Applying the algorithm from \Cref{thm:intro_algorithm} we computed the optimal 
slices of famous polytopes in low dimensions, such as the Platonic solids, 
the permutahedron, and the cross-polytope.
This is an interesting result, since very little is known about optimal 
slices or combinatorial types of specific polytopes \cite{Kball1,Chakerian+Filliman86,Lawrence79,Ardilaetal2021slicingpermutahedron,nayar+tkocz-2022extremalsecproj}.

\noindent {\bf Overview:}
We begin in \Cref{sec:rotational_slices} by reviewing central hyperplane sections, 
and show that the integral of a polynomial over all such sections 
is a piecewise rational function. In \Cref{sec:translational_slices} we show 
the analogue for parallel sections in a fixed direction, and for orthogonal projections in \Cref{section:projections}. We merge these results in \Cref{sec:mergingslices} to arbitrary 
affine hyperplane sections. The proofs of \Cref{thm:intro_structure,thm:intro-combinatorialtypes} are the main content of \Cref{sec:mergingslices}. 

Algorithmic results are proved in \Cref{sec:complexity}. The proof of \Cref{thm:intro_algorithm} is in \Cref{sec:polytime_complexity}. While our algorithm runs in polynomial time for rational polytopes and polynomials in fixed dimension, the problems are hard in non-fixed dimension (see \Cref{subsec:hardness}). We close with our experimental results in \Cref{section:applications}. We provide the maximal volume slices for all Platonic solids, we investigate the $3$-dimensional permutahedron for different optimality criteria, and we present lists of all combinatorial types of slices of the cross-polytope of dimensions 4 and 5.
In this paper we use notions from polyhedral combinatorics and computational 
convex geometry (see \cite{GrunbaumBook67,ZieglerBook,grandpabibleI-1994,grandpabible2-1994,grandpabible3-1997}).

\section{Sections and Projections}

We begin our study of hyperplane sections and projections of a polytope. We analyze two families of hyperplane sections: rotational and translational ones. By rotational or central slices we mean hyperplanes that pass through a common point, which we assume to be the origin. On the other hand, translational or parallel slices are parallel affine hyperplane sections with a common normal vector. These two points of view mirror two standard constructions in convex geometry, namely intersection bodies and monotone path polytopes, respectively. 
In both settings, the combinatorial type of the hyperplane section $P\cap H$ is governed by a hyperplane arrangement $\mathcal{C}$, which is either central or parallel. In the open chambers of the hyperplane arrangements, one can then integrate polynomials over the slices parametrically, as rational functions. This is the main result of this section, stated in \Cref{thm:integral_polynomial_over_P,thm:parallel_sections}.
A similar situation arises when looking at projections. We describe the associated hyperplane arrangement, exploiting the duality with intersections, and prove an analogous result for the integral over the projections in \Cref{thm:integral_polynomial_projection}.
Notice that we state our theorems for full-dimensional polytopes. Analogous results can be stated and proved for lower-dimensional polytopes by considering them inside their affine span.

\subsection{Rotational slices}\label{sec:rotational_slices}

We discuss now rotational slices of a fixed polytope $P \subset \R^d$ which are obtained by hyperplanes passing through the origin with normal vector $\bu \in S^{d-1}$. We show that given a polynomial $f$, the integral of  $f$ over these sections is a piecewise rational function in variables $u_1,\dots,u_d$.
More specifically, the hyperplanes we consider in this section are of the form
\[
\{ \bx\in \R^d \mid \langle \bu, \bx \rangle = 0\}, \text{ where } \bu \in S^{d-1}.
\]
Understanding the volume of central hyperplane sections is a crucial step in the construction of the \emph{intersection body} of $P$ \cite{Lutwak:IntersectionBodies}.
Indeed, in this section we make use of results from \cite{BBMS:IntersectionBodiesPolytopes}, which studies intersection bodies of polytopes. A key argument is the existence of a central hyperplane arrangement, as follows.

\begin{lemma}[{\cite[Lemma 2.4]{BBMS:IntersectionBodiesPolytopes}}]\label{lemma:subdivision_of_sphere}
Let $P$ be a full-dimensional polytope in $\R^d$ and consider the central hyperplane arrangement 
\[
\chamrad(P) = \{ \bv^\perp \mid \bv \text{ is a vertex of $P$ and not the origin} \},
\]
where $\bv^\perp = \{\bx \in \R^d \mid \langle \bx, \bv \rangle = 0 \}$ denotes the central hyperplane with normal vector $\bv$.
The maximal open chambers $C$ of $\chamrad(P)$ satisfy the following property:
For all $\bx \in C$, the hyperplane $\bx^\perp$ intersects a fixed set of edges of $P$. In particular, the polytopes $Q = P \cap \bx^\perp$ are of the same combinatorial type for all $\bx \in C$. 
\end{lemma}

A (maximal, open) slicing chamber, or simply \emph{chamber}, of $\chamrad(P)$ is a connected component of $\R^d \setminus \chamrad(P)$.
In order to simplify the notation, we write $C\subset\chamrad(P)$ when $C$ is a maximal chamber of the hyperplane arrangement $\chamrad(P)$.
We illustrate the above statement on a $2$-dimensional example. This will serve as a running example which we develop throughout the article to illustrate the main concepts and constructions. 

\begin{example}\label{ex:chambers-radial}
    Consider the pentagon $P = \conv( (-1,-1), (1,-1), (1,1), (0,2), (-1,1) ) \subset \R^2$. Any generic hyperplane through the origin intersects $P$ in a pair of edges of $P$. There are $6$ different such pairs, and the normal vectors of all hyperplanes intersecting a fixed pair forms a maximal open chamber of the hyperplane arrangement 
    \[
        \chamrad(P) = (1,1)^\perp \cup (1,-1)^\perp \cup (0,2)^\perp,
    \]
    as depicted in \Cref{fig:chambers-radial}, left. This hyperplane arrangement only consists of three distinct hyperplanes, since $(-1,-1)^\perp = (1,1)^\perp$ and $(1,-1)^\perp = (-1,1)^\perp$. An arrangement with six distinct hyperplanes is obtained, e.g., for $P \!+\! \bt$ with $\bt = -(\tfrac{1}{3},\tfrac{1}{2})$, in \Cref{fig:chambers-radial}, right.
    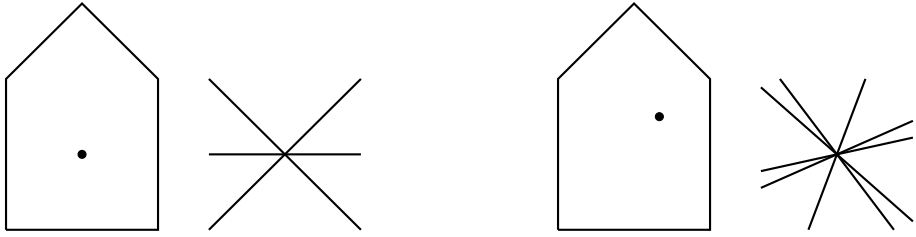
\begin{figure}[ht]
        \centering
        \begin{tikzpicture}
    \draw[thick] (-1,-1) -- (1,-1) -- (1,1) -- (0,2) -- (-1,1) -- (-1,-1)  ;
    \filldraw (0,0) circle (1.5 pt);
\end{tikzpicture}
\hspace*{1em}
\begin{tikzpicture}
    \draw[thick] (-1,-1) -- (1,1) ;
    \draw[thick] (1,-1) -- (-1,1) ;
    \draw[thick] (-1,0) -- (1,0) ;
\end{tikzpicture} 
\hspace*{6em}
\begin{tikzpicture}
    \draw[thick] (-1,-1) -- (1,-1) -- (1,1) -- (0,2) -- (-1,1) -- (-1,-1)  ;
    \filldraw (1/3,1/2) circle (1.5 pt);
\end{tikzpicture}
\hspace*{1em}
\begin{tikzpicture}
    \draw[thick] (1, -8/9) -- (-1, 8/9) ;
    \draw[thick] (3/8, 1) -- (-3/8, -1) ;
    \draw[thick] (1, 2/9) -- (-1, -2/9) ;
    \draw[thick] (1, 4/9) -- (-1, -4/9) ;
    \draw[thick] (-3/4, 1) -- (3/4, -1) ;
\end{tikzpicture}
        \caption{Left: The polytope $P$ and the central hyperplane arrangement $\chamrad(P)$ from \Cref{ex:chambers-radial}. Right: $P+\bt$ and $\chamrad(P+\bt)$ for $\bt = -(\tfrac{1}{3},\tfrac{1}{2})$.}
        \label{fig:chambers-radial}
    \end{figure}
\end{example}

We are interested in integrating polynomials over the hyperplane sections of the polytope. This is a generalization of a volume computation, namely the integral of the constant function $1$. For this purpose, we will need the following lemma which provides a recipe to efficiently integrate powers of linear forms over simplices via rational function formulas.

\begin{lemma}[{\cite[Theorem 2.1]{lasserre01_multidimensionalversion}},{\cite[Remark 9]{baldoni11_howintegratepolynomial}}]\label{lem:integrals_linear_power_exp}
Let $\Delta = \conv(\bs_1,\dots,\bs_{n}) \subset \R^d$ be a simplex, let~$\bp \in \R^d$ and $D \in \Z_{\geq 0}$. Then, writing $|\bk|=\sum_{j=1}^{n} k_j$, we have 
\begin{equation*}
\int_\Delta \langle \bp, \bx \rangle^D \dx= (n-1)! \vol(\Delta) \frac{D!}{(D+n-1)!}
\sum_{\substack{ \bk \in \Z_{\geq 0}^{n}, \\ | \bk| = D}}\langle \bp, \bs_1\rangle^{k_1}\cdots \langle \bp, \bs_{n}\rangle^{k_{n}}.
\end{equation*}
\end{lemma}

Additionally, we will make use of the following result about the 
decomposition of polynomials into sums of powers of linear forms.
\Cref{lem:monomial_to_linear_power} shows one way to express any polynomial 
of degree $D$ as a sum of $D$th powers of linear forms. As a consequence, if we know 
how to integrate powers of linear forms, we know how to integrate any polynomial.

\begin{lemma}[{\cite[Equation 13]{baldoni11_howintegratepolynomial}}]\label{lem:monomial_to_linear_power}
Any monomial can be written as a sum of powers of linear forms of the same degree as follows:
\[
x_1^{\alpha_1}x_2^{\alpha_2}\cdots x_d^{\alpha_d}=
\frac{1}{|\balpha|!} \sum_{\substack{\bp \in \Z^d_{\geq 0} \\ \bp \leq \balpha}}(-1)^{|\balpha|-|\bp|}
  \binom{\alpha_1}{p_1}\cdots \binom{\alpha_d}{p_d}(p_1 x_1+\cdots+p_d x_d)^{|\balpha|},
\]
where $|\balpha| = \alpha_1 + \dots + \alpha_d$ and $\bp \leq \balpha$ means that $p_i \leq \alpha_i$ for all coordinates $i \in [d]$.
\end{lemma}
We note that another formula for polynomial integration over simplices is described in \cite{Lasserre:Integration}. 
However, since the vertices of our simplices are parametrized by the normal vector of the central hyperplane, this formula is not suitable in our setting.

We now prove the main result of this section. 
The proof of this result is an adaption for integration of the proof of \cite[Theorem 2.6]{BBMS:IntersectionBodiesPolytopes}, which deals with volume computation.

\begin{theorem}\label{thm:integral_polynomial_over_P}
Let $P\subset \R^d$ be a full-dimensional polytope, let $f(\bx) = \sum_{\balpha} c_{\balpha} \bx^{\balpha}$ be a polynomial, and let $C \subset \R^d$ be a maximal open slicing chamber of the central hyperplane arrangement $\chamrad(P)$ from \Cref{lemma:subdivision_of_sphere}. Restricted to directions $\bu \in C \cap S^{d-1}$, the integral $\int_{P \cap \bu^\perp} f(\bx) \dx$ is a rational function in variables $u_1,\dots,u_d$.
\end{theorem}
\begin{proof}
    Let $Q(\bu) = P \cap \bu^\perp$ for some $\bu \in C\cap S^{d-1}$ and fix a triangulation $\mathcal{T}$ of $Q(\bu)$ without any additional vertices. By construction, the set of edges of $P$ which intersect $\bu^\perp$ are uniquely determined by $C$, and thus the triangulation of $P\cap \bu^\perp$ can be chosen for all $\bu \in C \cap S^{d-1}$.
    Let $\bv_1(\bu), \ldots, \bv_n(\bu)$ be the vertices of $Q(\bu)$ and let $\conv(\ba_i,\bb_i)$ be the corresponding edges of $P$ such that $\bv_i(\bu) \in \conv(\ba_i,\bb_i)$. 
    Given a simplex $\Delta\in \mathcal{T}$ with vertices $\bv_{j_1}(\bu),\dots,\bv_{j_d}(\bu)$ its volume can be computed as $\vol (\Delta) = \frac{1}{(d-1)!}|\det M_{\Delta}(\bu)|$ where 
    \[
    M_\Delta(\bu) = \begin{bmatrix}
    \bv_{j_2}(\bu) - \bv_{j_1}(\bu) \\
    \bv_{j_2}(\bu) - \bv_{j_1}(\bu) \\
    \vdots \\
    \bv_{j_{d}}(\bu) - \bv_{j_1}(\bu) \\
    \bu
    \end{bmatrix} ,
    \qquad
    \bv_{j}(\bu) = \frac{\langle \bb_{j}, \bu \rangle \ba_{j} - \langle \ba_{j}, \bu \rangle \bb_{j}}{\langle \bb_{j}-\ba_{j}, \bu \rangle},
    \]
    and for a fixed region $C$, the sign of the determinant of $M_\Delta(\bu)$ is constant for all $\bu \in C \cap S^{d-1}$.
    We now apply \Cref{lem:monomial_to_linear_power} to rewrite our polynomial $f$ as a combination of powers of linear forms, in order to then apply \Cref{lem:integrals_linear_power_exp}. For any $\bx$, we have that
    \[
    f(\bx) = \sum_{\balpha} \frac{c_{\balpha}}{|\balpha|!} \sum_{\substack{\bp \in \Z^d_{\geq 0} \\ \bp \leq \balpha}}(-1)^{|\balpha|-|\bp|}
    \binom{\alpha_1}{p_1}\cdots \binom{\alpha_d}{p_d}(p_1 x_1+\cdots+p_d x_d)^{|\balpha|}.
    \]
    Hence, the integral of $f$ over $Q(\bu)$ can be computed exactly:
    \small
    \begin{align}\label{eq:long-computation}
        \int_{Q(\bu)} & f(\bx) \dx = \sum_{\balpha} \frac{c_{\balpha}}{|\balpha|!} \sum_{\substack{p \in \Z^d_{\geq 0} \\ \bp \leq \balpha}}(-1)^{|\balpha|-|\bp|}
        \binom{\alpha_1}{p_1}\cdots \binom{\alpha_d}{p_d} \int_{Q(\bu)} \langle \bp, \bx \rangle^{|\balpha|} \dx\\ \nonumber
        &= \sum_{\balpha} \frac{c_{\balpha}}{|\balpha|!} \sum_{{\substack{\bp \in \Z^d_{\geq 0} \\ \bp \leq \balpha}}}(-1)^{|\balpha|-|\bp|}
        \binom{\alpha_1}{p_1}\cdots \binom{\alpha_d}{p_d} \sum_{\Delta \in \mathcal{T}} \int_{\Delta} \langle \bp, \bx \rangle^{|\balpha|} \dx\\ \nonumber
        &= \sum_{\balpha} \frac{c_{\balpha}}{|\balpha|!} \sum_{{\substack{\bp \in \Z^d_{\geq 0} \\ \bp \leq \balpha}}}(-1)^{|\balpha|-|\bp|}
        \binom{\alpha_1}{p_1}\cdots \binom{\alpha_d}{p_d} \sum_{\Delta \in \mathcal{T}} 
        \Bigg[ (d-1)! \vol(\Delta) \frac{|\balpha|!}{(|\balpha|+d-1)!} \\ \nonumber
        &  \qquad \qquad \qquad  \qquad \qquad \qquad \qquad \qquad \qquad \qquad
        \sum_{\substack{ \bk \in \Z_{\geq 0}^{d}, \\ | \bk| = |\balpha|}}\langle \bp, \bv_{j_1}(\bu)\rangle^{k_1}\cdots \langle \bp, \bv_{j_d}(\bu)\rangle^{k_{d}} \Bigg] \\ \nonumber
        &= \sum_{\Delta \in \mathcal{T}} \left| \det M_{\Delta}(\bu) \right| \sum_{\balpha} \frac{c_{\balpha}}{(|\balpha|+d-1)!} \sum_{{\substack{\bp \in \Z^d_{\geq 0} \\ \bp \leq \balpha}}}(-1)^{|\balpha|-|\bp|}
        \binom{\alpha_1}{p_1}\cdots \binom{\alpha_d}{p_d}  
        \\ \nonumber
        &  \qquad \qquad \qquad  \qquad \qquad \qquad \qquad \qquad \qquad \qquad
        \sum_{ \substack{ \bk \in \Z_{\geq 0}^{d}, \\ | \bk| = |\balpha|}}\langle \bp, \bv_{j_1}(\bu)\rangle^{k_1}\cdots \langle \bp, \bv_{j_d}(\bu)\rangle^{k_{d}},
        \nonumber
    \end{align}
    which is a rational function in variables $u_1,\dots,u_d$.
\end{proof}

\begin{remark}\label{rmk:closures}
For all hyperplane arrangements we encounter in this and the following sections we define chambers and regions as connected components of the complement of the arrangement, and thus the chambers and regions are open and full-dimensional by definition. All statements regarding the integration of a polynomial are stated solely for these open full-dimensional polyhedra. 
However, since the rational functions do not have poles at the boundary of the regions and chambers, we can extend these statements to the closures of these polyhedra, which yields rational functions on the entire induced polyhedral complex. Thus all statements also hold on these lower-dimensional faces: The integral is a rational function, and these functions are specializations of the rational functions which are defined on the maximal polyhedra containing such a face.
\end{remark}

\begin{example}\label{ex:integration-radial}
Continuing \Cref{ex:chambers-radial}, we illustrate the statement of \Cref{thm:integral_polynomial_over_P}. We compute the volume $\int_{P\cap \bu^\perp} 1 \dx$, the sum of the first moments $\int_{P\cap \bu^\perp} x_1 + x_2 \dx$ and the sum of the second moments $\int_{P\cap \bu^\perp} x_1^2 + x_1 x_2 + x_2 \dx$ over all central sections of the pentagon $P$ from \Cref{ex:chambers-radial}. As shown in \Cref{fig:integration-radial}, each of the integrals is a rational function in $u_1,u_2$ along each of the chambers of the central hyperplane arrangement $\chamrad(P)$. One can check that the functions in two adjacent chambers agree along the common face, as explained in \Cref{rmk:closures}. We note that the fact that the volume is a factor of the latter two integrals is an artefact of low dimension, where every section is a $1$-dimensional simplex, and so no further triangulation is needed.
\end{example}
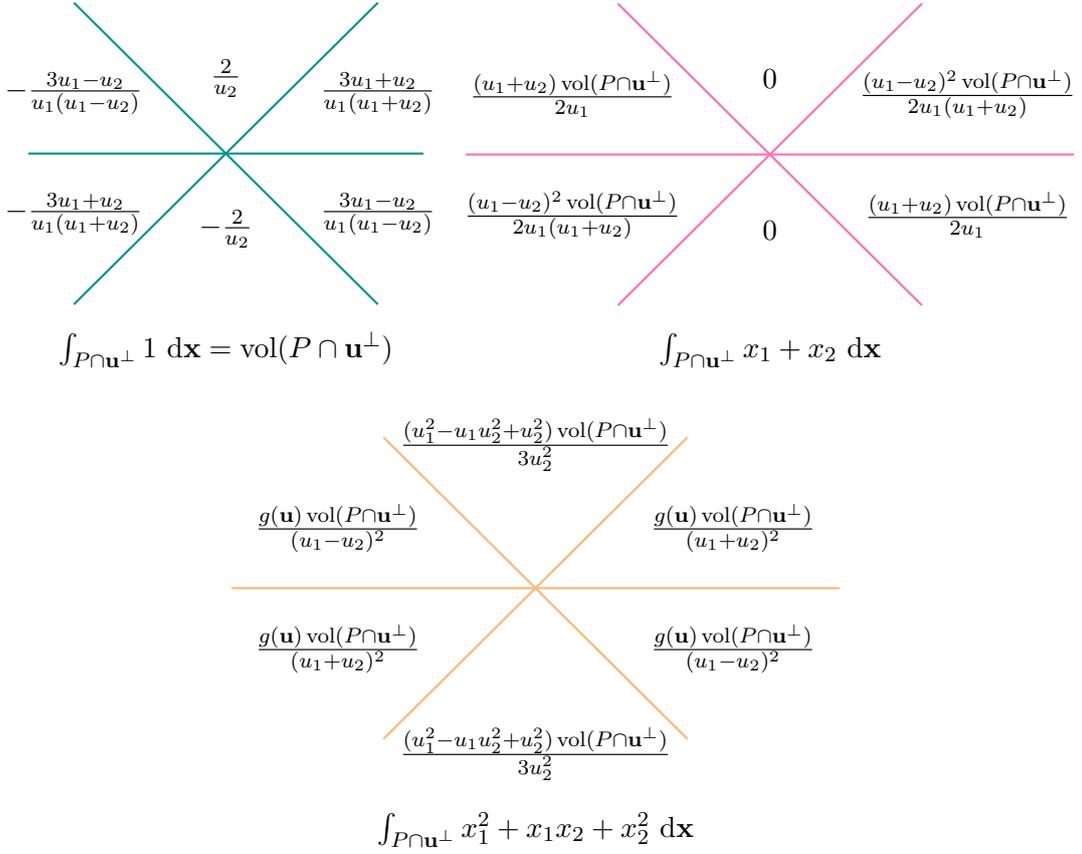
\begin{figure}[ht]
    \centering
    \begin{tikzpicture}[scale = 2]
    \draw[thick, cb-green-sea] (-1,-1) -- (1,1) ;
    \draw[thick, cb-green-sea] (1,-1) -- (-1,1) ;
    \draw[thick, cb-green-sea] (-1.3,0) -- (1.3,0) ;
    \node at (0,.5) {$\tfrac{2}{u_2}$};
    \node at (0,-.5) {$-\tfrac{2}{u_2}$};
    \node at (-1,.4) {$-\tfrac{3u_1 - u_2}{u_1(u_1 - u_2)}$};
    \node at (-1,-.4) {$-\tfrac{3u_1 + u_2}{u_1(u_1 + u_2)}$};
    \node at (1,-.4) {$\tfrac{3u_1 - u_2}{u_1(u_1 - u_2)}$};
    \node at (1,.4) {$\tfrac{3u_1 + u_2}{u_1(u_1 + u_2)}$};
    \node at (0,-1.3) {$\int_{P \cap \bu^\perp} 1 \dx = \vol(P\cap \bu^\perp)$};
\end{tikzpicture}
\begin{tikzpicture}[scale = 2]
    \draw[thick, cb-rose] (-1,-1) -- (1,1) ;
    \draw[thick, cb-rose] (1,-1) -- (-1,1) ;
    \draw[thick, cb-rose] (-2,0) -- (2,0) ;
    \node at (0,.5) {$0$};
    \node at (0,-.5) {$0$};
    \node at (-1.3,.4) {$\tfrac{(u_1 + u_2)\vol(P\cap \bu^\perp)}{2u_1} $};
    \node at (-1.3,-.4) {$\tfrac{(u_1 - u_2)^2 \vol(P\cap \bu^\perp)}{2u_1(u_1 + u_2)}$};
    \node at (1.3,-.4) {$\tfrac{(u_1 + u_2) \vol(P\cap \bu^\perp)}{2u_1}$};
    \node at (1.3,.4) {$\tfrac{(u_1 - u_2)^2  \vol(P\cap \bu^\perp) }{2u_1(u_1 + u_2)}$};
    \node at (0,-1.3) {$\int_{P \cap \bu^\perp} x_1 + x_2 \dx$};
\end{tikzpicture}\\
\vspace*{1em}
\begin{tikzpicture}[scale = 2]
    \draw[thick, cb-salmon-pink] (-1,-1) -- (1,1) ;
    \draw[thick, cb-salmon-pink] (1,-1) -- (-1,1) ;
    \draw[thick, cb-salmon-pink] (-2,0) -- (2,0) ;
    \node at (0,0.95) {$\tfrac{(u_1^2 - u_1 u_2^2 + u_2^2)\vol(P\cap \bu^\perp)}{3 u_2^2}$};
    \node at (0,-1.1) {$\tfrac{(u_1^2 - u_1 u_2^2 + u_2^2)\vol(P\cap \bu^\perp)}{3 u_2^2}$};
    \node at (-1.3,.4) {$\tfrac{g(\bu) \vol(P\cap \bu^\perp)}{(u_1 - u_2)^2} $};
    \node at (-1.3,-.4) {$\tfrac{g(\bu) \vol(P\cap \bu^\perp)}{(u_1 + u_2)^2} $};
    \node at (1.3,-.4) {$\tfrac{g(\bu) \vol(P\cap \bu^\perp)}{(u_1 - u_2)^2} $};
    \node at (1.3,.4) {$\tfrac{g(\bu) \vol(P\cap \bu^\perp)}{(u_1 + u_2)^2} $};
    \node at (0,-1.6) {$\int_{P \cap \bu^\perp} x_1^2 + x_1 x_2 + x_2^2 \dx$};
\end{tikzpicture}
    \caption{The integrals of the sum of moments over central sections of the pentagon, as described in \Cref{ex:integration-radial}, with
    $g(\bu) = \tfrac{(3u_1^2 + u_2^2)(u_1^2 -u_1u_2 + u_2^2)}{3 u_1^2}$.}
    \label{fig:integration-radial}
\end{figure}

\subsection{Translational slices}\label{sec:translational_slices}

In this section we analyze translational slices of a polytope $P \subset \R^d$, which are obtained by translates of a hyperplane with a fixed normal vector $\bu \in S^{d-1}$. Similarly to \Cref{sec:rotational_slices}, we show that for a fixed polynomial $f$ the integral over these affine sections is a univariate polynomial.
Concretely, fix a vector $\bu\in S^{d-1}$ and consider the family of affine hyperplanes
\[
H(\beta) = \{ \bx\in \R^d \mid \langle \bu, \bx \rangle = \beta \}
\]
orthogonal to $\bu$, parametrized by $\beta \in \R$. Parallel, or translational, slices arise naturally in the context of \emph{monotone path polytopes}, which are special instances of fiber polytopes \cite{BilStu:FiberPolytopes}. Also in this case, there is a hyperplane arrangement, consisting of parallel hyperplanes, which governs the combinatorial structure of the translational slices.

\begin{lemma}\label{lemma:hyparr_translation}
    Let $P \subset \R^d$ be a polytope and fix a direction $\bu \in S^{d-1}$. Consider the affine hyperplane arrangement, made of parallel hyperplanes
    \[
        \chamtra = \{ H(\langle \bu, \bv \rangle ) \mid  \bv \text{ is a vertex of } P \}.
    \]
    The maximal open chambers $C$ of $\chamtra(P)$ satisfy the following property: For all $H(\beta) \in C$, the hyperplanes $H(\beta)$ intersects a fixed set of edges of $P$. Moreover, the polytopes $Q(\beta) = P \cap H(\beta)$ are normally equivalent, i.e., they have the same normal fan.
\end{lemma}

\begin{proof}
    Let $C$ be a fixed chamber of $\chamtra$ and $H(\beta) \in C$. The hyperplane $H(\beta)$ intersects an edge $\conv(\bv_1,\bv_2)$ in its interior if and only if $\langle \bu, \bv_1 \rangle < \beta < \langle \bu, \bv_2 \rangle$. Thus, the set of edges intersected by $H(\beta)$ is fixed in each chamber $C$. Consequently, the set of intersected faces of arbitrary dimension is fixed along the open chamber $C$. Thus, the combinatorial type of $Q(\beta)$ is fixed, and so is the combinatorial type of its normal fan. Note that any facet $F$ of $P\cap H(\beta)$ arises as intersection of a facet $G$ of $P$ with $H(\beta)$, and the normal vector of $F$ is a projection of the normal vector of $G$ onto $H(\beta)$. Since projections are invariant under affine translations of $H(\beta)$, all affine sections within $C$ are normally equivalent. 
\end{proof}

\begin{remark}
    We note that the hyperplane arrangement $\chamtra(P)$ induces a partition of $\R^1$:
    \[
        \{ \langle \bu, \bv \rangle \mid \bv \text{ is a vertex of } P \}.
    \]
    Indeed, this partition and $\chamtra(P)$ are equivalent, as $\beta$ can be uniquely determined from $H(\beta)$, when the normal vector $\bu$ is fixed, and vice versa. In the remaining of this article, we allow ourselves to write $\beta \in C$ instead of $H(\beta) \in C$ for a chamber $C$ of the parallel arrangement $\chamtra(P)$, as we have chosen the convention in \Cref{lemma:hyparr_translation} purely for esthetic reasons. As already pointed out in the rotational case, we will write $C\subset \chamtra(P)$ for the maximal chambers.
\end{remark}

\begin{remark}\label{rmk:orderings}
    For generic $\bu \in S^{d-1}$, the linear functional $\langle \bu, \cdot \rangle$ induces an ordering $\bv_1,\dots,\bv_n$ on the vertices of $P$ such that $\langle \bu, \bv_i \rangle < \langle \bu, \bv_{i+1} \rangle$ for all $i = 1,\dots,n-1$. By construction, a chamber of $\chamtra(P)$ consists precisely of those parallel hyperplanes which are orthogonal to $\bu$ and separate $\bv_i$ from $\bv_{i+1}$ for some $i \in [n-1]$.
\end{remark}

\begin{example}\label{ex:parralel-chambers}
   Recall the pentagon $P$ from \Cref{ex:chambers-radial} with vertices
    \[
        \bv_1 = ( -1,-1 ), \  \bv_2 = (1,-1), \ \bv_3 = (1,1), \ \bv_4 = (0,2), \ \bv_5 =  (-1,1).
    \]
    For a generic direction $\bu$ the hyperplane arrangement $\chamtra(P)$ induces six slicing chambers, four of which have nonempty intersection with $P$. \Cref{fig:parralel-chambers} shows the arrangement for a non-generic (left) and a generic (right) choice of $\bu$, namely for $\bu = \tfrac{1}{2} (-1,-1)$ and $\bu = \tfrac{1}{\sqrt5}(1,2)$.
    The parallel hyperplane arrangement for the pentagon $P$ and $\bu = \tfrac{1}{\sqrt5}(1,2)$ induces the ordering of the vertices $\bv_1, \bv_2, \bv_5, \bv_3, \bv_4$.
\end{example}
\begin{figure}[h]
    \centering
    \begin{tikzpicture}[scale = 1]
    \draw[thick, color=black!20] (-1,-1) -- (1,-1) -- (1,1) -- (0,2) -- (-1,1) -- (-1,-1)  ;
    \draw[thick] (-0.6, -1.4) -- (-1.5, -0.5) ;
    \draw[thick] (1.4, -1.4) -- (-1.5, 1.5) ;
    \draw[thick] (1.5, 0.5) -- (-0.2, 2.2) ;
    \draw[->,thick] (0.5,-.5) -- (0.2,-.8);
    \node at (0.0,-0.65) {$\bu$};
\end{tikzpicture}
\hspace*{5em}
\begin{tikzpicture}[scale = 1]
    \draw[thick, color=black!20] (-1,-1) -- (1,-1) -- (1,1) -- (0,2) -- (-1,1) -- (-1,-1)  ;
    \draw[thick] (-0.2, -1.4) -- (-1.5, -0.75) ;
    \draw[thick] (1.5, -1.25) -- (-1.5, 0.25) ;
    \draw[thick] (1.5, -0.25) -- (-1.5, 1.25) ;
    \draw[thick] (1.5, 0.75) -- (-1.4, 2.2) ;
    \draw[thick] (1.5, 1.25) -- (-0.4, 2.2) ;
    \draw[->,thick] (0.6,-0.8) -- (0.8,-0.4);
    \node at (0.99,-0.65) {$\bu$};
\end{tikzpicture}
    \caption{The parallel arrangement $\chamtra$ for normal directions $\bu = \tfrac{1}{2} (-1,-1)$ on the left and  $\bu = \tfrac{1}{\sqrt5}(1,2)$ on the right.}
    \label{fig:parralel-chambers}
\end{figure}
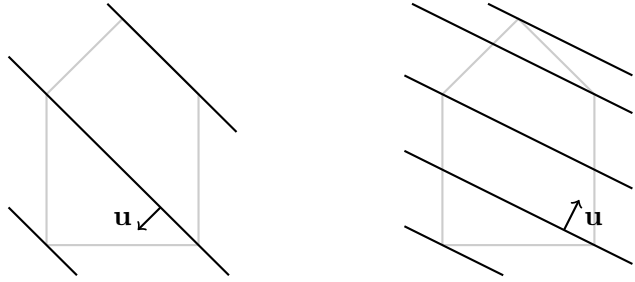

We now prove the analogue of \Cref{thm:integral_polynomial_over_P} for parallel affine sections.

\begin{theorem}\label{thm:parallel_sections}
    Let $P \subset \R^d$ be a full-polytope, let $f(\bx) = \sum_{\balpha} c_{\balpha} \bx^{\balpha} $ be a polynomial, fix a normal direction $\bu \in S^{d-1}$ and let $C \subset \R^d$ be a maximal open chamber of the hyperplane arrangement \ $\chamtra(P)$ from \Cref{lemma:hyparr_translation}. Restricted to values $\beta \in C$, the integral $\int_{P\cap H(\beta)} f(\bx) \dx$ is a polynomial in the variable $\beta$.
\end{theorem}
\begin{proof}
    This proof has the same structure as the proof of \Cref{thm:integral_polynomial_over_P}. The main difference lies in the parametrization of the vertices of the sections of the polytope.
    
    Let $\beta \in C$ and $Q(\beta) = H(\beta) \cap P$. Let $\bv_1(\beta),\dots,\bv_n(\beta)$ denote the vertices of $Q(\beta)$ and let $\conv(\ba_i,\bb_i)$ be the edge of $P$ such that $\bv_i(\beta) = \conv(\ba_i,\bb_i)\cap H(\beta)$. From
    \begin{equation*}
        \bv_i(\beta) = \lambda \ba_i + (1-\lambda )\bb_i, \qquad
        \langle \bu, \bv_i(\beta) \rangle = \beta,
    \end{equation*}
    we obtain
    \begin{equation*}\label{eq:vertices_translation}
        \bv_i(\beta) = \frac{\beta}{\langle \bu, \bb_i-\ba_i\rangle} (\bb_i - \ba_i) + \frac{\langle \bu, \bb_i \rangle \ba_i - \langle \bu, \ba_i \rangle \bb_i}{\langle \bu, \bb_i - \ba_i\rangle},
    \end{equation*}
    which depends linearly on $\beta$. Let $\mathcal{T}$ be a triangulation of $Q(\beta)$ which uses only the vertices of $Q(\beta)$. By construction of $C$ this triangulation can be chosen equal for every $\beta \in C$. Let $
    \bv_{j_1}(\beta),\dots,\bv_{j_d}(\beta)$ denote the vertices of a simplex $\Delta \in \mathcal T$, and let 
    \[
    M_\Delta(\beta) = \begin{bmatrix}
    \bv_{j_2}(\beta) - \bv_{j_1}(\beta) \\
    \bv_{j_2}(\beta) - \bv_{j_1}(\beta) \\
    \vdots \\
    \bv_{j_{d}}(\beta) - \bv_{j_1}(\beta) \\
    \bu
    \end{bmatrix}.
    \]
    Repeating the computation of \eqref{eq:long-computation} in the proof of \Cref{thm:integral_polynomial_over_P}, $\int_{P \cap H(\beta)} f(\bx) \dx$ equals
    {\small
    \begin{equation*}
        \sum_{\Delta \in \mathcal{T}} \left| \det M_{\Delta}(\beta) \right| \sum_{\balpha} \frac{c_{\balpha}}{(|\balpha|+d-1)!} \!\!\sum_{{\substack{\bp \in \Z^d_{\geq 0} \\ \bp \leq \balpha}}}\!\!(-1)^{|\balpha|-|\bp|}
        \binom{\alpha_1}{p_1}\cdots \binom{\alpha_d}{p_d}
        \!\!\sum_{ \substack{ \bk \in \Z_{\geq 0}^{d}, \\ | \bk| = |\balpha|}}\!\!\langle \bp, \bv_{j_1}(\beta)\rangle^{k_1}\cdots \langle \bp, \bv_{j_d}(\beta)\rangle^{k_{d}} ,
    \end{equation*}}
    which is a polynomial in $\beta$, for $\beta \in C$. 
\end{proof}

\begin{example}\label{ex:chambers-translational}
    We continue \Cref{ex:chambers-radial,ex:parralel-chambers} by computing integrals of sums of moments over parallel sections of the pentagon $P$, with respect to the normal direction $\bu = \tfrac{1}{\sqrt5}(1,2)$.
    The polynomials describing the function $\int_{P\cap H(\beta)} f(\bx) \dx$ for $f(\bx) = 1$, $f(\bx) = x_1 + x_2$, and $f(\bx) = x_1^2 + x_1 x_2 + x_2^2$ respectively, for each chamber in the arrangement, are shown in \Cref{fig:chambers-translational}.
    \begin{figure}[ht]
        \centering
        \begin{tikzpicture}[scale = 1.3]
    \draw[thick, color=black!20] (-1,-1) -- (1,-1) -- (1,1) -- (0,2) -- (-1,1) -- (-1,-1)  ;
    \draw[thick, cb-green-sea] (-0.2, -1.4) -- (-1.5, -0.75) ;
    \draw[thick, cb-green-sea] (1.5, -1.25) -- (-1.5, 0.25) ;
    \draw[thick, cb-green-sea] (1.5, -0.25) -- (-1.5, 1.25) ;
    \draw[thick, cb-green-sea] (1.5, 0.75) -- (-1.4, 2.2) ;
    \draw[thick, cb-green-sea] (1.5, 1.25) -- (-0.4, 2.2) ;
    \node at (-0.4,-0.7) {$\tfrac{5\beta + 3\sqrt5}{2}$};
    \node at (0,0) {$\sqrt5$};
    \node at (0,1) {$\tfrac{-5\beta + 7\sqrt5}{6}$};
    \node at (1.1,1.9) {$\tfrac{-10\beta + 8\sqrt5}{3}$};
    \draw [thick,black] (0.5,1.9) to [out=180,in=40] (0,1.7);
    \node at (0,-1.7) {$\int_{P\cap H(\beta)} 1 \dx = \vol(P\cap H(\beta))$};
\end{tikzpicture}
\begin{tikzpicture}[scale = 1.3]
    \draw[thick, color=black!20] (-1,-1) -- (1,-1) -- (1,1) -- (0,2) -- (-1,1) -- (-1,-1)  ;
    \draw[thick, cb-rose] (-0.2, -1.4) -- (-1.5, -0.75) ;
    \draw[thick, cb-rose] (1.5, -1.25) -- (-1.5, 0.25) ;
    \draw[thick, cb-rose] (1.5, -0.25) -- (-1.5, 1.25) ;
    \draw[thick, cb-rose] (1.5, 0.75) -- (-1.4, 2.2) ;
    \draw[thick, cb-rose] (1.5, 1.25) -- (-0.4, 2.2) ;
    \node at (-0.4,-0.7) {\rotatebox{-27}{$\tfrac{(3\sqrt 5\beta + 1)\vol(P\cap H(\beta))}{4}$}};
    \node at (0,0) {\rotatebox{-27}{$\tfrac{5\beta}{2}$}};
    \node at (0,1) {\rotatebox{-27}{$\quad \tfrac{(7\sqrt5\beta - 1)\vol(P\cap H(\beta))}{12}$}};
    \node at (1.1,1.9) {\rotatebox{-27}{$\tfrac{(\sqrt5\beta + 2)\vol(P\cap H(\beta))}{3}$
    }};
    \draw [thick,black] (0.5,1.9) to [out=180,in=40] (0,1.7);
    \node at (0,-1.7) {$\int_{P\cap H(\beta)} x_1 + x_2 \dx$};
\end{tikzpicture}
\begin{tikzpicture}[scale = 1.3]
    \draw[thick, color=black!20] (-1,-1) -- (1,-1) -- (1,1) -- (0,2) -- (-1,1) -- (-1,-1)  ;
    \draw[thick, cb-salmon-pink] (-0.2, -1.4) -- (-1.5, -0.75) ;
    \draw[thick, cb-salmon-pink] (1.5, -1.25) -- (-1.5, 0.25) ;
    \draw[thick, cb-salmon-pink] (1.5, -0.25) -- (-1.5, 1.25) ;
    \draw[thick, cb-salmon-pink] (1.5, 0.75) -- (-1.4, 2.2) ;
    \draw[thick, cb-salmon-pink] (1.5, 1.25) -- (-0.4, 2.2) ;
    \node at (-0.4,-0.7) {\rotatebox{-27}{$\tfrac{(10\beta^2 + 3\sqrt5 \beta + 3)\vol(P\cap H(\beta))}{4}$}};
    \node at (0,0) {\rotatebox{-27}{$\tfrac{\sqrt5(5\beta^2 +1)}{4}$}};
    \node at (0,1) {\rotatebox{-27}{$\quad \tfrac{(50\beta^2 - 5\sqrt5\beta + 13)\vol(P\cap H(\beta))}{36}$}};
    \node at (1.1,1.9) {\rotatebox{-27}{$\tfrac{2(10\beta^2 - 7\sqrt5 \beta + 14)\vol(P\cap H(\beta))}{9}$
    }};
    \draw [thick,black] (0.5,1.9) to [out=180,in=40] (0,1.7);
    \node at (0,-1.7) {$\int_{P\cap H(\beta)} x_1^2 +x_1x_2 + x_2^2 \dx$};
\end{tikzpicture}
        \caption{The integrals over the sum of moments over affine sections of the pentagon in direction $\tfrac{1}{\sqrt5}(1,2)$, as described in \Cref{ex:chambers-translational}.}
        \label{fig:chambers-translational}
    \end{figure}
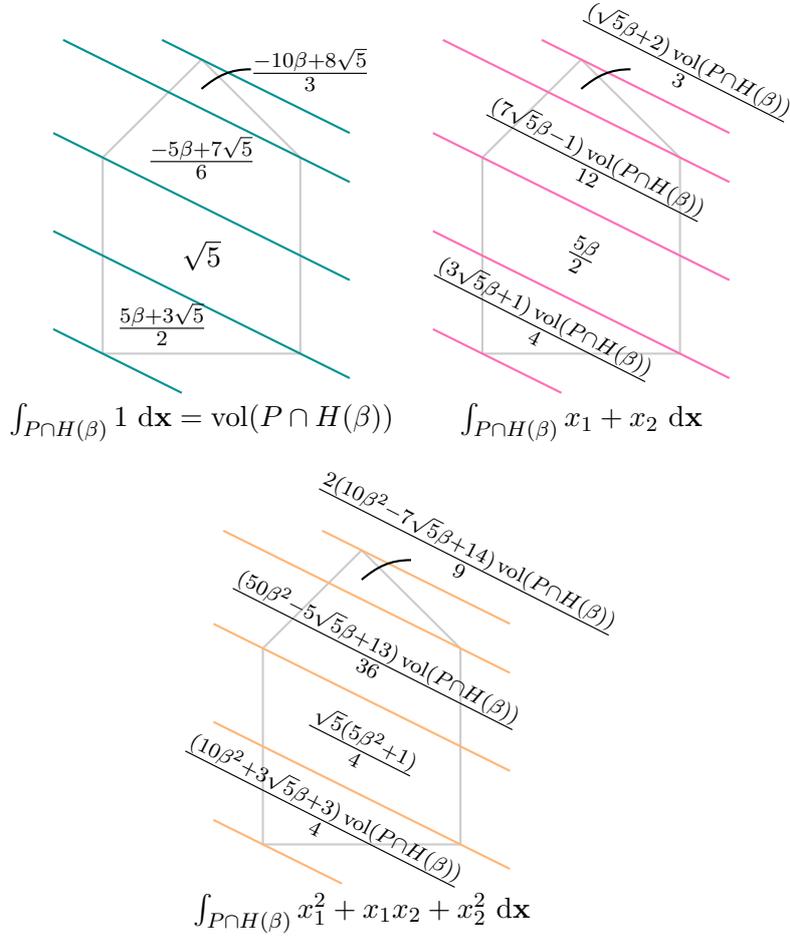
\end{example}

The above discussion has a discrete version, where instead of volumes or integrals we count lattice points in the section. There are several discrete variants of well-known continuous inequalities in convex geometry and Brunn-Minkowski theory \cite{Freyer+Henk2022}. While we are unable to do optimization over discrete sections of polytopes, we obtain the following partial result. We give give a sketch of the proof in \Cref{rmk:ehrhart}.

\begin{theorem} \label{thm:ehrhart}
    The Ehrhart function $Ehr(P \cap H(\beta))$ that counts lattice points of dilations of polytopes is a piecewise rational function in $\beta$. In fixed dimension and for a rational polytope $P$, these formulas can be computed in polynomial time.
\end{theorem}

\begin{remark}\label{rmk:ehrhart}
\Cref{thm:ehrhart} is a direct consequence of the theory of rational functions encoding the lattice points of polyhedra \cite{barvipom}. Each lattice point is thought of as the exponent vector of a 
monomial, turning the set of all lattice points in $P$ into a Laurent polynomial $g_P(\bx)=\sum_{\balpha \in P\cap\Z^d} \bx^{\balpha}.$ 
This monomial sum can be written as a sum of rational 
functions
\[
g_P(\bx) = \sum_{i\in I} {E_i \frac{\bx^{\bu_i}} {\prod\limits_{j=1}^d
(1-\bx^{\bv_{ij}})}},
\]
where $I$ is an indexing set, $E_i\in\{1,-1\}$, and $\bu_i, \bv_{ij} \in\Z^d$ for all $i$ and $j$. The formula above coincides for polytopes with the same normal fan, which, by \Cref{lem:sweep-arrangement}, is the case within the chambers of $\chamtra(P)$. Moreover, assuming the dimension $d$ is fixed, the size of the sum is polynomial in the input size.
\end{remark}

\begin{remark}
    In the setting of rotational slices, computing the number of lattice points in the section $P\cap H$ is a much more difficult problem than the translational case. This lies in the fact that the polytopes $P\cap \bu_1^\perp$ and $P\cap \bu_2^\perp$ are \emph{not} normally equivalent, even for $\bu_i$ in the same chamber.
\end{remark}

\subsection{Projections and polarity}\label{section:projections}

The dual version of intersections is given by projections. It is therefore natural to wonder if any of the results of the previous sections apply in this context. They actually do apply, and involve hyperplane arrangements that we have already encountered. If central, rotating sections are connected to the construction of intersection bodies, and parallel, translating sections to monotone path polytopes, here we should keep in mind the concept of \emph{projection bodies} \cite[Chapter 4]{Gardnerbook}. This is another construction coming from convex geometry and it encodes in its support function the volume of all $(d-1)$-dimensional projections 
of a given convex body in $\R^d$. In our setting, we denote by $P^\circ$ the polar of $P$ and we identify $\R^d$ with its dual space, via the standard scalar product.

\begin{lemma}\label{lemma:hyparr_projection}
    Let $P \subset \R^d$ be a polytope containing the origin in its interior and let $P^\circ$ be the polar of $P$. Consider the affine hyperplane arrangement $\chamrad(P^\circ)$. The maximal open chambers $C$ of $\chamrad(P^\circ)$ satisfy the following property: For all $\bu \in C$, the projection $\pi_\bu(P)$ has as vertices the projections of a fixed set of vertices of $P$. In particular, the polytopes $\pi_\bu(P)$ are combinatorially equivalent.
\end{lemma}
\begin{proof}
    Let $C$ be a maximal open chamber of $\chamrad(P^\circ)$, and let $\bu \in C$. By polarity of projections and intersections, we have that 
    \[
    \pi_\bu(P) = \left( P^\circ \cap \bu^\perp\right)^\circ.
    \]
    Since the combinatorial type of $P^\circ \cap \bu^\perp$ does not change when $\bu \in C$, the combinatorial type of $\pi_\bu(P)$ does not change as well.
    Moreover, the vertices of $\pi_\bu(P)$ are the projection of those vertices of $P$, whose corresponding facet $F$ of $P^\circ$ defines a facet $F \cap \bu^\perp$
    of $P^\circ \cap \bu^\perp$.
\end{proof}

Applying the same strategy as in the proofs of \Cref{thm:integral_polynomial_over_P,thm:parallel_sections}, we obtain an analogous result about the integral of a polynomial for projections of polytopes, which constitutes the first step towards the proof of \Cref{thm:intro_algorithm} \ref{mainthm:projection}.

\begin{theorem}\label{thm:integral_polynomial_projection}
Let $P\subset \R^d$ be a full-dimensional polytope, let $f(\bx) = \sum_{\balpha} c_{\balpha} x^{\balpha}$ be a polynomial, and let $C \subset \R^d$ be a maximal open chamber of the hyperplane arrangement $\chamrad(P^\circ)$ in \Cref{lemma:hyparr_projection}. Restricted to directions $\bu \in C \cap S^{d-1}$, the integral $\int_{\pi_{\bu}(P)} f(\bx) \dx$ is a polynomial in variables $u_1,\dots,u_d$.
\end{theorem}
\begin{proof}
    Also in this case, the proof is a straightforward consequence of the proof of \Cref{thm:integral_polynomial_over_P}, after determining the parametrization of the vertices of the projection.
    
    Let $Q(\bu) = \pi_\bu(P)$ for some $\bu \in C$. By construction, the vertices of $P$ whose projections are vertices of $Q(\bu)$ are uniquely determined by $C$, and thus the triangulation $\mathcal{T}$ of $\pi_\bu(P)$ can be chosen for all $\bu \in C \cap S^{d-1}$.
    Let $\pi_\bu(\bv_1), \ldots, \pi_\bu(\bv_n)$ be the vertices of $Q(\bu)$ and let
    \[
    M_\Delta(\bu) = \begin{bmatrix}
    \pi_\bu(\bv_{j_2}) - \pi_\bu(\bv_{j_1}) \\
    \pi_\bu(\bv_{j_3}) - \pi_\bu(\bv_{j_1}) \\
    \vdots \\
    \pi_\bu(\bv_{j_d}) - \pi_\bu(\bv_{j_1}) \\
    \bu
    \end{bmatrix} ,
    \qquad
    \pi_\bu(\bv_j) = \bv_j - \langle \bu, \bv_j \rangle \bu,
    \]
    where $\pi_\bu(\bv_{j_1}), \ldots, \pi_\bu(\bv_{j_d})$ are the vertices of $\Delta\in \mathcal{T}$.
    Repeating the computation of \eqref{eq:long-computation}, the integral of $f$ over $Q(\bu)$ can be computed exactly as
    {\small
    \begin{equation*}
        \sum_{\Delta \in \mathcal{T}} \!\! \left| \det M_{\Delta}(\bu) \right| \sum_{\balpha} \frac{c_{\balpha}}{(|\balpha|+d-1)!} \!\sum_{{\substack{\bp \in \Z^d_{\geq 0} \\ \bp \leq \balpha}}}\!\!(-1)^{|\balpha|-|\bp|}
        \binom{\alpha_1}{p_1}\cdots \binom{\alpha_d}{p_d}
        \!\sum_{ \substack{ \bk \in \Z_{\geq 0}^{d}, \\ | \bk| = |\balpha|}}\!\!\langle \bp, \pi_\bu(\bv_{j_1})\rangle^{k_1}\cdots \langle \bp, \pi_\bu(\bv_{j_d})\rangle^{k_{d}}
    \end{equation*}}
    which is a polynomial in $u_1,\dots,u_d$, for $\bu\in C \cap S^{d-1}$.
\end{proof}
\begin{remark}
    If $D$ is the degree of $f$, then
    the degree of the polynomial $\int_{Q(\bu)} f(\bx) \dx$ is at most $2d(D + 1)-1$.
   When $f$ is a constant then our function is the support function of a zonotope, the projection body of $P$ \cite[Section 10.9]{Schneider:Bible}, and must therefore be linear in $\bu$. 
\end{remark}

\begin{example}\label{ex:chambers-projection}
    We continue \Cref{ex:chambers-radial,ex:chambers-translational} and compute the volume of all projections of the pentagon $P$.
    The polynomials describing $\vol(\pi_\bu(P))$, for $\bu\in S^{1}$, in each chamber of the arrangement $\chamrad(P^\circ)$ are shown in \Cref{fig:house_projection}.
\end{example}

There is also another possible approach, based on \cite[pp. 260-261]{lawrence91} and \cite[Theorem 1]{Filliman92}, which is however more complicated. It involves again polarity, but of the hyperplane sections of $P$. This alternative approach is also based on the fact that  $\pi_\bu(P) = \left( P^\circ \cap \bu^\perp\right)^\circ$,
and there are formulas to compute the volume of the polar of a polytope from the polytope itself. Lawrence formula is stated only for simple polytopes, but by Brion's theorem and signed cone decompositions methods \cite{Barvinok:CourseConvexity} we can obtain a more general formula which holds for any polytope. The second formula holds for non-codegenerate polytopes, so in order to apply it one should verify this property for all projections $\pi_\bu(P)$.
Because of these subtle conditions of genericity, we prefer to use here the above approach of triangulating directly the projection, which already implies a good bound on the complexity of associated optimization problems, as we will discuss at the end of \Cref{sec:polytime_complexity} when proving the complexity part of \Cref{thm:intro_algorithm} \ref{mainthm:projection}.

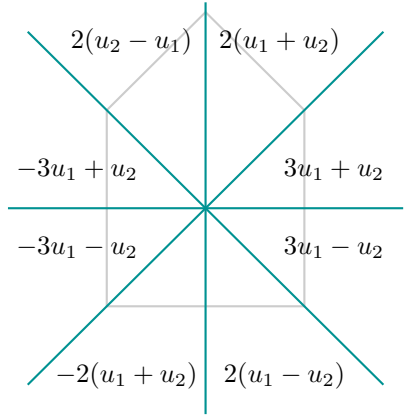
\begin{figure}[h]
    \centering
    \begin{tikzpicture}[scale = 1.3]
    \draw[thick, color=black!20] (-1,-1) -- (1,-1) -- (1,1) -- (0,2) -- (-1,1) -- (-1,-1)  ;
    \draw[thick, cb-green-sea] (0, 2.1) -- (0, -2.1) ;
    \draw[thick, cb-green-sea] (-1.8, 1.8) -- (1.8, -1.8) ;
    \draw[thick, cb-green-sea] (-2, 0) -- (2, 0) ;
    \draw[thick, cb-green-sea] (-1.8, -1.8) -- (1.8, 1.8) ;
    \node at (0.75,1.7) {\small $2(u_1+u_2)$};
    \node at (-0.8,-1.7) {\small $-2(u_1+u_2)$};
    \node at (1.3,0.4) {\small $3 u_1+u_2$};
    \node at (-1.3,-0.4) {\small $-3 u_1-u_2$};
    \node at (-0.75,1.7) {\small $2(u_2-u_1)$};
    \node at (0.8,-1.7) {\small $2(u_1-u_2)$};
    \node at (-1.3,0.4) {\small $-3 u_1+u_2$};
    \node at (1.3,-0.4) {\small $3 u_1-u_2$};
\end{tikzpicture}
    \caption{The volume of all the projections $\pi_\bu(P)$, for $\bu\in S^1$, of the pentagon $P$, as described in \Cref{ex:chambers-projection}.}
    \label{fig:house_projection}
\end{figure}

\section{Merging the two types of slices}\label{sec:mergingslices}

In \Cref{sec:rotational_slices,sec:translational_slices} we parametrically computed the integral of a polynomial $f$ over central and parallel hyperplane sections of a polytope $P$. In this section, we generalize the results to arbitrary affine sections, based on the following observation. There are two natural ways to consider all hyperplane sections. One can first choose a point in $\R^d$ as a center and then examine all (rotational) hyperplanes through that point; or one can first fix a direction and examine all the affine (parallel) hyperplanes orthogonal to it. In \Cref{sec:translating-the-rotation}, we take the first point of view, generalizing the approach of central sections from \Cref{sec:rotational_slices}. This yields an affine hyperplane arrangement $\regrad(P)$ in the space of translation vectors of $P$. On the other hand, in the spirit of the second point of view on hyperplane sections, generalizing the approach of parallel sections from \Cref{sec:translational_slices} yields a central hyperplane arrangement $\regtra(P)$ in the space of normal vectors of 
affine hyperplanes, as we discuss in \Cref{sec:rotating-the-translation}.

Another standard procedure to turn central hyperplane sections into affine ones uses homogenization. Namely, embed the $d$-dimensional polytope $P$ in $\R^{d+1}$ inside the hyperplane $x_{d+1}=1$. Then all $(d-1)$-dimensional slices of $P$ can be obtained by hyperplanes through the origin in $\R^{d+1}$. 
However, since in this setting the polytope $P$ is not full dimensional, the parametric computation of the volume of the section would involve finding a parametric orthonormal basis. This is not suitable for computations, where we prefer to stick to rational data. For these reasons, we do not discuss this approach in more detail.

\subsection{Translating the rotation}\label{sec:translating-the-rotation}

In \Cref{sec:rotational_slices} we fixed a polytope $P$ and considered hyperplanes through the origin with normal vectors $\bu \in S^{d-1}$, yielding a central hyperplane arrangement $\chamrad(P)$, in which each maximal open chamber consists of normal vectors $\bu$ such that the central hyperplane $\bu^\perp$ intersects a fixed set of edges of $P$. In this section we extend this construction, allowing to vary the position of the origin by considering translations $P+\bt$.

\begin{lemma}[{\cite[Lemma 3.2, Proposition 3.4]{BraMer:IntersectionBodies}}]\label{th:cocircuit-arrangement}
    Let $P \subset \R^d$ be a full-dimensional polytope and let $R$ be a maximal region of the affine hyperplane arrangement
    \[
        \regrad(P) = \{ \operatorname{aff}(-\bv_1, \ldots, -\bv_d) \,|\, \bv_1, \ldots, \bv_d \text{ are affinely independent vertices of } P\},
    \]
    called the cocircuit arrangement.
    Then, the following holds:
    \begin{enumerate}[label=\textup{(}\roman*\textup{)}]
        \item For all $\bt \in R$, the central hyperplane arrangements $\chamrad(P+\bt)$ define the same realizable oriented matroid $\chi$. Moreover, the hyperplanes defining a chamber $C(\bt)$ of $\chamrad(P+\bt)$ are parametrized linearly by $t_1,\dots,t_d$.
        
        \item Let $\bt,\bt' \in R$ and let $C(\bt) \subset \chamrad(P+\bt), C(\bt') \subset \chamrad(P+\bt')$ be maximal chambers such that the topes (cocircuits) of $\chi$ corresponding to $C(\bt), C(\bt')$ agree. Then, 
        \[
            \{e \text{ edge of P} \mid (e+\bt)\cap \bu^\perp \neq \emptyset  \} = \{e \text{ edge of P} \mid (e+\bt')\cap (\bu')^\perp \neq \emptyset  \}
        \]
        for any $\bu \in C(\bt)$, $\bu' \in C(\bt')$. In words, $\bu^\perp$ and $(\bu')^\perp$ intersect the same set of edges (of $P+\bt$ and $P+\bt'$, respectively).
    \end{enumerate}
\end{lemma}

We refer to $\regrad(P)$ as the \emph{cocircuit arrangement} of $P$, and a \emph{region} $R$ is an open connected component of $\R^d \setminus \regrad(P)$. As before, in order to simplify the notation, we will write $R\subset~\!\!\regrad(P)$ for a region of the arrangement. 
Notice that we are dealing at the same time with two distinct hyperplane 
arrangements: $\regrad(P)$ and $\chamrad(P)$. The names of the arrangements 
correspond to the names we use for their complement: the connected components of $\R^d \setminus\chamrad(P)$, as defined in \Cref{sec:rotational_slices}, are called \emph{chambers}, whereas the connected components of $\R^d \setminus\regrad(P)$ are called \emph{regions}. 
Putting the two arrangements together, we can parametrize all hyperplanes in $\R^d$. 
A point $\bt$ in a region of $\regrad(P)$ fixes a translation of $P$ or, analogously, 
it fixes the position of the origin with respect to the polytope. Then, 
the chambers of $\chamrad(P+\bt)$ parametrize the hyperplanes through the chosen origin. 
When we change $\bt$, we capture new hyperplane sections.

\begin{example}\label{ex:regions-radial}
    We continue the \Cref{ex:chambers-radial,ex:integration-radial}.
    The cocircuit arrangement $\regrad(P)$ consists of $\binom{5}{2} = 10$ hyperplanes, subdividing $\R^2$ into $11$ bounded and $26$ unbounded regions, as shown in \Cref{fig:regions-radial-cocircuit}.
    Let 
    \[
    R = \{\bt \in \R^2 \mid 
    \langle \sma -3 \\ 1 \strix, \bt \rangle  > -2, \ \langle \sma-1 \\ -1\strix ,\bt \rangle > 0, \ \langle \sma 3 \\ 1 \strix, \bt \rangle  > -2, \
     \langle \sma 1 \\  -1 \strix, \bt \rangle > 0, \  \langle \sma 0 \\ 1 \strix, \bt \rangle  > -1 \}
    \]
    be the pentagonal shaded region of $\regrad(P)$. Note that $R$ contains the translation 
    vector $\bt = -(\frac{1}{3},\frac{1}{2})$ from \Cref{ex:chambers-radial}.
    For all $\bt \in R$ the central arrangement $\chamrad(P+\bt)$ defines the same oriented matroid. In particular, all central hyperplane arrangements have the same combinatorial structure, and are parametric in $t_1,t_2$, as shown in \Cref{fig:regions-radial-central}. Varying $\bt$ induces a rotation of the hyperplanes defined by $\langle \bv_i + \bt , \bx \rangle = 0$, where $\bv_i$ are the vertices of $P$. 
    Note that the (strict) inequalities defining $R$ guarantee that under these rotations no two hyperplanes in $\chamrad(P+\bt)$ collapse.
    \begin{figure}
        \centering
        \begin{subfigure}[t]{0.45\textwidth}
            \centering
            \begin{tikzpicture}
    \fill[black!20] (1/3, -1) -- (1/2, -1/2) -- (0, 0) -- (-1/2, -1/2) -- (-1/3, -1) ;
    \draw[thick] (1, 3/2) -- (1, -5/2) ;
    \draw[thick] (7/6, 3/2) -- (-1/6, -5/2) ;
    \draw[thick] (2, 1) -- (-2, 1) ;
    \draw[thick] (3/2, 3/2) -- (-2, -2) ;
    \draw[thick] (2, 0) -- (-1/2, -5/2) ;
    \draw[thick] (2, -2) -- (-3/2, 3/2) ;
    \draw[thick] (2, -1) -- (-2, -1) ;
    \draw[thick] (-7/6, 3/2) -- (1/6, -5/2) ;
    \draw[thick] (1/2, -5/2) -- (-2, 0) ;
    \draw[thick] (-1, 3/2) -- (-1, -5/2) ;
\end{tikzpicture}
            \caption{The cocircuit arrangement $\regrad(P)$. \\ The region $R$ is shaded in gray.}
            \label{fig:regions-radial-cocircuit}
        \end{subfigure}
        \begin{subfigure}[t]{0.45\textwidth}
            \centering
            \begin{tikzpicture}[scale = 1.5]
    \draw[thick] (1, -8/9) -- (-1, 8/9);
    \node[anchor=west] at (1,-8/9) {$\langle \bv_1 + \bt, \bx \rangle = 0$};
    \draw[thick] (3/8, 1) -- (-3/8, -1) ;
    \node[anchor=south] at (3/8,1) {$\langle \bv_5 + \bt, \bx \rangle = 0$};
    \draw[thick] (1, 2/9) -- (-1, -2/9) ;
    \node[anchor=west] at (1,2/9) {$\langle \bv_4 + \bt, \bx \rangle = 0$};
    \draw[thick] (1, 4/9) -- (-1, -4/9) ;
    \node[anchor=west] at (1,5/9) {$\langle \bv_2 + \bt, \bx \rangle = 0$};
    \draw[thick] (-3/4, 1) -- (3/4, -1) ;
    \node[anchor=north] at (3/4,-1) {$\langle \bv_3 + \bt, \bx \rangle = 0$};
    \node at (-0.6,0.15) {$C(\bt)$};
\end{tikzpicture}
            \caption{The (parametric) central \\ arrangement $\chamrad(P+ \bt)$ for $\bt \in R$.}
            \label{fig:regions-radial-central}
        \end{subfigure}
        \caption{The hyperplane arrangements from \Cref{ex:regions-radial}.}
    \end{figure}
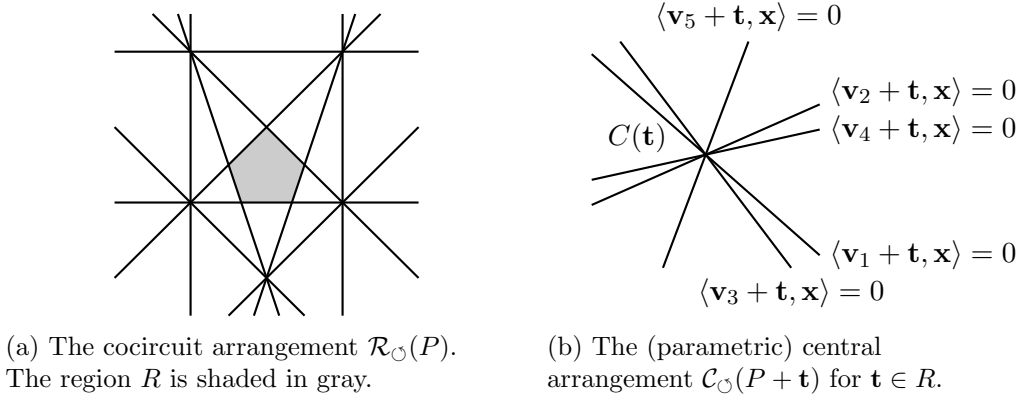
\end{example}

We have used hyperplane arrangements to identify the different slicing chambers. We now want to count them, using classical properties of arrangements of hyperplanes in $d$-dimensional Euclidean space. For the theory of enumeration of faces and cells in a hyperplane arrangement we refer the reader to \cite{Stanley:IntroHyperplaneArr,Zaslavsky}.
As we will see in \Cref{sec:complexity}, this allows us to answer any purely combinatorial question regarding the slices of a polytope, in polynomial time.

\begin{prop}\label{prop:countingchamberscocircuit_intersectionbody}
    Let $P\subset \R^d$ be a polytope with $n$ vertices. In the cocircuit arrangement $\regrad(P)$ there are at most $\binom{n}{d}$ affine hyperplanes. There are at most $O(n^{d^2})$ $d$-dimensional polyhedral regions and the total number of chambers (including lower-dimensional cells)
    in the associated cocircuit arrangement is bounded by $O(n^{d^2} 2^d)$. Thus, the number of slicing 
    chambers (counting also cells of lower dimension for both regions and chambers) is $O(n^{d^2+d} 2^{2d})$.
\end{prop}   

\begin{proof}
We denote by $f_k$ the number of polyhedral faces of dimension $k$. By the classic Zaslavsky's theorem (see \cite{Zaslavsky} or \cite[Proposition 2.4]{Stanley:IntroHyperplaneArr}), the number of top-dimensional polyhedral regions $f_d(\regrad(P))$ of the arrangement is 
\[
f_d(\regrad(P)) \leq \sum_{i=1}^d \binom{\binom{n}{d}}{i} \sim O(n^{d^2}).
\]
Note that one can write precise formulas, when the associated matroid is known (see \cite{Stanley:IntroHyperplaneArr}), but the generic case provides an upper bound.
Moreover, it is well-known that one can bound the number of $k$-faces in the 
arrangement by the inequality (see \cite{KFukudaetal-hyperpcount1991}):
$$f_k(\regrad(P)) \leq \binom{d}{k} f_d(\regrad(P)), \quad 0 \leq k \leq d,
$$
which implies that the desired bound on all polyhedral regions (top dimensional or not) is
$O(n^{d^2}2^d)$. For each such region we have a corresponding hyperplane arrangement $\chamrad(P+\bt)$ to consider, which has $n$ hyperplanes and thus at most 
$\sum_{i=1}^d \binom{n}{i}$
full-dimensional slicing chambers. Therefore, for each choice of region, we get a number of slicing chambers of order $O(n^{d})$ or, counting also its lower-dimensional faces, of order $O(n^d2^d)$. This number multiplied by the upper bound for the number of cells of $\regrad(P)$, gives a bound of order $O(n^{d^2+d} 2^{2d})$.
\end{proof}

These pairs consisting of cocircuit arrangement and the corresponding central arrangement of vertices (after a choice of center) detect all the changes in the combinatorial structure of the hyperplane sections of $P$. In fixed dimension $d$, the number of regions and of faces of these regions is polynomial in the number of vertices of $P$, and on each of these cells the hyperplane section $P\cap H$ has a given combinatorial type.

We can extend \Cref{thm:integral_polynomial_over_P} to take into account translations when integrating a polynomial over hyperplane sections of $P$. This is an extension of \cite[Theorem 3.5]{BraMer:IntersectionBodies}, which contains an analogous statement for the volume.

\begin{theorem}\label{thm:piecewise_rat_translate_rotation}
Let $P\subset \R^d$ be a full-dimensional polytope and let $f(\bx) = \sum_{\balpha} c_{\balpha} x^{\balpha}$ be a polynomial. Let $R \subset \R^d$ be a region of the cocircuit arrangement $\regrad(P)$ and let $C(\bt)\subset \R^d$ be a chamber of the central arrangement $\chamrad(P+\bt)$, for $\bt \in R$. Restricted to $\bt \in R$ and $\bu\in C(\bt)\cap S^{d-1}$, the integral $\int_{(P+\bt) \cap \bu^\perp} f(\bx) \dx$ is a rational function in variables $t_1,\dots,t_d$, $u_1,\dots,u_d$.
\end{theorem}
\begin{proof}
    By \Cref{th:cocircuit-arrangement}, the chambers $C(\bt)$ of $\chamrad(P+\bt)$ are linearly dependent on $\bt$, when $\bt \in R$. Let $\bu\in C(\bt)$. As in the proof of \Cref{thm:integral_polynomial_over_P}, we triangulate $(P+\bt) \cap \bu^\perp$, obtaining
    \[
    M_\Delta(\bt,\bu) = \begin{bmatrix}
    \bv_{j_2}(\bt,\bu) - \bv_{j_1}(\bt,\bu) \\
    \bv_{j_3}(\bt,\bu) - \bv_{j_1}(\bt,\bu) \\
    \vdots \\
    \bv_{j_{d}}(\bt,\bu) - \bv_{j_1}(\bt,\bu) \\
    \bu
    \end{bmatrix},\]
    \begin{align*}
    \bv_{j}(\bt,\bu) &= \frac{\langle \bb_{j} + \bt, \bu \rangle (\ba_{j} + \bt) - \langle \ba_{j} + \bt, \bu \rangle (\bb_{j} + \bt)}{\langle \bb_{j}-\ba_{j}, \bu \rangle}\\
    &= \frac{\langle \bb_{j} +\bt , \bu \rangle \ba_{j} - \langle \ba_{j} + \bt , \bu \rangle \bb_{j}}{\langle \bb_{j}-\ba_{j}, \bu \rangle}+\bt, 
    \end{align*}
    where $\bv_{j}(\bt,\bu)$ denotes a vertex of a simplex $\Delta$ in the triangulation $\mathcal T$, which lies on the interior of the edge $\conv(\ba_j + \bt, \bb_j + \bt)$ of $P + \bt$.
    Repeating the computation of \eqref{eq:long-computation} in the proof of \Cref{thm:integral_polynomial_over_P} yields
    {\small
    \begin{align*}
        \int_{(P+\bt) \cap \bu^\perp} f(\bx) \dx  
        = \sum_{\Delta \in \mathcal{T}} \left| \det M_{\Delta}(\bt,\bu) \right| 
        \sum_{\balpha} &\frac{c_{\balpha}}{(|\balpha|+d-1)!} \sum_{{\substack{\bp \in \Z^d_{\geq 0} \\ \bp \leq \balpha}}}(-1)^{|\balpha|-|\bp|}
        \binom{\alpha_1}{p_1}\cdots \binom{\alpha_d}{p_d} 
        \cdot
         \\
        & \qquad\quad \cdot \sum_{ \substack{ \bk \in \Z_{\geq 0}^{d}, \\ | \bk| = |\balpha|}}\langle \ p, \ \bv_{j_1}(\bt,\bu) \ \rangle^{k_1}\cdots \langle \ p, \ \bv_{j_d}(\bt,\bu) \ \rangle^{k_{d}},
    \end{align*}}
    which is a rational function in $t_1,\dots,t_d$, $u_1,\dots,u_d$.
\end{proof}
This result implies that $\int_{(P+\bt) \cap \bu^\perp} f(\bx) \dx$ is a piecewise rational function. The domains over which it is rational are pairs of polyhedra in the following sense. Each region $R$ is a polyhedron in $\R^d$, and each chamber $C(\bt)$ is a polyhedron in $\R^d$ parametrized by $\bt$. In particular, the domain of rationality is a semialgebraic subset of $\R^{2d}$ with the following shape:
\[
\{ (\bt,\bu)\in \R^{2d} \,|\, \bt \in R, \bu \in C(\bt)\}\cap \left(\R^d \times S^{d-1}\right), 
\]
for some region $R\subset \regrad(P)$ and chamber $C(\bt)\subset \chamrad(P+\bt)$, where $\bt \in R$. Notice that the condition $\bu \in C(\bt)$ is quadratic in the variables $t_1,\dots,t_d$, $u_1,\dots,u_d$, see \Cref{ex:regions-radial}. We can bound the degree of the polynomials appearing in the integration formula, independently of the chamber and region they come from, as follows.

\begin{prop}\label{prop:deg_rational_pieces}
    In each of slicing chambers the rational function $\int_{(P+\bt) \cap \bu^\perp} f(\bx) \dx$ respects the following degree bounds:
    \begin{equation*}
        \begin{aligned}
        \deg\,  \operatorname{numerator} \left( \int_{(P+\bt) \cap \bu^\perp} f(\bx) \dx \right) &\leq \Big(f_1(P) - (d-1)\Big) ( D+d-1 ) +d(D+1), \\
        \deg\,  \operatorname{denominator} \left( \int_{(P+\bt) \cap \bu^\perp} f(\bx) \dx \right) &\leq \Big(f_1(P) - (d-1)\Big) ( D+d-1 ).
        \end{aligned}
    \end{equation*}
    Here $f_k(P)$ denotes the number of $k$-dimensional faces of $P$ and $D = \deg f$.
\end{prop}

\begin{proof}
    We want to study the degree of numerator and denominator of the piecewise rational function $\int_{(P+\bt) \cap \bu^\perp} f(\bx) \dx$. We fix a region $R\subset \regrad(P)$ and a chamber $C(\bt)\subset\chamrad(P+\bt)$, for $\bt \in R$. 
    By the proof above, the integral of $f$ over the hyperplane section $(P+\bt) \cap \bu^\perp$ is the sum over all simplices $\Delta \in \mathcal{T}$ of $\left| \det M_{\Delta}(\bt,\bu) \right|$ multiplied by
    \begin{equation}\label{eq:interior_sum_degree}
    \small
        \sum_{\balpha} \frac{c_{\balpha}}{(|\balpha|+d-1)!}
    \sum_{{\substack{\bp \in \Z^d_{\geq 0} \\ \bp \leq \balpha}}}(-1)^{|\balpha|-|\bp|}
        \binom{\alpha_1}{p_1}\cdots \binom{\alpha_d}{p_d}
        \sum_{ \substack{ \bk \in \Z_{\geq 0}^{d}, \\ | \bk| = |\balpha|}}\langle \ \bp, \ \bv_{j_1}(\bt,\bu) \ \rangle^{k_1}\cdots \langle \ \bp, \ \bv_{j_d}(\bt,\bu) \ \rangle^{k_{d}},
    \end{equation}
    where $\bv_{j_1}(\bt,\bu),\ldots,\bv_{j_d}(\bt,\bu)$ are the vertices of the simplex $\Delta$. Since the denominator of $\bv_{j_i}(\bt,\bu)$ is the same in every coordinate of the vector, we can pull it out of the scalar product with $\bp$; moreover, denoting $D = \deg f$, we have that $k_i\leq D$ for all $i=1,\dots,d$. Therefore, \eqref{eq:interior_sum_degree} becomes
    \begin{equation*}
        \frac{\varphi_\Delta(\bt,\bu)}{\prod_i \langle \bu, \bb_{j_i}-\ba_{j_i}\rangle^{D}}
    \end{equation*}
    where $\bv_{j_i}$ belongs to the edge of $P$ with extrema $\ba_{j_i},\bb_{j_i}$, and $\varphi_\Delta$ is a polynomial.

    It is not hard to see that if we sum a bunch of rational functions such that the degree of their numerator minus the degree of their denominator is constant for all summands, then also the difference of the degrees of numerator and denominator of the sum of these functions is going to be that same number. Using this elementary fact, and noticing that for every summand in \eqref{eq:interior_sum_degree} the degree of the numerator is twice the degree of the denominator, we can deduce that $\deg\varphi_\Delta(\bt,\bu) = 2 \deg \left(\prod_i \langle \bu, \bb_{j_i}-\ba_{j_i}\rangle^{D}\right) =  2 \cdot d\cdot D$.

    Let us now describe the rational function $\left| \det M_{\Delta}(\bt,\bu) \right|$. All but the last row of the matrix have entries that are quotients of a cubic and a quadratic polynomial, in $t_1,\ldots,t_d,u_1,\ldots,u_d$. The denominator of each entry of the $(i-1)$th row is $ \langle \bu, \bb_{j_i}-\ba_{j_i}\rangle \langle \bu, \bb_{j_1}-\ba_{j_1}\rangle$ for all $i=2,\ldots,d$. 
    The determinant is thus the quotient of a polynomial of degree $3(d-1)+1$ and another polynomial of degree $2(d-1)$.
    Then,
    \[
    \left| \det M_{\Delta}(\bt,\bu) \right| \frac{\varphi_\Delta(\bt,\bu)}{\prod_i \langle \bu, \bb_{j_i}-\ba_{j_i}\rangle^{D}} = \frac{\phi_\Delta(\bt,\bu)}{\langle \bu, \bb_{j_1}-\ba_{j_1}\rangle^{D+d-1}\prod_{i\neq 1} \langle \bu, \bb_{j_i}-\ba_{j_i}\rangle^{D+1}},
    \]
    for some other polynomial $\phi_\Delta$ of degree $2 d D + 3(d-1)+1$. Notice that the difference between the degrees of numerator and denominator is $d(D+1)$, independently of the simplex $\Delta$.

    We are now ready to sum over the simplices in the triangulation. Since the summands have similar denominators, with some possible redundancy, we obtain the following expression for $\int_{(P+\bt)\cap \bu^\perp} f(\bx) \dx$:
    \begin{equation}\label{eq:sum_degree}
        \sum_{\Delta\in\mathcal{T}} \frac{\phi_\Delta(\bt,\bu)}{\langle \bu, \bb_{j_1}-\ba_{j_1}\rangle^{D+d-1}\prod_{i\neq 1} \langle \bu, \bb_j-\ba_j\rangle^{D+1}} = \frac{\psi(\bt,\bu)}{\prod_{j} \langle \bu, \bb_j-\ba_j\rangle^{D+d-1}},
    \end{equation}
    where the product now runs over all vertices of $(P+\bt)\cap \bu^\perp$. Since for all the summands in the left-hand side of \eqref{eq:sum_degree} we know that the difference between the degrees of numerator and denominator is $d$, we deduce that $\deg \psi (\bt,\bu) = (f_0((P+\bt)\cap \bu^\perp))(D+d-1) + d(D+1)$. Therefore, using \cite[Theorem 5.6]{BBMS:IntersectionBodiesPolytopes} to bound $f_0((P+\bt)\cap \bu^\perp)$, we obtain the degree bounds claimed in the statement, valid in every chamber of every region.
\end{proof}

\begin{example}\label{ex:radial-integral-parametric}
    We continue \Cref{ex:regions-radial}. Let $R$ be the pentagonal shaded region of the cocircuit arrangement $\regrad(P)$ in \Cref{fig:regions-radial-cocircuit} and let 
    \[
    C(\bt) = \{ \bx \in \R^2 \mid \langle \bv_1 + \bt, \bx \rangle > 0, \langle \bv_4 + \bt, \bx \rangle > 0  \}.
    \]
    The integral over the constant function $1$ for $\bt \in R$ and $\bu \in C(\bt)$ is
    \[
        \vol((P+\bt)\cap \bu^\perp) = \int_{(P+\bt)\cap \bu^\perp} 1 \dx = -\frac{t_1 u_1 + t_2 u_2 + 3 u_1 - u_2}{u_1(u_1 - u_2)}.
    \]
    Note that for $\bt = (0,0)$ this specializes to the polynomial in the corresponding chamber in \Cref{fig:integration-radial} from \Cref{ex:integration-radial}, as explained in \Cref{rmk:closures}.
    The bounds from \Cref{prop:deg_rational_pieces} in this case are $6$ for the degree of the numerator and $4$ for the degree of the denominator, both of which are not tight.
\end{example}

\subsection{Rotating the translation}\label{sec:rotating-the-translation}

In this section we extend the theory developed in \Cref{sec:translational_slices}, where we fixed a polytope $P$ and a unit direction $\bu \in S^{d-1}$, and studied the sections of $P$ by parallel affine hyperplanes $H(\beta)$ with normal vector $\bu$. Recall from \Cref{rmk:orderings} that $\bu$ induces an ordering on the vertices of $P$. We obtained a parallel hyperplane arrangement $\chamtra(P)$ where, for each chamber, the hyperplanes $H(\beta)$ intersect $P$ in a fixed set of edges, and separate two vertices which are consecutive in the induced ordering. 
Now, we allow ourselves to first vary the direction $\bu \in S^{d-1}$, and then construct $\chamtra(P)$, depending on $\bu$. 

\begin{lemma}\label{lem:sweep-arrangement}
    Let $P\subset \R^d$ be a full-dimensional polytope and let $R$ be a region of the central hyperplane arrangement
    \[
        \regtra(P) = \{ (\bv_i - \bv_j)^\perp \mid \bv_i, \bv_j \text{ are distinct vertices of $P$}\}.
    \]
    Then the following holds:
    For all $\bu \in R$ the linear functional $\langle \bu, \cdot \rangle$ induces the same ordering $\bv_1,\dots,\bv_n$ of the vertices of $P$.
\end{lemma}

The central hyperplane arrangement $\regtra(P)$ is called \emph{sweep arrangement} or 
\emph{lineup arrangement} \cite{PadrolPhilippe2021}. The above lemma implies the following: Let $\bu,\bu' \in R, i \in [n-1]$ 
and let $C(\bu) \subset \chamtra(P)$, $C(\bu') \subset \chamtraprime(P)$ such that all hyperplanes in $C(\bu)$ and $C(\bu')$ separate $\bv_i$ from $\bv_{i+1}$. Then, any two hyperplanes $H(\bu,\beta) \in C(\bu)$, $H(\bu',\beta') \in C(\bu')$ intersect $P$ in the same set of edges. However, the polytopes $P \cap H(\bu,\beta)$ and $H(\bu',\beta')$ are not normally equivalent.

\begin{example}\label{ex:sweep-arrangement}
    We continue \Cref{ex:parralel-chambers,ex:chambers-translational}.
    The sweep arrangement of $P$ is shown in \Cref{fig:sweep-arrangement}. 
    It consists of $6$ distinct hyperplanes, subdividing $\R^2$ into $12$ regions. 
    Each of these regions corresponds to a possible ordering of the vertices of $P$ 
    induced by a linear functional. The shaded region
    \[
        R = \{\bx \in \R^2 \mid 
        \langle \bv_1 - \bv_2, \bx \rangle < 0, \ 
        \langle \bv_5 - \bv_3, \bx \rangle < 0, \
        \langle \bv_2 - \bv_5, \bx \rangle < 0, \
        \langle \bv_3 - \bv_4, \bx \rangle < 0 \}
    \]
    corresponds to the ordering $\bv_1, \bv_2, \bv_5, \bv_3, \bv_4$ from \Cref{ex:parralel-chambers}.
    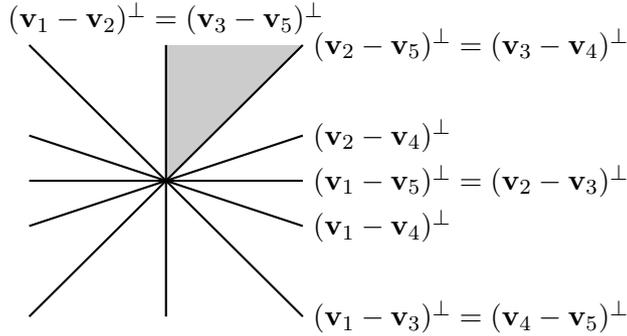
\begin{figure}[ht]
        \centering
        \begin{tikzpicture}[scale = 1.8]
\fill[black!20] (0,0) -- (1,1) -- (0,1);
\draw[thick] (1, 0) -- (-1, 0) ;
\node[anchor=west] at (1,0) {$(\bv_1 - \bv_5)^\perp = (\bv_2 - \bv_3)^\perp$};
\draw[thick] (1, -1/3) -- (-1, 1/3) ;
\node[anchor=west] at (1, -1/3) {$(\bv_1 - \bv_4)^\perp$};
\draw[thick] (0, 1) -- (0, -1) ;
\node[anchor=south] at (0, 1) {$(\bv_1 - \bv_2)^\perp = (\bv_3 - \bv_5)^\perp$};
\draw[thick] (-1, 1) -- (1, -1) ;
\node[anchor=west] at (1, -1) {$(\bv_1 - \bv_3)^\perp = (\bv_4 - \bv_5)^\perp$};
\draw[thick] (-1, -1) -- (1, 1) ;
\node[anchor=west] at (1, 1) {$(\bv_2 - \bv_5)^\perp = (\bv_3 - \bv_4)^\perp$};
\draw[thick] (1, 1/3) -- (-1, -1/3) ;
\node[anchor=west] at (1, 1/3) {$(\bv_2 - \bv_4)^\perp$};
\end{tikzpicture}
        \caption{The sweep arrangement of $P$ from \Cref{ex:sweep-arrangement}. The region $R$ is shaded in gray.}
        \label{fig:sweep-arrangement}
    \end{figure}
\end{example}

Each region $R \subset \regtra(P)$ corresponds to the same ordering of the vertices of $P$, i.e., the parallel arrangement $\chamtra(P)$ has the same combinatorial structure for all $\bu \in R\cap S^{d-1}$. Also in this case we have a finite number of regions, each of which defines a finite number of chambers. The following is the analogue of \Cref{prop:countingchamberscocircuit_intersectionbody} for the count of total parallel slicing chambers, including also lower-dimensional cells of the arrangements. This will provide also the proof of \Cref{thm:intro-combinatorialtypes}.

\begin{prop}\label{prop:countingchamberssweep_monotone}
    Assuming that $P\subset \R^d$ has $n$ vertices, the sweep arrangement $\chamtra(P)$ has at most $\binom{n}{2}$ affine hyperplanes. There are at most $O(n^{2d})$ $d$-dimensional polyhedral regions in $\regtra(P)$ and the total number of regions (including lower-dimensional cells) in the sweep arrangement is bounded by $O(n^{2d} 2^d)$. For each region in $\regtra (P)$ we have an arrangement of parallel hyperplanes \ $\chamtra (P)$, thus the final number of slicing chambers 
    (even of lower dimension) is $O(n^{2d+1} 2^{d})$.
\end{prop}    

\begin{proof} Each hyperplane of the sweep arrangement of a polytope $P$ is identified by a pair of vertices: their difference is the normal of the hyperplane. As in \Cref{prop:countingchamberscocircuit_intersectionbody}, by Zaslavsky's theorem (see \cite{Zaslavsky} or \cite[Proposition 2.4]{Stanley:IntroHyperplaneArr}) one can prove that the number of 
top-dimensional polyhedral regions $f_d(\regtra(P))$ of the sweep arrangement 
is 
\[
f_d(\regtra(P)) \leq \sum_{i=1}^d \binom{\binom{n}{2}}{i} \sim O(n^{2d}).
\]
Again, we can bound the number of $k$-faces in the arrangement by the inequality (see \cite{KFukudaetal-hyperpcount1991}):
\[
f_k(\regtra(P)) \leq \binom{d}{k} f_d(\regtra(P)), \quad 0 \leq k \leq d,
\]
which implies the desired bound on all polyhedral regions in the sweep arrangement 
(of all dimensions) of $O(n^{2d}2^d)$. 

For each such region of the sweep arrangement we have a corresponding parallel hyperplane arrangement $\chamtra$.
For each choice of translation direction $\bu$, the number of chambers is of the order of $O(n)$.
Indeed, the chambers of this arrangement are much easier, as they are defined, in the worst case, by $n$ parallel hyperplanes, one for each vertex of $P$; each chamber is bounded 
by exactly two parallel hyperplanes. Thus, the total number of $d$-dimensional chambers is $O(n^{2d}n)$ and the final total count of slicing chambers is bounded by $O(n^{2d+1}2^d)$ when we consider all the lower-dimensional regions and the corresponding chambers.
\end{proof}

\begin{proof}[Proof of \Cref{thm:intro-combinatorialtypes}]
    Given a polytope $P$, both \Cref{prop:countingchamberscocircuit_intersectionbody} and \Cref{prop:countingchamberssweep_monotone} provide an upper bound for the number of combinatorial types of hyperplane sections of $P$. Among the two bounds, the smallest is $O(n^{2d+1}2^d)$, and the claim follows.
\end{proof}

Recall our notation $H(\bu,\beta) = \{\bx \in \R^d \mid \langle \bu, \bx \rangle = \beta \}$.
Once more, we can extend \Cref{thm:parallel_sections} by first choosing the direction $\bu$, and then considering all parallel sections with normal vector $\bu$. This is a second way to parametrize all  affine hyperplane sections, and the following result is the analogue of \Cref{thm:piecewise_rat_translate_rotation}.

\begin{theorem}\label{thm:piecewise_rat_rotate_translation}
Let $P\subset \R^d$ be a full-dimensional polytope and let $f(\bx) = \sum_{\balpha} c_{\balpha} \bx^{\balpha}$ be a polynomial. Let $R \subset \R^d$ be a region of the sweep arrangement $\regtra(P)$ and let $C(\bu)\subset \R$ be a chamber of the parallel arrangement $\chamtra(P)$, for $\bu\in R$. Restricted to $\bu \in R \cap S^{d-1}$ and $\beta \in C(\bu)$, the integral $\int_{P \cap H(\bu,\beta)} f(\bx) \dx$ is a rational function in variables $u_1,\dots,u_d,\beta$.
\end{theorem}

\begin{proof}
    By \Cref{lem:sweep-arrangement} the chambers $C(\bu)$ of $\chamtra(P)$ are linearly dependent on $\bu$ when restricting to $\bu \in R\cap S^{d-1}$. Let $\beta \in C(\bu)$. As in the proof of \Cref{thm:piecewise_rat_translate_rotation}, we triangulate $P \cap H(\bu,\beta)$, obtaining the matrix 
    \[
    M_\Delta(\bu,\beta) = \begin{bmatrix}
    \bv_{j_2}(\bu,\beta) - \bv_{j_1}(\bu,\beta) \\
    \bv_{j_2}(\bu,\beta) - \bv_{j_1}(\bu,\beta) \\
    \vdots \\
    \bv_{j_{d}}(\bu,\beta) - \bv_{j_1}(\bu,\beta) \\
    \bu
    \end{bmatrix},\]
    \[
    \bv_{j}(\bu,\beta) = \frac{\beta}{\langle \bu, \bb_i-\ba_i\rangle} (\bb_i - \ba_i) + \frac{\langle \bu, \bb_i \rangle \ba_i - \langle \bu, \ba_i \rangle \bb_i}{\langle \bu, \bb_i - \ba_i\rangle},
    \]
    where $\bv_{j}(\bu,\beta)$ denotes the vertex of a simplex $\Delta$ in the triangulation $\mathcal T$, which lies on the interior of the edge $\conv(\ba_j + \bt, \bb_j + \bt)$ of $P$.
    Repeating the computation of \eqref{eq:long-computation} in the proof of \Cref{thm:integral_polynomial_over_P} yields
    {\small
    \begin{align*}
        \int_{P \cap H(\bu,\beta)} f(\bx) \dx  
        = \sum_{\Delta \in \mathcal{T}} \left| \det M_{\Delta}(\bu,\beta) \right| \sum_{\balpha} & \frac{c_{\balpha}}{(|\balpha|+d-1)!} \sum_{{\substack{ \bp \in \Z^d_{\geq 0} \\ \bp \leq \balpha}}}(-1)^{|\balpha|-|\bp|}
        \binom{\alpha_1}{p_1}\cdots \binom{\alpha_d}{p_d}  \cdot \\
        & \qquad\quad
        \cdot \sum_{ \substack{ \bk \in \Z_{\geq 0}^{d}, \\ | \bk| = |\balpha|}}\langle \ \bp, \ \bv_{j_1}(\bu,\beta) \ \rangle^{k_1}\cdots \langle \ \bp, \ \bv_{j_d}(\bu,\beta) \ \rangle^{k_{d}},
    \end{align*}}
    which is a rational function in $u_1,\dots,u_d,\beta$.
\end{proof}

This result implies that $\int_{P \cap H(\bu,\beta)} f(\bx) \dx$ is a piecewise rational function. The domains over which it is rational are now polyhedra, in contrast to \Cref{sec:translating-the-rotation}, up to restriction to a sphere, and they live in $\R^{d+1}$:
\[
\{ (\bu,\beta)\in \R^{d+1} \,|\, \bu \in R, \beta \in C(\bu)\}\cap \left( S^{d-1}\times \R \right) ,
\]
for some region $R\subset \regtra(P)$ and chamber $C(\bu)\subset \chamtra(P)$, where $\bu \in R\cap S^{d-1}$. In this translational setting, the chamber $C(\bu)$ has shape $\langle \bv_i,\bu \rangle \leq \beta \leq \langle \bv_{i+1},\bu \rangle$, and therefore it defines a polyhedron in $\R^{d+1}$. There is one such polyhedron, and hence one such rational function, for each chamber of every region. The following is the analogue of \Cref{prop:deg_rational_pieces} in the translational setting. 

\begin{prop}\label{prop:deg_pieces_translation}
    In each of the slicing chambers the rational function $\int_{P \cap H(\bu,\beta)} f(\bx) \dx$ 
    respects the following degree bounds:    \begin{equation}\label{eq:int_deg_bound_translation}
        \begin{aligned}
        \deg\,  \operatorname{numerator} \left( \int_{P\cap H(\bu,\beta)} f(\bx) \dx \right) &\leq \Big(f_1(P) - (d-1)\Big) ( D+d-1 ) +1, \\
        \deg\,  \operatorname{denominator} \left( \int_{P\cap H(\bu,\beta)} f(\bx) \dx \right) &\leq \Big(f_1(P) - (d-1)\Big) ( D+d-1 ).
        \end{aligned}
    \end{equation}
    Here $f_i(P)$ denotes the number of $i$-dimensional faces of $P$ and $D = \deg f$.
\end{prop}
\begin{proof}
    We can repeat exactly the same computations in the proof of \Cref{prop:deg_rational_pieces}. The only difference is that in this framework lies in the rational function $\left| \det M_{\Delta}(\bu,\beta) \right|$. All but the last row of the matrix have entries that are quotients of two quadratic polynomials. The denominator of each entry of the $i$th row is $ \langle \bu, \bb_{j_i}-\ba_{j_i}\rangle \langle \bu, \bb_{j_1}-\ba_{j_1}\rangle$ for all $i=2,\ldots,d$. 
    The determinant is thus the quotient of a polynomial of degree $2(d-1)+1$ and another polynomial of degree $2(d-1)$. Therefore, in the end, the degrees of numerator and denominator of $\int_{P \cap H(\bu,\beta)} f(\bx) \dx$ differ only by $1$.
\end{proof}

\begin{example}\label{ex:sweep-integral}
We continue \Cref{ex:parralel-chambers,ex:chambers-translational,ex:sweep-arrangement}. Let $R$ denote the region of the sweep arrangement of $P$ defined in \Cref{ex:sweep-arrangement}, which is depicted in \Cref{fig:sweep-arrangement} as shaded cone. 
\Cref{fig:sweep-integral} shows the volume of $P \cap H(\bu,\beta)$, i.e., 
the integral of the constant function $1$, parametrically for directions 
$\bu \in R\cap S^{d-1}$ and $\beta \in \chamtra(P)$. Note that, in accordance with \Cref{rmk:closures}, the rational functions that are displayed specialize 
to the ones in \Cref{fig:chambers-translational} when evaluated at $\bu = \tfrac{1}{\sqrt5} (1,2)$.
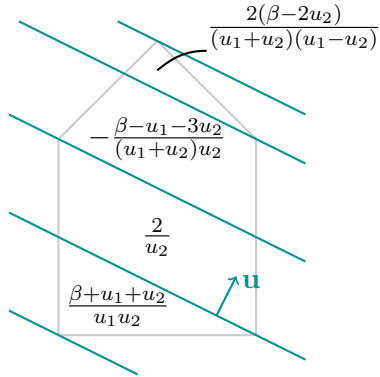
\begin{figure}[ht]
    \centering
    \begin{tikzpicture}[scale = 1.3]
    \draw[thick, color=black!20] (-1,-1) -- (1,-1) -- (1,1) -- (0,2) -- (-1,1) -- (-1,-1)  ;
    \draw[thick, cb-green-sea] (-0.2, -1.4) -- (-1.5, -0.75) ;
    \draw[thick, cb-green-sea] (1.5, -1.25) -- (-1.5, 0.25) ;
    \draw[thick, cb-green-sea] (1.5, -0.25) -- (-1.5, 1.25) ;
    \draw[thick, cb-green-sea] (1.5, 0.75) -- (-1.4, 2.2) ;
    \draw[thick, cb-green-sea] (1.5, 1.25) -- (-0.4, 2.2) ;
    \node at (-0.4,-0.7) {$\tfrac{\beta + u_1 + u_2}{u_1 u_2}$};
    \node at (0,0) {$\tfrac{2}{u_2}$};
    \node at (0,1) {$-\tfrac{\beta - u_1 - 3u_2}{(u_1 + u_2)u_2}$};
    \node[anchor=south west] at (0.4,1.8) {$\tfrac{2(\beta - 2 u_2)}{(u_1 + u_2)(u_1 - u_2)}$};
    \draw [thick,black] (0.5,1.9) to [out=180,in=40] (0,1.7);
    %\node at (0,-1.8) {$\int_{P\cap H(\bu,\beta)} 1 \dx = \vol(P\cap H(\bu,\beta))$};
    \draw[->,thick, cb-green-sea] (0.6,-0.8) -- (0.8,-0.4);
    \node[cb-green-sea] at (0.95,-0.45) {$\bu$};
\end{tikzpicture}
    \caption{The function $\int_{P\cap H(\bu,\beta)} 1 \dx = \vol(P\cap H(\bu,\beta))$ for $\bu \in R\cap S^{d-1}$, as defined in \Cref{ex:sweep-integral}.}
    \label{fig:sweep-integral}
\end{figure}
\end{example}

To conclude, we note that this section presented two different cell decompositions, inspired by the rotational and the translational approach respectively. Each of these decompositions is given by the choice of a pair of hyperplane arrangements,
namely the cocircuit arrangement and the central arrangement, or the sweep arrangement and the parallel arrangement. The structure was summarized in \Cref{table:arrangements-overview}, presented in the introduction.

\section{Computational complexity of finding optimal slices}\label{sec:complexity}

The aim of this section is to use the results of \Cref{thm:intro_structure}, explained in more details in \Cref{sec:mergingslices}, in order to obtain \Cref{thm:intro_algorithm}. 
We make use of the structure of the pairs of hyperplane arrangements from the previous sections for optimization purposes. The proofs of \Cref{thm:piecewise_rat_translate_rotation,thm:piecewise_rat_rotate_translation} imply the existence of algorithms to find a slice of a polytope where the integral of a polynomial attains the largest value. In particular, we can find the slice with the largest volume. In the same spirit, we can also find the slice of $P$ with maximal (or minimal) combinatorial properties. 

Related optimization problems involving halfspaces and projection can also be solved with minor adjustments of our algorithm. 
Furthermore, we can combine the optimization criteria in \Cref{thm:intro_algorithm} to find, e.g., a slice of maximal volume of a fixed combinatorial type.
All of these achievements are related to both problems in combinatorial optimization and convex geometry. We will prove that 
some of them are in general hard from the point of view of complexity theory, but our algorithms have polynomial complexity in fixed dimension.
In the following sections, we always assume that the input of our algorithms are rational. Therefore, from now on, a polytope $P$ is going to be rational, unless otherwise stated, and a polynomial $f$ is going to have rational coefficients.

\subsection{Polynomial-time complexity in fixed dimension}\label{sec:polytime_complexity}
Probably the most important feature of our hyperplane arrangements  and the associated chamber decomposition is that within each chamber, an affine hyperplane intersects the polytope $P$ in a fixed set of edges. Thus, they capture all the combinatorial types of hyperplane sections of a polytope. This allows us to tackle optimization questions regarding combinatorial aspects of the slices of a polytope, proving the first part of \Cref{thm:intro_algorithm} \labelcref{mainthm:sect_k_faces}.

\begin{prop}\label{prop:polytime_kfaces}
        We have an algorithm that receives as input a polytope $P\subset \R^d$, and outputs the maximizer and the maximum of $f_k(P\cap H)$ over all hyperplanes $H$, where $f_k(K)$ is the number of $k$-dimensional 
        faces of $K$. The algorithm runs in polynomial time for $d$ fixed. 
\end{prop}

\begin{proof}
    The proof is elementary and it is follows directly from our construction of the hyperplane arrangements. We will prove it using the sweep arrangement and parallel slices, but an analogous proof can be obtained in the rotational setting. Fix a region $R\subset\regtra(P)$ and a chamber $C(\bu) \subset \chamtra(P)$ for some $\bu\in R\cap S^{d-1}$. Then, for every $\bu$ in the relative interior of $R\cap S^{d-1}$ and for every $\beta$ in the interior of $C(\bu)$, the hyperplane section $P \cap H(\bu,\beta)$ has the same combinatorial type. By \Cref{prop:countingchamberssweep_monotone}, we have polynomially many cases to check for fixed dimension $d$.
\end{proof}
A straightforward generalization of the proof of \Cref{prop:polytime_kfaces} provides a proof of the weighted version of the optimization problem stated in \Cref{thm:intro_algorithm} \ref{mainthm:weighted_k_faces}. Indeed, it is enough to do the weighted count in each chamber of each region, and compare them. The second part of \Cref{thm:intro_algorithm} \labelcref{mainthm:sect_k_faces} also follows from \Cref{prop:polytime_kfaces}, since the hyperplanes in a given slicing chamber separate the same sets of $k$-dimensional faces of $P$.

Based on our hyperplane arrangement and our previous results, we can also solve some metric optimization problems.
For each chamber of each region, the function describing the integral of a polynomial (or, in particular, the volume) over a section of $P$ is piecewise rational by \Cref{thm:piecewise_rat_translate_rotation,thm:piecewise_rat_rotate_translation}. More precisely, let $R$ be a region of $\mathcal{R}(P)$ for either $\circlearrowleft$ or $\text{\footnotesize\rotatebox{-15}{$\uparrow$}}$. Let $C$ be a slicing chamber of $\mathcal{C}$ for points in the given region $R$. Then, there exist two polynomials $p_C,q_C$ such that 
$\int_{P \cap H} f(\bx) \dx =\tfrac{p_C}{q_C}$ for a given polynomial $f$ and hyperplanes $H$ from the slicing chamber.
We can ask for the hyperplane section which maximizes $\int_{P \cap H} f(\bx) \dx$ by maximizing the single rational functions in their respective chambers, and then taking the maximum over this finite list of values.
From the point of view of complexity, the two types of hyperplane arrangements behave in similar but different ways:

\begin{minipage}{.475\textwidth}
    \begin{align}\label{eq:max_problems_radial}
    \begin{split}
        \max\; & \frac{p_C(\bt,\bu)}{q_C(\bt,\bu)} \\
        \text{s.t. } & \bt \in R \subset \regrad(P), \\
         & \bu \in C(\bt) \subset \chamrad(P+\bt), \\
         & \bu \in S^{d-1},
    \end{split}
    \end{align}
\end{minipage}
\begin{minipage}{.475\textwidth}
    \begin{align}\label{eq:max_problems_translational}
    \begin{split}
        \max\; & \frac{p_C(\bu,\beta)}{q_C(\bu)} \\
        \text{s.t. } & \bu \in R \subset \regtra(P), \\
         & \bu \in S^{d-1}, \\
        & \beta \in C(\bu) \subset \chamtra(P).
    \end{split}
    \end{align}
\end{minipage}
~\\

Out of any of these optimization problems, we get (at least) one solution for every chamber of every region. Comparing all the results leads to a maximum over all hyperplane sections. 

\begin{example}
    Continuing the showcase of the pentagon, we give an example for the maximization of the integral of the constant polynomial $f = 1$, both for the radial and the translational case within a fixed region. In other words, we seek to find the slice of maximum volume in this example.
    We begin with the radial case.
    The pentagonal region $R_\circlearrowleft \subset \regrad(P)$ from \Cref{ex:regions-radial} is
    \begin{align*}
    R_\circlearrowleft = \{\bt \in \R^2 \mid 
    & -3t_1 + t_2 \geq -2, \ -t_1 -t_2 \geq 0, \  3t_1 + t_2 \geq -2 
     t_1 - t_2 \geq 0, \  t_2 \geq -1 \},
    \end{align*}
    and is shown in \Cref{fig:regions-radial-cocircuit}
    (recall from \Cref{rmk:closures} that we are allowed to take closures of regions and chambers).
    We have seen in \Cref{ex:radial-integral-parametric} that, when restricting to 
    $\bt \in R_\circlearrowleft$, one of the slicing chambers in the central arrangement is
    \begin{align*}
    C_\circlearrowleft(\bt) 
    & = \{ \bx \in \R^2 \mid (t_1 - 1)x_1 + (t_2 - 1)x_2 \geq 0, \ t_1x_1 + (t_2 + 2)x_2 \geq 0  \},
    \end{align*}
    and that, restricted to this region and chamber, the volume of $P \cap H(\bt,\bu)$ is given by
    \[
      \frac{p(\bt,\bu)}{q(\bt,\bu)} =   \frac{-(t_1 u_1 + t_2 u_2 + 3 u_1 - u_2)}{(u_1 - u_2) u_1},
    \]
    where $\bu\in C_\circlearrowleft(\bt) \cap S^{d-1}$.
    Maximizing this function subject to $\bt \in R, \bu \in C_\circlearrowleft(\bt)\cap S^{d-1}$ yields $\sqrt{10}$, and a maximizer is given by
    $\bt = (\tfrac{5}{12}, -\tfrac{3}{4}) \in R$ and $\bu = \frac{1}{\sqrt{10}}(-3,1) \in C_\circlearrowleft(\bt)\cap S^{d-1}$. Indeed, for this choice of parameters, $(P + \bt) \cap \bu^\perp$ yields the line segment $\conv((-\frac{7}{12},-\frac{7}{4}),(\frac{5}{12},\frac{5}{4}))$ and the Euclidean volume of this line segment inside its affine span is $\sqrt{10}$.
    
    For the translational case, we have seen in \Cref{ex:sweep-arrangement} that 
    one of the regions in the sweep arrangement $\regtra(P)$ is  
    \begin{align*}
    R_\trarrow = \{\bx \in \R^2 \mid 
         x_1 \geq 0, x_1 - x_2 \leq 0
    \}.
    \end{align*}
    By \Cref{ex:sweep-integral}, when restricting to the chamber
   $C_\trarrow(\bu) = \{\beta \in \R \mid 
         \langle \bu, \bv_5 \rangle \leq \beta \leq \langle \bu, \bv_3 \rangle\}$, for $\bu \in R_\trarrow\cap S^{d-1}$,
    the volume is given by 
    \[
        \frac{p_C(\bu,\beta)}{q_C(\bu)} = \frac{-(\beta - u_1 - 3u_2)}{(u_1 + u_2)u_2}.
    \]
    The maximum of this function subject to $\bu \in R_\trarrow \cap S^{d-1}$, $\beta \in C_\trarrow(\bu)$ is $2\sqrt{2}$ and obtained at $\bu = \tfrac{1}{\sqrt{2}}(1,1) \in R_\trarrow\cap S^{d-1}$ and $\beta = \langle \bu, \bv_5 \rangle = 0$, which is smaller than the maximum we have found in the rotational case. Indeed, one can check that the maximum value $\sqrt{10}$ is (among other regions) attained in the region 
    \[
    R'_\trarrow = \{\bx \in \R^2 \mid 
         x_1 - 3x_2 \leq 0, x_1 - x_2 \geq 0
    \},
    \]
    which contains the unit vector $\bu = \tfrac{1}{\sqrt{10}}(3,-1) \in R'_\trarrow$.
\end{example}

How do the two optimization problems \eqref{eq:max_problems_radial} 
and \eqref{eq:max_problems_translational} compare?
Both problems must be solved for a finite (in fact polynomial) number 
of regions, each with a finite number of chambers. For the rotational version, 
namely \eqref{eq:max_problems_radial}, \Cref{prop:countingchamberscocircuit_intersectionbody} gives an upper bound on the total number of slicing chambers, whereas \Cref{prop:countingchamberssweep_monotone} bounds the total number for the translational version \eqref{eq:max_problems_translational}. 
Both of these bounds are polynomial in the number of vertices of $P$ for fixed dimension $d$. 
However, these two problems are very different from a complexity point of view. The complexity analysis for the maximization problem is based on \cite[Algorithm 14.9]{BPR:algoinRAG}, which we summarize in the following lemma.

\begin{lemma}[{\cite[Algorithm 14.9]{BPR:algoinRAG}}]\label{lemma:algo_Basu}
    Let $D$ be an ordered domain contained in a real closed field (such as $D = \Z$ or $D = \Q$). Let $\mathcal{P}\subset D[x_1,\ldots,x_k]$ be a finite set and let $S$ be a semialgebraic set defined by polynomial inequalities involving only the polynomials in $\mathcal{P}$. Consider $F\in D[x_1,\ldots,x_k]$. Denote by $s$ 
    the number of elements of $\mathcal{P}$, by $\delta$ an upper bound on the degree of the elements of $\mathcal{P}$ and $F$. Then, the complexity of computing an 
    infimum of $F$ over $S$, and a minimizer in case such point exists, is $s^{2k+1} \delta^{O(k)}$.
\end{lemma}

We now apply \Cref{lemma:algo_Basu} to our concrete situation of an optimization problem coming from our slicing chambers. Recall that we assume the polytope $P$ and the polynomial $f$ to be defined over $\Q$.

\begin{prop} 
\label{prop:one cellopt}
    Let $P$ be a polytope in fixed dimension $d$. For each region $R$ and each chamber $C$ in both the rotational or the translational setting, there is a polynomial time algorithm that computes the maximum and the maximizer of the rational function $\int_{P\cap H} f(\bx)\dx$. 
\end{prop}

\begin{proof}
    We prove the result in the translational setting. A similar estimate can be obtained in the rotational setting, as explained below in \Cref{rmk:complexity_rotational}.
    Fix a region $R \subset \regtra(P)$ and a slicing chamber $C \subset \chamtra(P)$ for $\bu \in R$. Then $R$ and $C$ are respectively of the form
    \[
    A \bu \geq \mathbf{0}, \qquad \qquad
    \ell_1(\bu) \leq \beta \leq \ell_2(\bu),
    \]
    where $A$ is a matrix and $\ell_i (\bu) = \langle \bv_i , \bu \rangle$ for some vertices $\bv_1, \bv_2$ of $P$. 
    Restricted to hyperplanes $H(\bu,\beta), \bu \in R, \beta \in C$, the integral is a rational function $\int_{P \cap H(\bu,\beta)} f(\bx) \dx = \frac{p_C(\bu,\beta)}{q_C(\bu)}$. 
    We introduce an auxiliary variable $z\in \R$ and define the following set:
    \[
    S_C = \{ (\bu,\beta,z) \in \R^{d+2} \,|\, q_C(\bu)\, z - p_C(\bu,\beta) = 0, \, A \bu \geq \mathbf{0},\, \sum_{i=1}^d u_i^2 = 1,\,  \ell_1(\bu) \leq \beta \leq \ell_2(\bu)\}.
    \]
    Finding the maximum of \eqref{eq:max_problems_radial} is equivalent to 
    finding the maximum of $z$ subject to $(\bu,\beta,z)\in S_C$. We can find such an 
    optimal value by using \cite[Algorithm 14.9]{BPR:algoinRAG}.
    Let $\delta$ be the 
    degree of $q_C(\bu,\beta)\, z - p_C(\bu,\beta)$, which by \eqref{eq:int_deg_bound_translation} is bounded by $\big(f_1(P) - (d-1)\big) (D+d-1)+1$, where $f_1(P)$ denotes the number of edges of $P$. 
    Bounded the number of rows of $A$ by the number of hyperplanes in $\regtra(P)$ yields that 
    the number of polynomials defining $S_C$ is at most $s = 1+\binom{n}{2}+1+2$. 
    Thus, according to \Cref{lemma:algo_Basu}, the complexity of finding the optimal value of our problem is at most
    \[
    s^{2d+5}\delta^{O(d)} = \left(\binom{n}{2} +4\right)^{2d+5}\Big(\big(f_1(P) - (d-1)\big) (D+d-1)+1\Big)^{O(d)},
    \]
    where $n = f_0(P)$ is the number of vertices $P$.
    For fixed $d$, the running time is polynomial.
\end{proof}

\begin{remark}\label{rmk:complexity_rotational}
    Notice that we can solve the optimization problem \eqref{eq:max_problems_translational} in complete analogy by applying \cite[Algorithm 14.9]{BPR:algoinRAG}. 
    The comparison of the complexity of the two problems reduces to the comparison of the total number of chambers and the number of hyperplanes needed to define each chamber. Since these two numbers depend in any case polynomially only on the number of vertices of $P$ (when the dimension $d$ is fixed), the complexity of our algorithm is polynomial in both variants. As noted in \Cref{sec:rotating-the-translation,sec:translating-the-rotation}, the domains over which the objective function is rational are polytopes in the translational case and semialgebraic sets defined by quadrics in the rotational case. Overall, the translational setting is cheaper.
\end{remark}

By \Cref{prop:countingchamberscocircuit_intersectionbody,prop:countingchamberssweep_monotone} we have to apply \Cref{lemma:algo_Basu} only polynomially many times. Putting all of this together, we can prove our main result on optimizing integrals and volumes over the slices of a polytope. This was stated earlier as \Cref{thm:intro_algorithm} \labelcref{mainthm:section_vol_int}. We summarize the algorithm analysis in the rotational and translational settings in \Cref{tab:summary_complexity}.

\begin{theorem}\label{notmainthm:section_vol_int}
    We have an algorithm that receives as input a rational convex polytope $P\subset \R^d$ and a polynomial $f$ of degree $D$ with rational coefficients, and outputs the maximizer and the maximum of $\int_{P\cap H} f(\bx)\dx$ over all affine hyperplane sections. The algorithm runs in polynomial time for fixed $d$. 
\end{theorem}
\begin{proof} 
In \Cref{sec:rotating-the-translation,sec:translating-the-rotation} we have presented two different decompositions of the space of all possible affine hyperplane sections of $P$.
To compute the section which maximizes $\int_{P\cap H} f(\bx)\dx$ we need to solve problem \eqref{eq:max_problems_radial} or \eqref{eq:max_problems_translational} for each region and chamber respectively, yielding a total number of $(\# \text{regions}) (\# \text{chambers})$ many optimization problems. 
To set up each of the optimization problems we detect the edges of $P$ which are intersected by the hyperplanes, yielding  the necessary combinatorial information of the slice $P \cap H$. This allows us to compute a triangulation of $P\cap H$ and to write the objective function as in the proofs of \Cref{thm:piecewise_rat_translate_rotation,thm:piecewise_rat_rotate_translation}. By construction, the triangulation will have number of simplices bounded polynomially in terms of the number of vertices of $P \cap H$. More precisely, for a $k$-dimensional polytope with $m$ vertices, the largest number of top-dimensional simplices in the triangulation is bounded by $O(m^{\lceil (k+1)/2 \rceil})$ \cite[Proposition 2.6.5]{TriangBook2010}. 

In our situation, $k+1=d$, which is a constant, and $m$, the number of vertices  of the slice $P \cap H$, is polynomially bounded by the number of edges and vertices of $P$. This is a polynomial in the number of variables, as well as the number of facet inequalities of $P$, and can be computed in polynomial time in fixed dimension. Thus, the mathematical programs can be encoded in time polynomial in the input size of $P$. Finally, \Cref{prop:one cellopt} and \Cref{rmk:complexity_rotational} imply that both Problem \eqref{eq:max_problems_radial} and Problem \eqref{eq:max_problems_translational} are solvable in polynomial time. We repeat this for each slicing chamber, which means we do this polynomially many times by \Cref{prop:countingchamberscocircuit_intersectionbody,prop:countingchamberssweep_monotone}. Putting all together the total complexity of computation is polynomial in the input.
\end{proof}
\begin{table}[!ht]
    \centering
    {\def\arraystretch{1.5}
    \begin{tabular}{c|c|c|c|c|c}
         & \multirow{2}{*}{\textbf{regions} }
         & \multirow{2}{*}{$\substack{\text{\normalsize \textbf{upper bound}} \vspace*{0.2em} \\ \text{\normalsize \textbf{num. of regions}}}$}  
         & \multirow{2}{*}{\textbf{chambers}} 
         & \multirow{2}{*}{$\substack{\text{\normalsize \textbf{upper bound}} \vspace*{0.1em} \\ \text{\normalsize \textbf{num. of chambers}} \vspace*{0.3em} \\  \text{\normalsize \textbf{per region}} } $}
         & \multirow{2}{*}{$\substack{\text{\normalsize \textbf{cost of}} \vspace*{0.2em} \\ \text{\normalsize \textbf{optimizing in}} \vspace*{0.2em} \\ \text{\normalsize \textbf{one chamber}}}$}  \\
         &  &  &  &  &   \\
        \hline
        \multirow{2}{*}{$\circlearrowleft$} &  
        \multirow{2}{*}{$\regrad$}
        & \multirow{2}{*}{$\sum_{i=1}^d \binom{\binom{n}{d}}{i}$} & \multirow{2}{*}{$\chamrad$} & \multirow{2}{*}{$\sum_{i=1}^d \binom{n}{i}$} & $s^{2d+5}\delta^{O(d)}$ \vspace*{-0.5em} \\
         & 
         &  &  &  &  as in Rmk. \ref{rmk:complexity_rotational}  \\
        \hline
        \multirow{2}{*}{\text{\rotatebox{-15}{$\small\uparrow$}}} & \multirow{2}{*}{$\regtra$} & \multirow{2}{*}{$\sum_{i=1}^d \binom{\binom{n}{2}}{i}$} & \multirow{2}{*}{$\chamtra$} & \multirow{2}{*}{$n-1$} & $s^{2d+5}\delta^{O(d)}$  \vspace*{-0.5em} \\
         &  &  &  &  & as in Prop. \ref{prop:one cellopt} \\
    \end{tabular}
    }
    \caption{Details of complexity analysis in the two methods of classifying slices. In each case $s,\delta$ are, respectively, the number of equations and upper bound on degree of the equations. They differ for the two decomposition methods. }
    \label{tab:summary_complexity}
\end{table}

It is important to note that, with little effort, the methods we used to prove 
\Cref{thm:intro_algorithm} \ref{mainthm:section_vol_int} can be extended to prove the additional computational results in items \ref{mainthm:halfspaces_int} and \ref{mainthm:projection} on projections and halfspace intersections. The latter is related to the \emph{densest hemisphere problem}: Given a set $K$ of $n$ points on the unit sphere $S^d$ in $d$-dimensional Euclidean space, a hemisphere of $S^d$ is densest if it contains a largest subset of $K$. This problem is already known to be solvable in polynomial time when the dimension $d$ is fixed \cite{Johnson+Preparata1978}.

The proofs of the complexity of our algorithm in the case of projections and halfspace sections are in complete analogy to the previous discussion. Thus, we only sketch the key missing details in the following proof.

\begin{proof}[Proof of \Cref{thm:intro_algorithm} \ref{mainthm:halfspaces_int}, \ref{mainthm:projection}]
For \ref{mainthm:halfspaces_int}, it suffices to note that the same slicing chambers of the central arrangement $\chamrad(P)$ that we used earlier, keep the vertices on the corresponding halfspace $H^+_0$ intact, and the combinatorial type of $P \cap H^+_0$ does not change. Similarly, the triangulation we use 
for the section $P \cap H_0$ can be easily extended to a triangulation of $P \cap H^+_0$. Therefore, \Cref{thm:integral_polynomial_over_P} can be generalized to prove that in each chamber the integral of the halfspace intersection is a rational function. Then, an analogous version of \Cref{notmainthm:section_vol_int} implies the polynomiality of the maximization or minimization.

\Cref{thm:intro_algorithm} \ref{mainthm:projection} is based on the results proved in \Cref{section:projections}. We can compute both the polar $P^\circ$ of our polytope and the associated hyperplane arrangement $\chamrad(P^\circ)$ in polynomial time in fixed dimension, and $\chamrad(P^\circ)$ has polynomially many maximal open chambers. Since projections do not distinguish between central and affine hyperplanes, we do not need to use regions in this framework. In each chamber the function $\int_{\pi_\bu(P)} f(\bx)\dx$ is a polynomial in $\bu\in C\cap S^{d-1}$, and its degree is bounded by $2d(D + 1)-1$. Therefore, we can use \Cref{lemma:algo_Basu} to compute the maximum or the minimum of this polynomial over all projections, in polynomial time when the dimension $d$ is fixed.
\end{proof}

\subsection{Hardness in non-fixed dimension} \label{subsec:hardness}

In \Cref{sec:polytime_complexity} we showed that several optimization problems are solvable in polynomial time for fixed dimension. The purpose of this section is to show that this does not hold true when the dimension is not fixed. 
We begin by proving the hardness of finding slices of maximum volume. For this, we use the following lemma, whose proof was kindly provided to us by Francisco Criado Gallart.

\begin{lemma}\label{lemma:pyramid_large_section}
    Let $K\subset \{\bx \in \R^d \mid x_d=0\}$ be a $(d-1)$-dimensional 
    convex body. There exists a pyramid $C\subset\R^d$ over $K$ such that 
    \[
    \max_{H \text{ \textup{aff. hyperplane} }} \vol (C\cap H) = \vol (K)
    \]
    where $K = C\cap \{\bx \in \R^d \mid x_d=0\}$ and the maximum is uniquely attained 
    at this hyperplane.
\end{lemma}
\begin{proof}
    Consider a point $\bp = (p_1,\ldots,p_{d-1})$ in the interior of $K$, let $h>0$ and define the pyramid over $K$ as 
    \[
    C = \conv \big(K, \sma \bp \\ h \strix\big) \subset \R^d.
    \]
    Our claim is that there exists $h>0$ such that the largest hyperplane section of $C$ is the base $K$. The proof will be divided into two steps. The first one concerns proving that there exists $h>0$ such that the width of $C$ is achieved uniquely in direction of $(0,\ldots,0,1)$. Recall that the width of $C$ is defined as \cite{grandpabibleI-1994}
    \[
    \omega(C) = \min_{\bu \in S^{d-1}} \left( \max_{\bx\in C} \langle \bu,\bx\rangle - \min_{\bx\in C} \langle \bu,\bx\rangle \right).
    \]
    Since $\bp$ lies in the interior of $K$, there exists $r>0$ such that the $(d-1)$-dimensional ball $B_{\bp,r}$ centered at $\bp$ of radius $r$ is strictly contained in $K$. Therefore, taking $h=\frac{r}{2}$, the cone $\widetilde{C} = \conv \big( B_{\bp,r}, \sma \bp \\ h \strix \big)$ is contained in $C$. Using some elementary geometry, it can be proved that $\omega(\widetilde{C}) = \frac{r}{2}$ and it is achieved uniquely in direction $(0,\ldots,0,1)$. Since $\widetilde{C}\subset C$, the width of $C$ is at least $\frac{r}{2}$. By construction, the width of $C$ is strictly larger than the width of $\widetilde{C}$ in all directions except $(0,\ldots,0,1)$, where they coincide. Hence, $\omega(C) = h$.

    The second step of the proof is to show that, with this choice of $h$ 
    so that the width of $C$ is realized by the last coordinate vector, 
    the base of $C$ is the hyperplane section with the largest volume.
    Notice that the volume of a section which does not intersect the base $K$ is always smaller that the volume of $K$ itself, by construction. Therefore, 
    assume that the section with largest volume is not the base but it is defined 
    by a hyperplane $H$ such that $H\cap K \neq \emptyset$ and let $\bu$ be a 
    normal vector to $H$. Let $\mathbf{p}_1,\mathbf{p}_2 \in C$ be respectively 
    the maximizer and minimizer of $\langle \bu,\bx\rangle$ over $C$. Then $\langle \bu,\mathbf{p}_1\rangle-\langle \bu,\mathbf{p}_2\rangle > \omega(C) =h$ and 
    the bipyramid $\conv \big((H\cap C), \mathbf{p}_1,\mathbf{p}_2 \big)$ is contained 
    in $C$. Then, we have the following chain of (in)equalities:
    \begin{align*}
        \frac{1}{d} \vol K \cdot h &= \vol C > \vol \left( \conv \big((H\cap C), \mathbf{p}_1,\mathbf{p}_2 \big) \right)\\
        &= \frac{1}{d} \vol (H\cap C) \cdot (\langle \bu,\mathbf{p}_1\rangle-\langle \bu,\mathbf{p}_2\rangle)\\[0.5em]
        &> \frac{1}{d} \vol K \cdot (\langle \bu,\mathbf{p}_1\rangle-\langle \bu,\mathbf{p}_2\rangle)
    \end{align*}
    which implies that $\langle \bu,\mathbf{p}_1\rangle-\langle \bu,\mathbf{p}_2\rangle < h$, giving a contradiction since we chose $h$ to be the width of $C$.
\end{proof}

It is well-known that it is $\# P$-hard to compute volumes of polytopes in arbitrary dimension, 
when presented in facet or vertex descriptions \cite{DyerFrieze88,BrightwellWinkler91,khachiyan93,lawrence91}. 
We can therefore combine these classical results with \Cref{lemma:pyramid_large_section}, yielding that it is hard to find the slice of a polytope with maximal volume.
\begin{prop}\label{prop:volume-nonfixed-hard}
    Let $P$ be a rational polytope of arbitrary dimension. It is $\#P$-hard to compute the  volume of the hyperplane section $P\cap H$ with largest volume. 
\end{prop}
\begin{proof}
    Suppose by contradiction that one could compute the volume of the largest hyperplane section of any polytope $P$ efficiently. Now consider the pyramid $P=\conv (Q, \sma \bp \\ h \strix)$ where $\bp\in Q, h > 0$ and $Q$ is a zonotope or an order polytope, or any polytope for which it is known that it is $\#P$-hard to compute the volume (for details see \cite{DyerFrieze88,BrightwellWinkler91,DyerGritzmannHufnagel98}).
By \Cref{lemma:pyramid_large_section} we can chose $h>0$ in such a way that the section of $P$ with largest volume is actually $Q$. This implies, by our hypothesis, that we could compute the volume of $Q$ efficiently too, which gives a contradiction.
\end{proof}

\Cref{prop:volume-nonfixed-hard} implies that, given a family of polytopes, computing their hyperplane sections with the largest volume is in general a hard task. However, by \Cref{thm:intro_algorithm}\ref{mainthm:section_vol_int}, with our hyperplane arrangements we can compute it in polynomial time when the dimension of the polytopes in the family is fixed. Similarly,  by \Cref{thm:intro_algorithm}\ref{mainthm:sect_k_faces} and \ref{mainthm:halfspaces_int} we can detect the slices with the maximum number of vertices. The following proposition shows that the weighted analogue of this task is $NP$-hard. Furthermore we show that same holds for the of maximizing the number of vertices which are contained in the intersection of $P$ with a central halfspace.

\begin{prop} Let $P$ be a polytope of arbitrary dimension with weights on its edges. It is $NP$-hard to compute the hyperplane section $P\cap H$ that maximizes the sums of weights of edges it intersects.
Similarly, given an arbitrary polytope with vertices in a sphere with center $\bf{o}$, finding a halfspace section, for hyperplanes passing through $\bf{o}$, that maximizes the number of vertices inside the halfspace is $NP$-hard for arbitrary dimension.
\end{prop}

\begin{proof} The hardness proof will be a reduction to the
It is well-known  the MAX-CUT problem is NP-hard (it is in the original list of famous NP-hard problems in \cite{NPhardbook79}). W
Note that every graph $G$ with $N$ nodes is a subgraph of the complete graph $K_N$, which is the  graph of the $(N-1)$-dimensional simplex $\Delta_N$. We assign weight $1$ to the edges of $G$ and zero otherwise. Now note that every subset of vertices of $\Delta_N$ 
can be separated by a hyperplane as long as edges that are not in $G$ have weight 
zero they do not count in the weighted cuts. Thus, solving the max cut on the original graph $G$ is equivalent to solving the task of \Cref{thm:intro_algorithm} \ref{mainthm:weighted_k_faces} for vertices inside the simplex $\Delta_N$.
We have proved that the problem is already hard for simplices whose edges have weights zeros and ones, therefore the hardness of the stronger statement in the proposition follows.

For the second part of the statement, notice that the \emph{densest hemisphere problem}, mentioned already in \Cref{sec:polytime_complexity}, is a special case of \Cref{thm:intro_algorithm} \ref{mainthm:halfspaces_int}, where the vertices of the polytope are taken on the sphere. Therefore its hardness directly implies hardness 
for our statement. Now, in a now classic paper in computational geometry Johnson and Preparata showed when $d$ is fixed there exists a polynomial time algorithm which solves the problem in polynomial time \cite{Johnson+Preparata1978}. But they also showed densest hemisphere is known to be $NP$-hard when the dimension $d$ are arbitrary, which implies our statement is hard in arbitrary dimension.
\end{proof}

To conclude, we conjecture that this is $NP$-hard even without weights and more in general for all dimensions of faces. We also conjecture that it is $\#P$-hard to compute the best projection. This is suggested by the fact that it is hard to compute the volume of \emph{zonotopes} as this implies that computing the volumes of projections of polytopes is hard \cite{DyerGritzmannHufnagel98}. 

\begin{conjs*} Let $P$ be a polytope. Then,
\begin{enumerate}[label=\textup{(}\roman*\textup{)}]
\item it is $NP$-hard to find a hyperplane $H$ which maximizes the number of $i$-dimensional faces of $P\cap H$.
\item it is $\#P$-hard to find a hyperplane $H$ which maximizes the volume of the orthogonal projection of $P$ onto $H$.
\item it is $NP$-hard to find a central hyperplane $H_0$ which maximizes the volume of the halfspace section $P\cap H^+_0$.
\end{enumerate}
\end{conjs*}

\section{Experimental results}\label{section:applications}

In this last section, we present some explicit computations carried out using the algorithms from the proofs of \Cref{thm:integral_polynomial_over_P,thm:parallel_sections,thm:piecewise_rat_translate_rotation,thm:piecewise_rat_rotate_translation}. We present slices of maximal volume of all five Platonic solids, optimal slices of the $3$-dimensional permutahedron for various optimality criteria, and present the different combinatorial types of polytopes which can occur as affine hyperplane sections of the cross-polytope of dimension $4$ and $5$. 

Taking into account the complexity analysis from \Cref{sec:complexity}, we chose to use the approach from \Cref{sec:rotating-the-translation} to compute general affine hyperplane sections. Since the maximum volume slice of a centrally symmetric polytope is always a central section, we used the approach from \Cref{sec:rotational_slices} to find the slice of maximum volume for the centrally symmetric Platonic solids. 
We implemented the algorithm in \texttt{SageMath (version 9.2)} \cite{sagemath} for all approaches, which computes the respective arrangements and the rational functions, or representatives of the combinatorial properties respectively. 
The implementations of all algorithms are available upon request.
% All these implementations can be found on \texttt{MathRepo} \cite{mathrepo}.
The maximization relies on the \texttt{Maximize}-command of \texttt{Mathematica (version 13.2)} \cite{mathematica}.

We emphasize that all computations have been performed on an ordinary laptop, and that our implementations can be largely optimized. For example, the computation for each region is independent and thus, after finishing the computation of the regions, this process can be parallelized.

\begin{example}[Permutahedron]\label{ex:permu}
    Let $P\subset \R^4$ be the $3$-dimensional permutahedron defined as the convex hull of the permutations of $(1,2,3,4)$. This is a $3$-dimensional polytope contained in the hyperplane $\{x_1+x_2+x_3+x_4 =10\}\subset \R^4$ having Euclidean volume $32$. One of the hyperplane sections of $P$ with maximum volume is the convex hull of the points
    \[
    (1,2,3,4), (1,3,2,4), (2,4,1,3), (3,4,1,2), (4,3,2,1), (4,2,3,1), (3,1,4,2), (2,1,4,3).
    \]
    The other such slices can be obtained by symmetry. 
    The hyperplane in the affine span of $P$ which produces that section has equation 
    \[
    \{x_1+x_2+x_3+x_4 =10\} \cap \{x_1+x_4 =5\},
    \]
    and the volume of the section is $14$. This is visualized in \Cref{fig:permu-vol-max} in the introduction. 
    
    On the other hand, the slice of $P$ through the origin having minimal volume is provided by the intersection of $P$ with
    \[
    \{x_1+x_2+x_3+x_4 =10\} \cap \{x_1-x_2 =0\},
    \]
    and the volume of this polygon is $8\sqrt{2}$. It is the convex hull of the points 
    \[
    (\tfrac{3}{2}, \tfrac{3}{2}, 3, 4), (\tfrac{3}{2}, \tfrac{3}{2}, 4, 3), 
    (\tfrac{5}{2}, \tfrac{5}{2}, 1, 4), (\tfrac{5}{2}, \tfrac{5}{2}, 4, 1),
    (\tfrac{7}{2}, \tfrac{7}{2}, 1, 2), (\tfrac{7}{2}, \tfrac{7}{2}, 2, 1),
    \]
    displayed in \Cref{fig:permu-vol-min}.
    Notice that this is the slice of the permutahedron fixed by the permutation $\sigma = (1 2)$, object of interest for \cite{Ardilaetal2021slicingpermutahedron}.
    
    From a purely combinatorial point of view, there are eight different types of slices of the permutahedron. These are polygons with $3,4,\ldots,10$ vertices. A section with $10$ vertices is shown in \Cref{fig:permu-vcs}. \Cref{fig:permu_all_sections} presents the three optimal slices discussed in this example, projected onto their affine spans.
    \begin{figure}[ht]
        \centering
        \begin{subfigure}[t]{0.3\textwidth}
            \centering
            \includegraphics[width=0.7\textwidth]{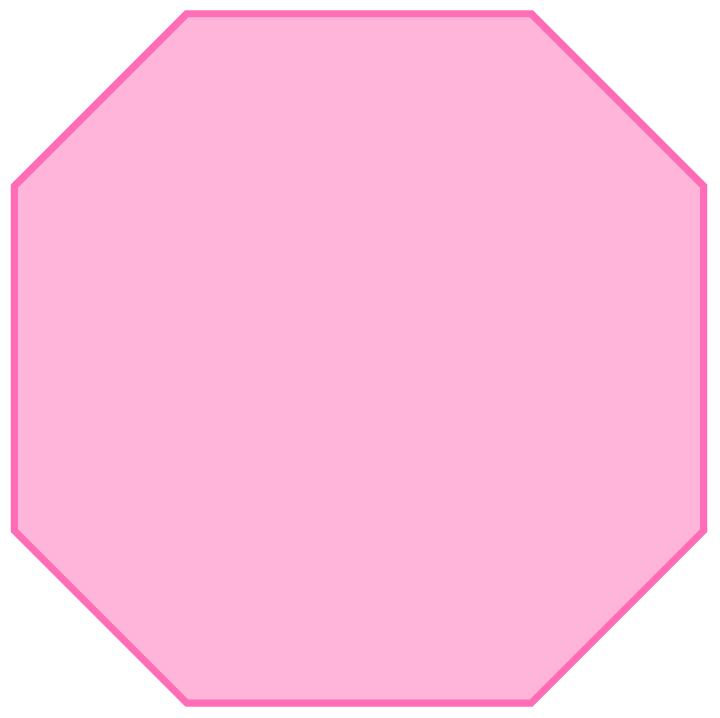}
            \caption{Slice of maximum volume.}
        \end{subfigure}
        \begin{subfigure}[t]{0.3\textwidth}
            \centering
            \includegraphics[width=0.745\textwidth]{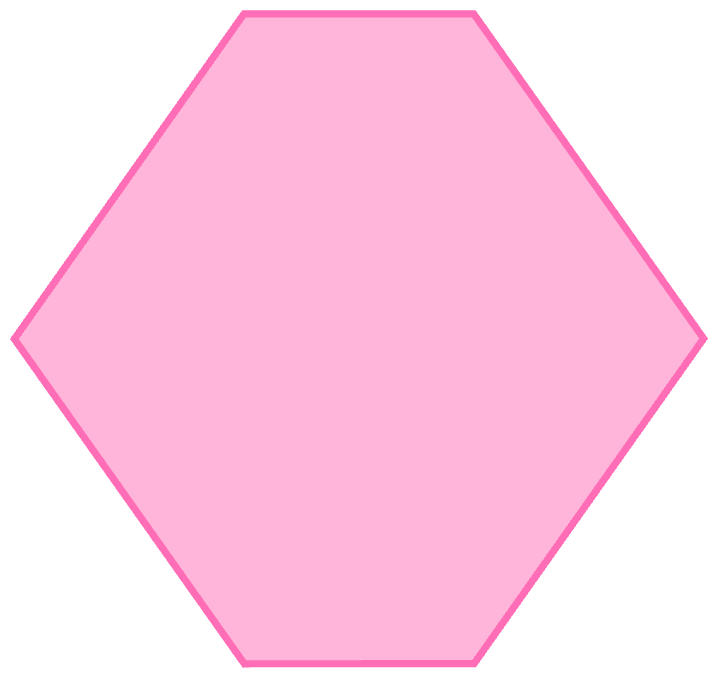}
            \caption{Slice of minimum volume containing the origin.}
        \end{subfigure}
        \begin{subfigure}[t]{0.3\textwidth}
            \centering
            \includegraphics[width=0.7\textwidth]{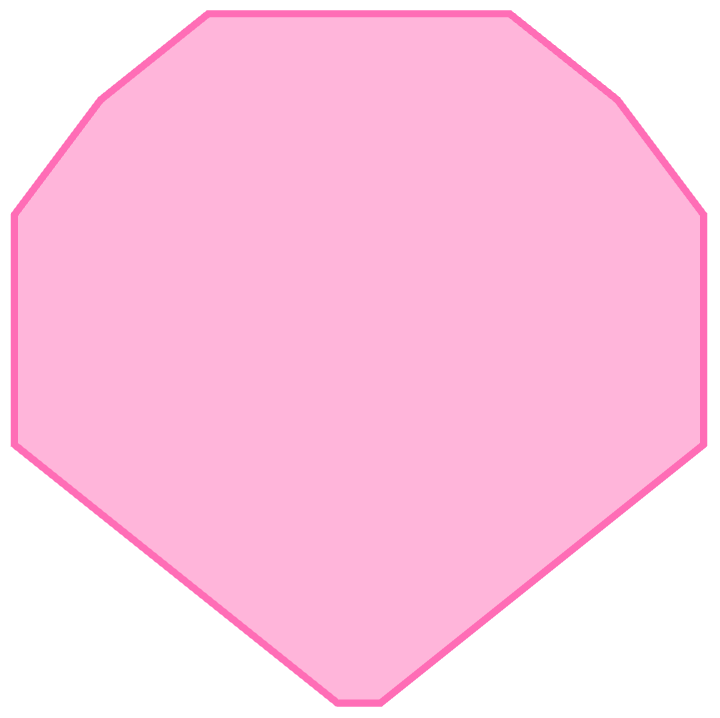}
            \caption{Slice with maximum number of vertices.}
        \end{subfigure}
        \caption{Optimal slices of the $3$-permutahedron for different optimality criteria.}
        \label{fig:permu_all_sections}
    \end{figure}
\end{example}

\begin{example}[Platonic solids and maximum volume]
Using our algorithm, we can find the slice with the largest volume of the Platonic
solids. Due to their symmetries, such a hyperplane is not unique. We give here the description of one possible answer. The others can be recovered using the symmetries of the polytopes.
We summarize our findings in \Cref{tab:platonic_solids} and we visualize them in \Cref{fig:PS}. We point out the surprising case of the dodecahedron, for which the slice with the largest volume does not contain any vertex of the dodecahedron.
\end{example}

\begin{figure}[!h]
    \centering
    \includegraphics[width=0.19\textwidth]{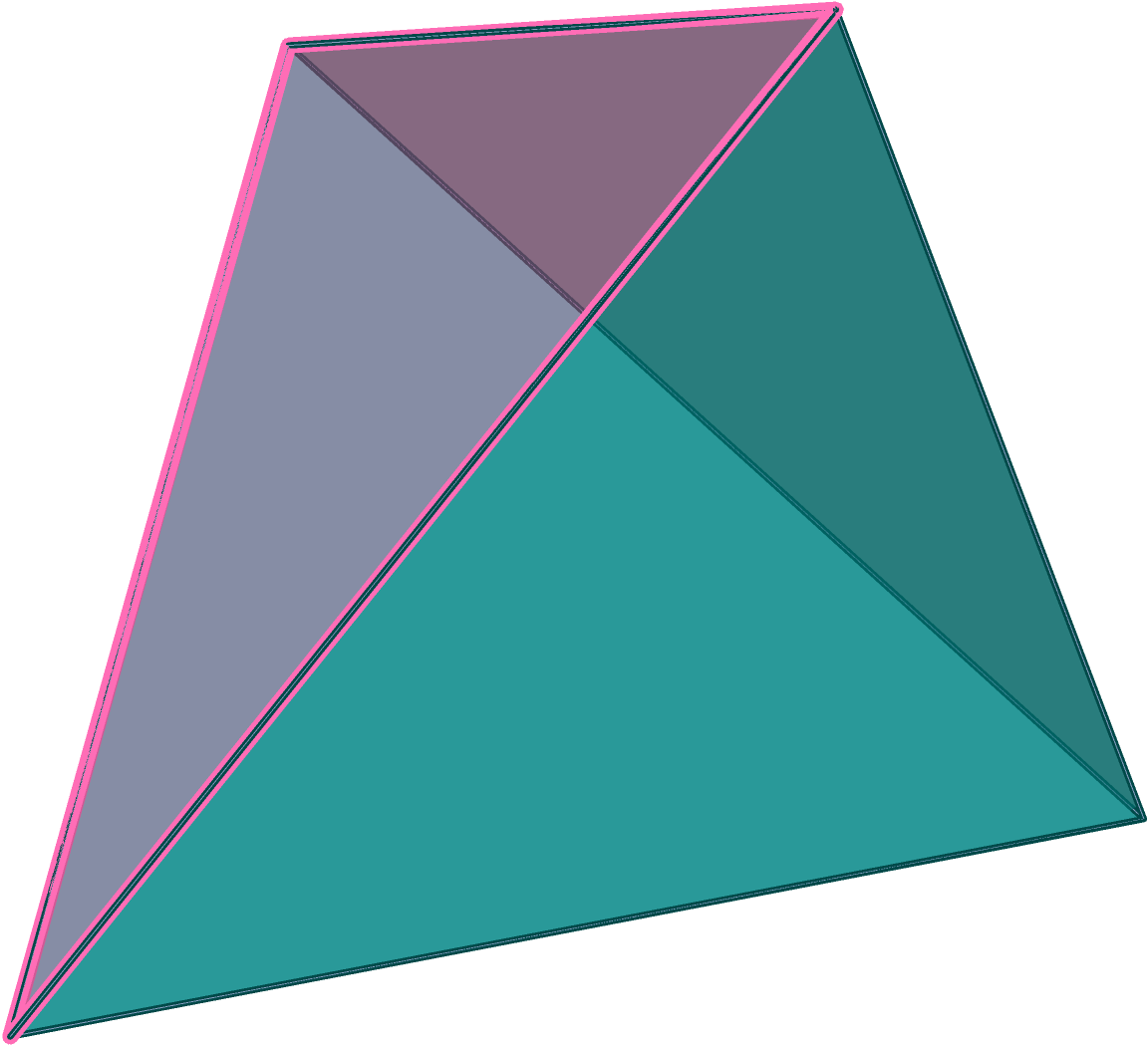}\qquad %0.25
    \includegraphics[width=0.18\textwidth]{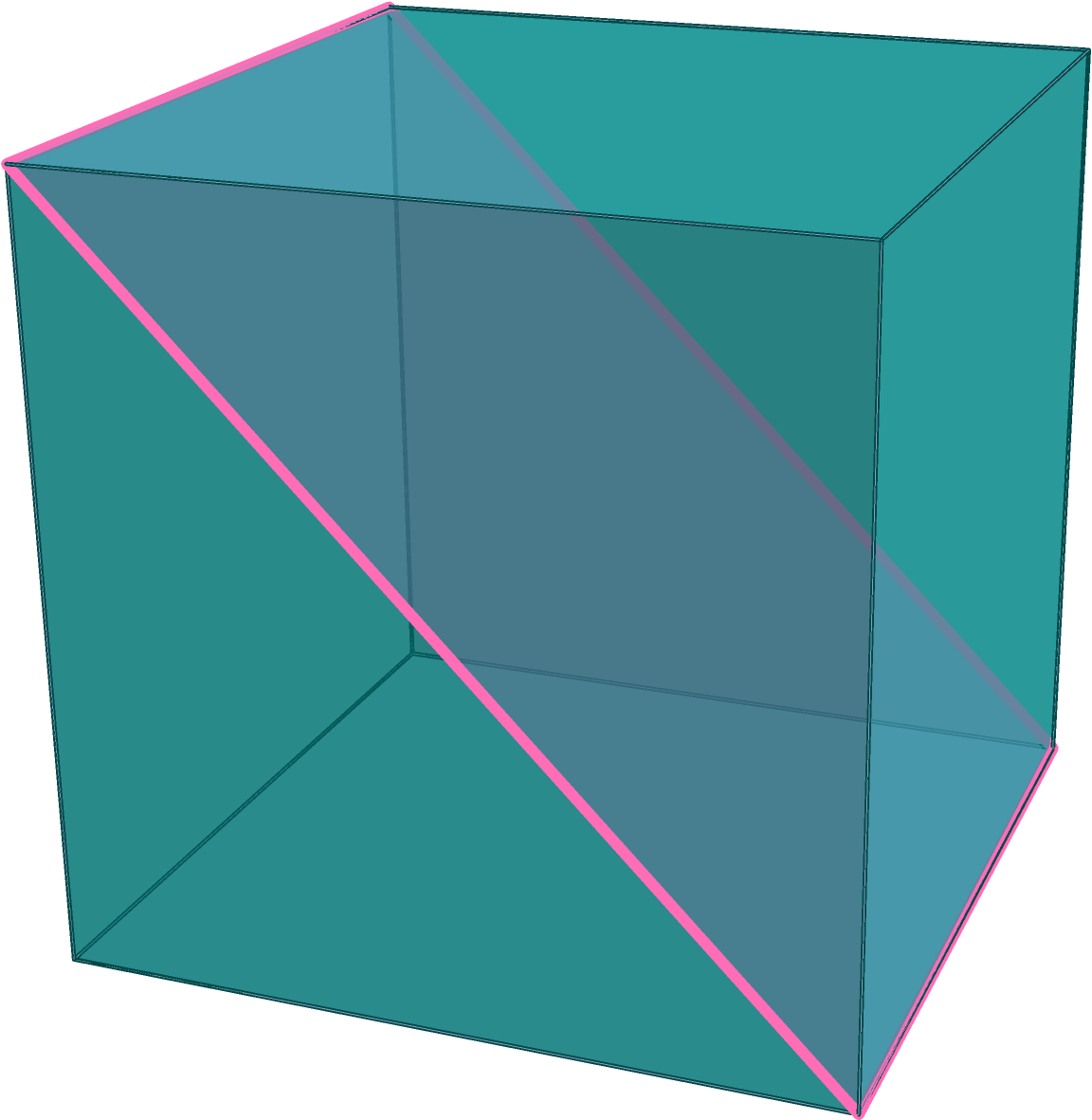}\qquad %0.22
    \includegraphics[width=0.2\textwidth]{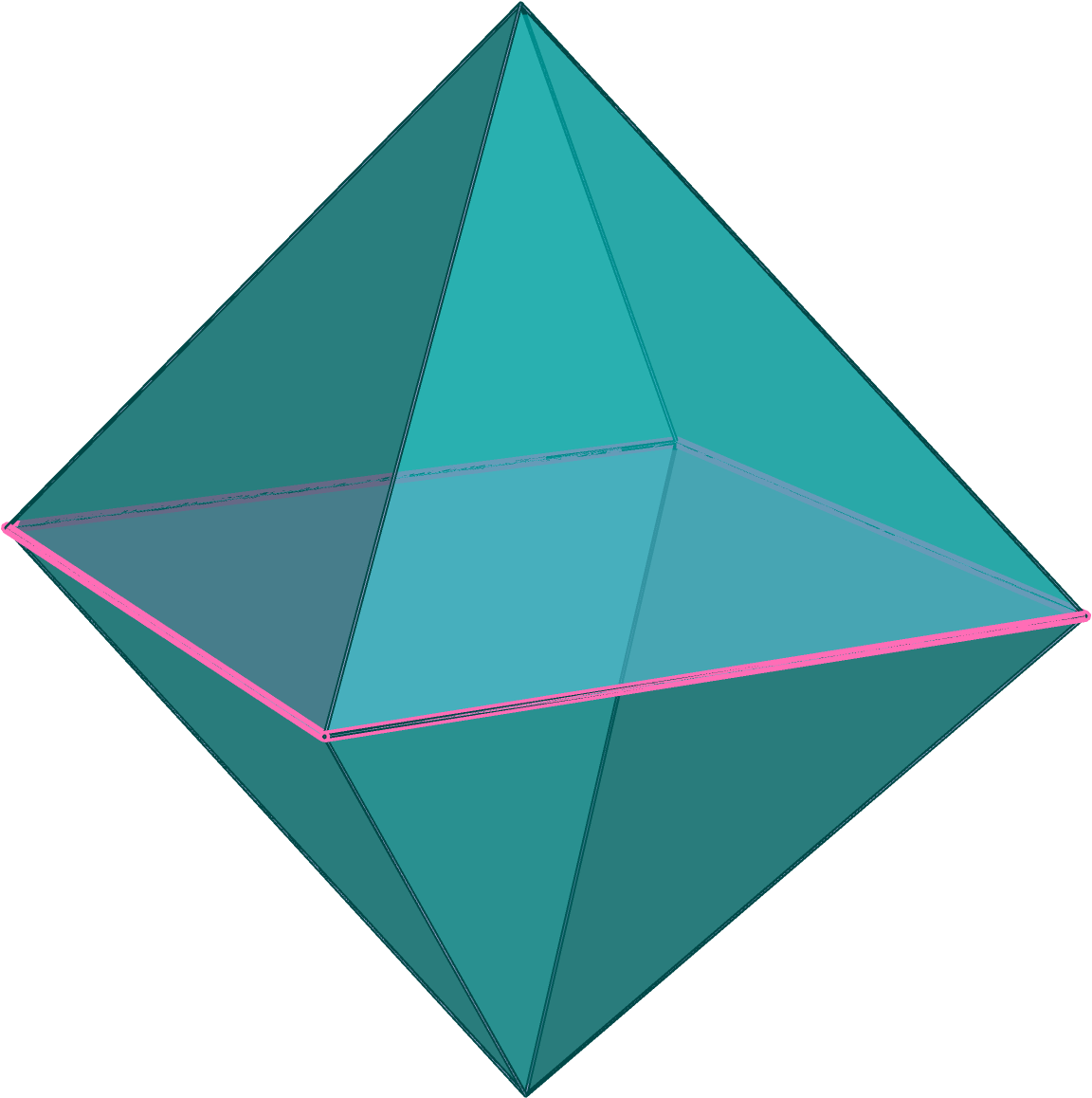}\\ %0.25
    \includegraphics[width=0.23\textwidth]{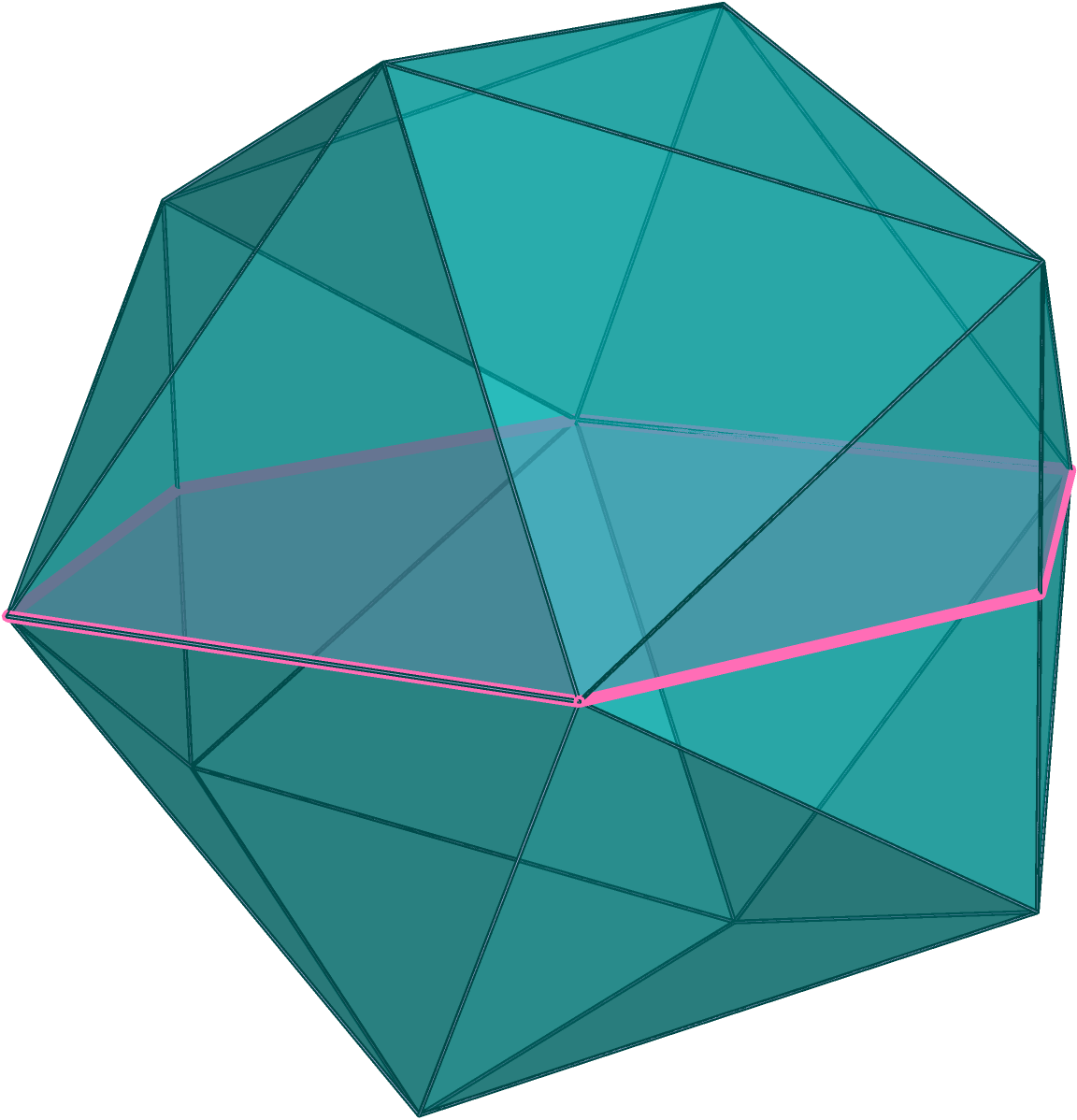}\qquad %0.3
    \includegraphics[width=0.23\textwidth]{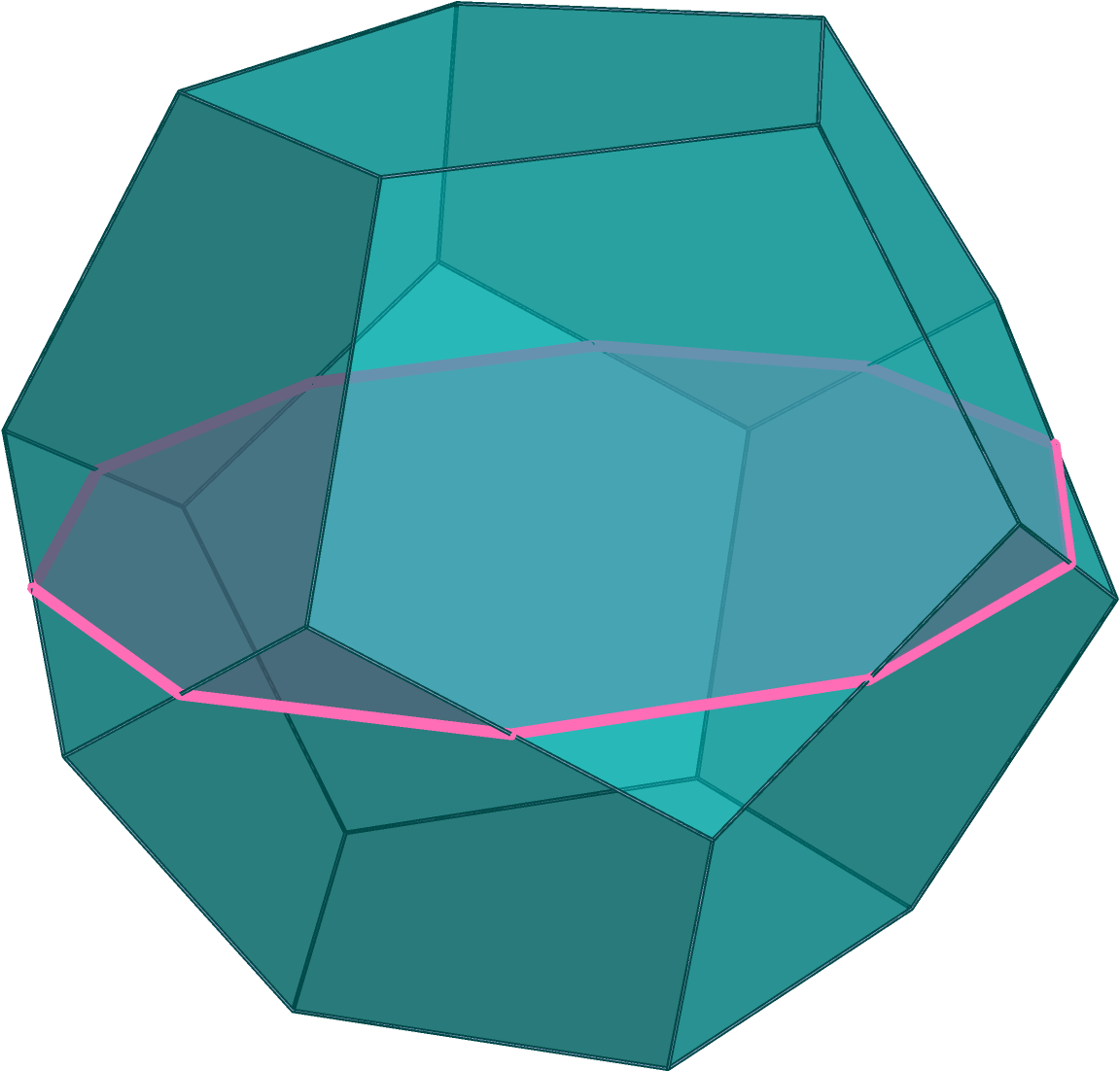} %0.3
    \caption{The Platonic solids and their pink hyperplane section (unique up to symmetry) with the largest volume.}
    \label{fig:PS}
\end{figure}
\begin{table}[!ht]
    \scriptsize
    \centering
    {\def\arraystretch{1.8}
    \begin{tabular}{c||c|c|c|c}
        $P$ & vertices of $P$ & $\vol P$ & hyperplane $H$ & $\vol (P\cap H)$ \\
        \hline
        \hline
        \multirow{2}{*}{tetrahedron} & $(\pm 1, 0, -\frac{1}{\sqrt{2}})$, & \multirow{2}{*}{$\frac{2 \sqrt{2}}{3}$} & \multirow{2}{*}{$2 x +1 =\sqrt{2} z$} & \multirow{2}{*}{$\sqrt{3}$} \\
        & $(0,\pm 1,\frac{1}{\sqrt{2}})$ & &  \\[1ex]
        \hline
        cube & $(\pm 1, \pm 1, \pm 1)$ & $8$ & $x+y=0$ & $4\sqrt{2}$ \\
        \hline
        \multirow{3}{*}{octahedron} & $(\pm 1, 0, 0)$, & \multirow{3}{*}{ $\frac{4}{3}$ } & \multirow{3}{*}{$z = 0$} & \multirow{3}{*}{ $2$ } \\
         & $(0, \pm 1, 0)$, & & & \\
         & $(0, 0, \pm 1)$ & & & \\
        \hline
        \multirow{3}{*}{icosahedron} & $\big( 0, \pm \frac{1}{2}, \pm (\frac{\sqrt{5}+1}{4}) \big)$, & \multirow{3}{*}{$\frac{5 \left(\sqrt{5}+3\right)}{12}$} & \multirow{3}{*}{$z = 0$} & \multirow{3}{*}{$\frac{\left(\sqrt{5}+1\right) \left(\sqrt{5}+3\right)}{8} \sim 2.12$} \\
        & $\big( \pm(\frac{\sqrt{5}+1}{4}), 0, \pm \frac{1}{2} \big)$, & & & \\
        & $\big( \pm \frac{1}{2}, \pm (\frac{\sqrt{5}+1}{4}), 0 \big)$ & & & \\
        \hline
        \multirow{4}{*}{dodecahedron} & $(0, \pm (\sqrt{5} - 1), \pm(2\sqrt{5} - 4))$ & \multirow{4}{*}{$400-176\sqrt{5}$} & \multirow{4}{*}{$x = \frac{\sqrt{5} - 1}{2}z$} & \multirow{2}{*}{$\frac{320 (15127 \sqrt{5} - 33825)}{(3\sqrt{5} - 7)^4 (\sqrt{5} - 1)\sqrt{-2\sqrt{5} + 10}} $} \\
        & $(\pm(2\sqrt{5} - 4), 0, \pm (\sqrt{5} - 1))$ & & & \\
        & $(\pm (\sqrt{5} - 1), \pm(2\sqrt{5} - 4), 0)$ & & & \multirow{2}{*}{$\sim 4.49$} \\
        & $(\pm (\sqrt{5} - 3),\pm( 3-\sqrt{5}), \pm( 3-\sqrt{5}))$ & & & \\
    \end{tabular}
    }
    \caption{The sections of maximum volume of the Platonic solids.}
    \label{tab:platonic_solids}
\end{table}

\vspace*{1em}
\begin{example}[Cross-polytope and combinatorial types]\label{ex:cross-polytope}
    How many combinatorial types of affine hyperplane sections $P\cap H$ can occur when $P$ is a cross-polytope? Using our algorithm, we obtained that if $P$ is the $3$-dimensional cross-polytope, its ($2$-dimensional) hyperplane sections can be of four combinatorial types: a triangle, a quadrilateral, a pentagon or an hexagon. We note that the pentagon can only be obtained when the defining hyperplane contains a vertex of the cross-polytope. 
    
    If $P$ is the $4$-dimensional cross-polytope, then we get nine different combinatorial types, summarized in \Cref{tab:sections_4crosspolytope}.
    We compute this by first computing the sweep arrangement $\regtra(P)$, which consists of $384$ maximal regions, and then sample a direction $\bu \in R$ for each of the $1696$ cones in $\regtra(P)$ that have dimension at least $1$. We then construct the arrangement $\chamtra(P)$ of parallel hyperplanes, and store the combinatorial type of one hyperplane section for each cell of the arrangement.

    The sweep arrangement of the $5$-dimensional cross-polytope has $24482$ cones of positive dimension, $3840$ of which are of dimension $5$. After running our algorithm, we get only $14$ distinct combinatorial types of hyperplane sections. These are summarized in \Cref{tab:5-cross-poly}.
    We point out that both in the case of dimension $4$ and $5$, there are exactly two combinatorial types of sections of the cross-polytope whose $f$-vectors agree. These maximize the number of $k$-dimensional faces for all $k$.
    
    Finally, we note the duality between intersections and projections, as discussed in \Cref{section:projections}. Since the cross-polytope is polar to the cube, the duality is inherited by their sections and projections. Therefore, we deduce that there are $5, 9, 14$ combinatorial types of projections of the $3, 4, 5$-dimensional cube onto a hyperplane, respectively. 
\end{example}

\newpage
\null\vfill
\begin{table}[ht]
    \centering
    {\def\arraystretch{1.5}
    \begin{tabular}{>{\centering\arraybackslash}m{1.5cm}|| >{\centering\arraybackslash}m{3.8cm} |>{\centering\arraybackslash}m{3.5cm}|>{\centering\arraybackslash}m{3.5cm} }
        $P\cap H$ & \includegraphics[height=2.2cm]{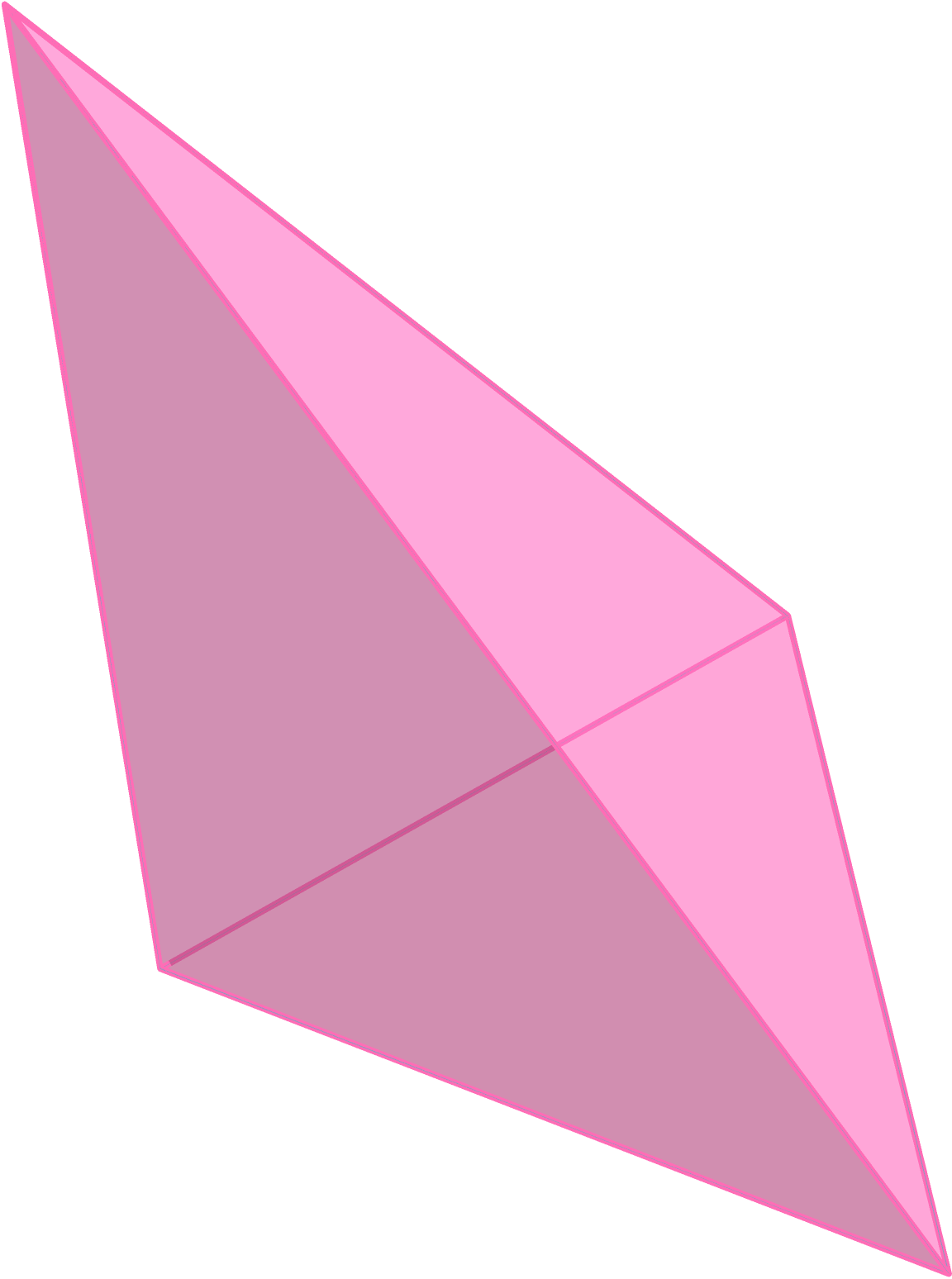} & \includegraphics[height=2.2cm]{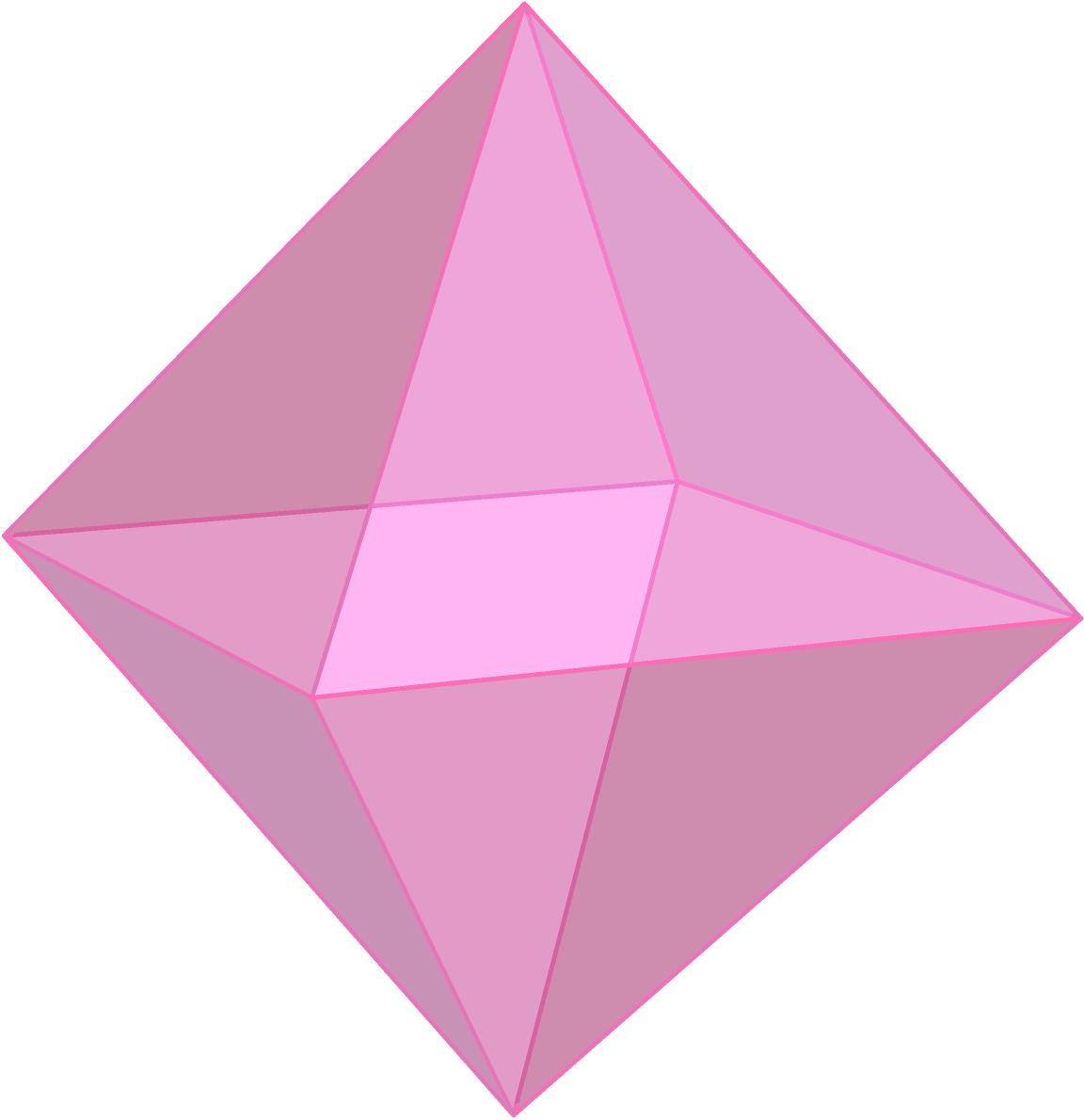} & \includegraphics[height=2.2cm]{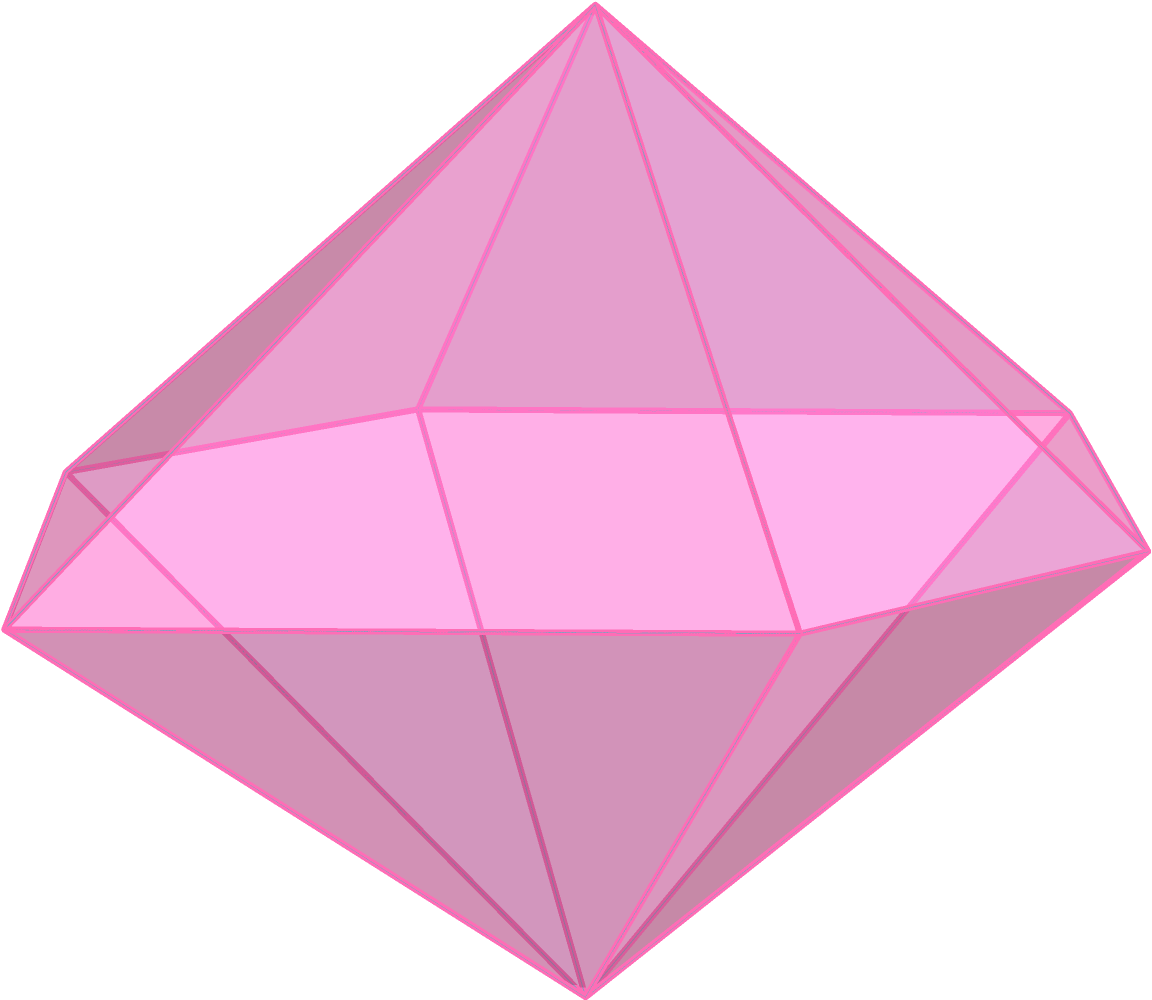} \\
         \hline
        $f$-vector & $(4,6,4)$ & $(6,12,8)$ & $(8,18,12)$ \\
        \hline
        $H$ & \small $x_1+x_2+x_3+x_4=1$ & \small $2 x_1 = 1$ & \small $x_1+x_2+x_3=0$
    \end{tabular}\\
    }
    \bigskip
    {\def\arraystretch{1.5}
    \begin{tabular}{>{\centering\arraybackslash}m{1.5cm}|| >{\centering\arraybackslash}m{3.8cm} |>{\centering\arraybackslash}m{3.5cm}|>{\centering\arraybackslash}m{3.5cm} }
        $P\cap H$ & \includegraphics[height=1.8cm]{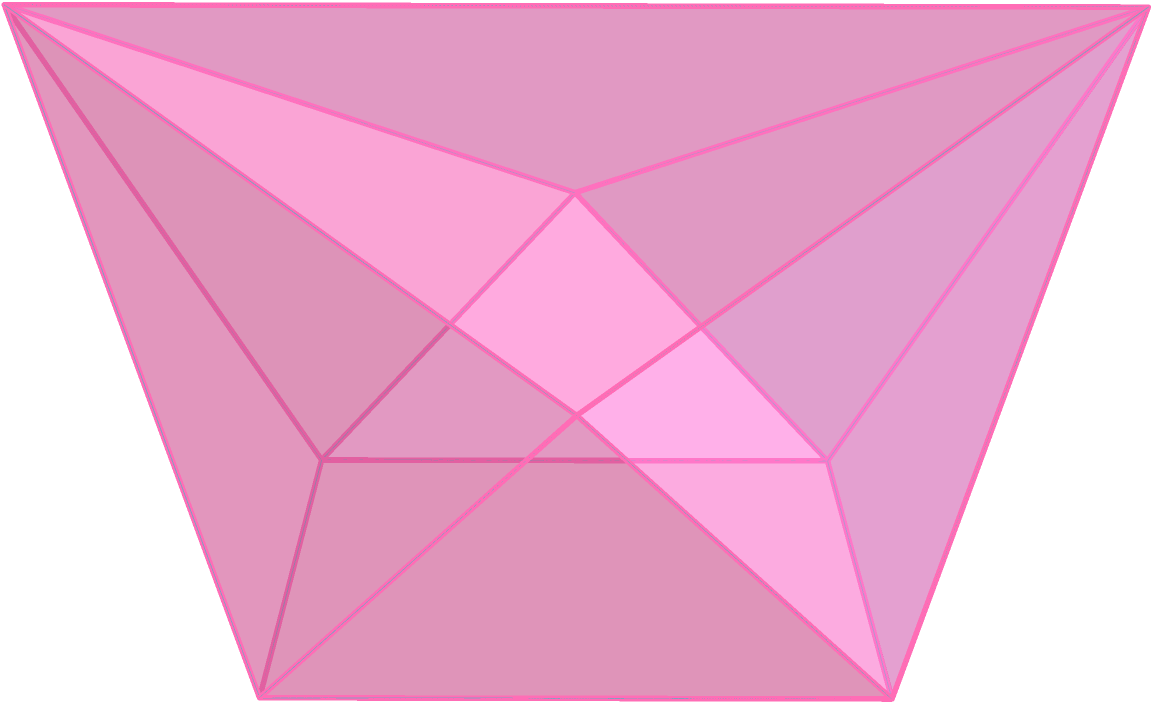} & \includegraphics[height=2.2cm]{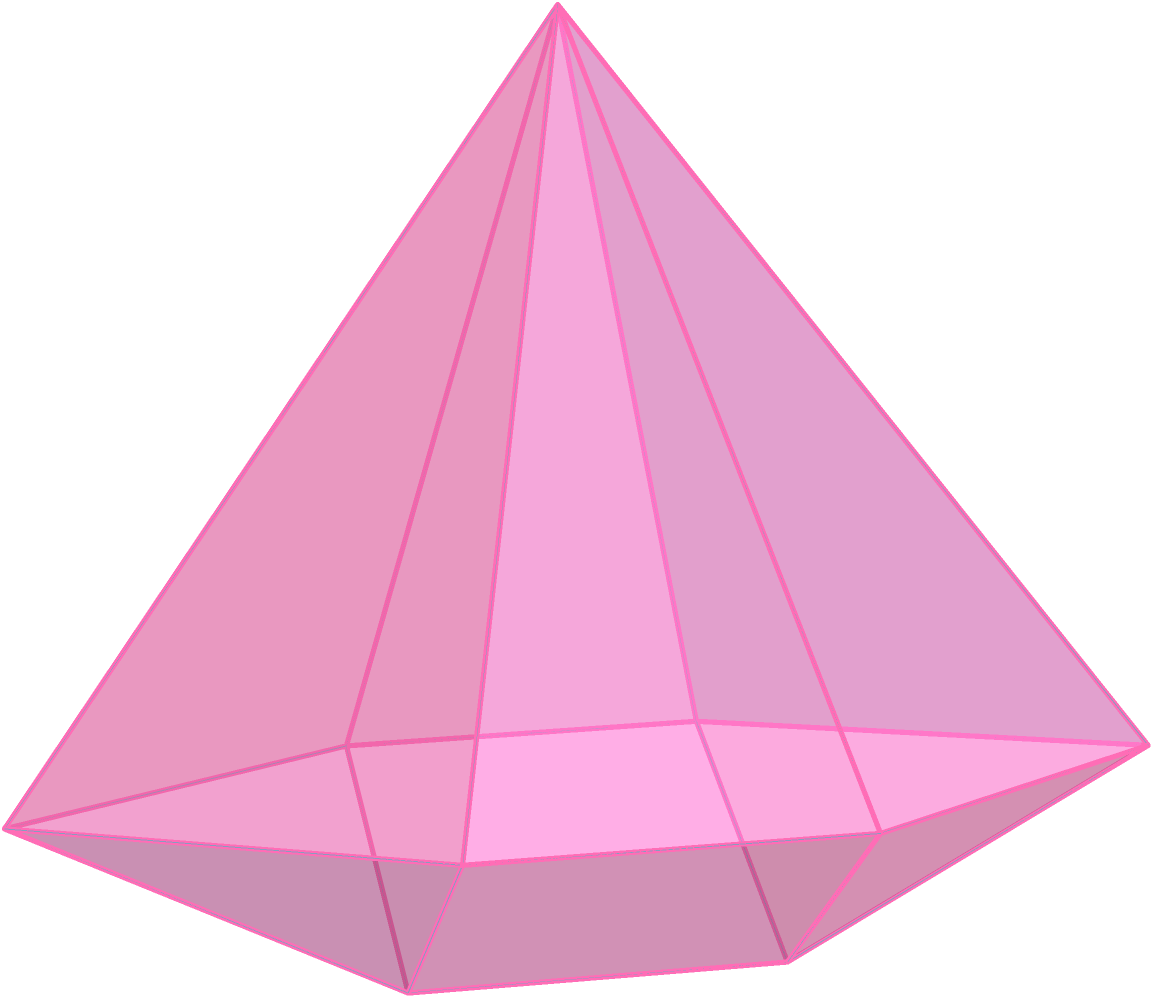} & \includegraphics[height=2.2cm]{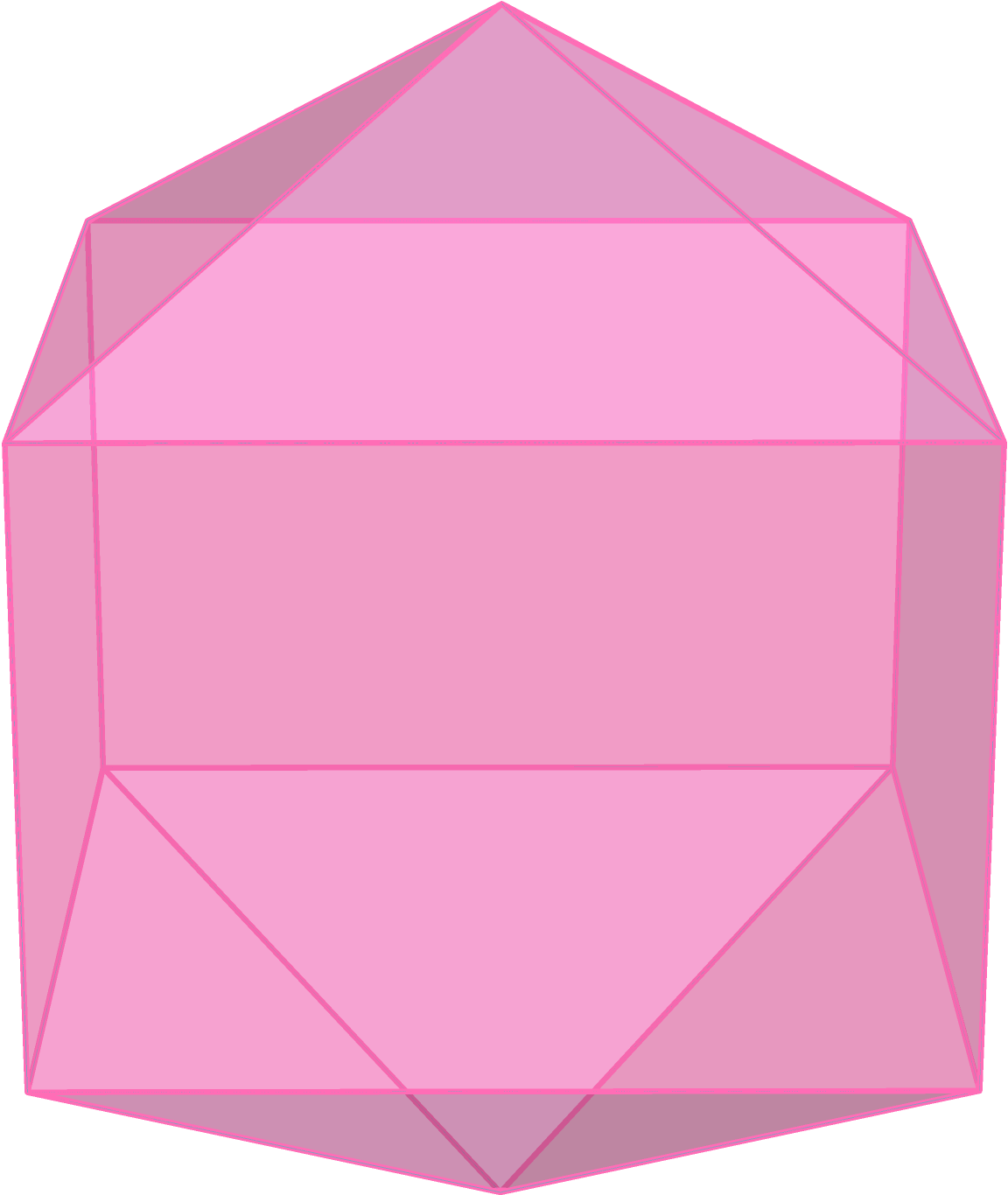} \\
         \hline
        $f$-vector & $(8,17,11)$ & $(9,19,12)$ & $(10,20,12)$ \\
        \hline
        $H$ & \small $2 x_1 + 2 x_2 + x_3 + x_4 = 1$ &  \small $2 x_1 + 2 x_2 + x_3 = 1$ & \small $2 x_1+2 x_2=1$
    \end{tabular}\\
    }
    \bigskip
    {\def\arraystretch{1.5}
    \begin{tabular}{>{\centering\arraybackslash}m{1.5cm}|| >{\centering\arraybackslash}m{3.8cm} |>{\centering\arraybackslash}m{3.5cm}|>{\centering\arraybackslash}m{3.5cm} }
        $P\cap H$ & \includegraphics[height=2.2cm]{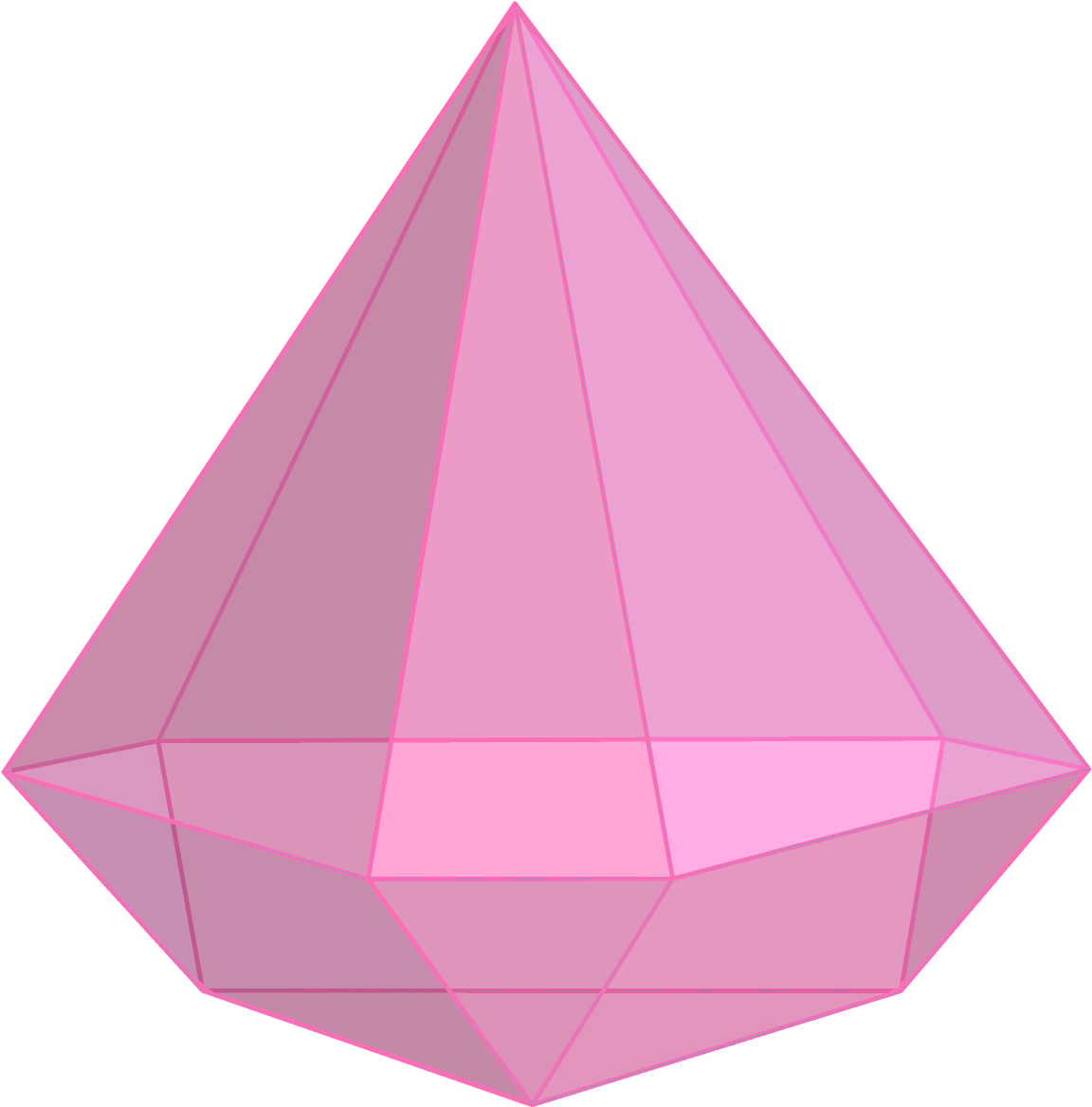} &\includegraphics[height=2.2cm]{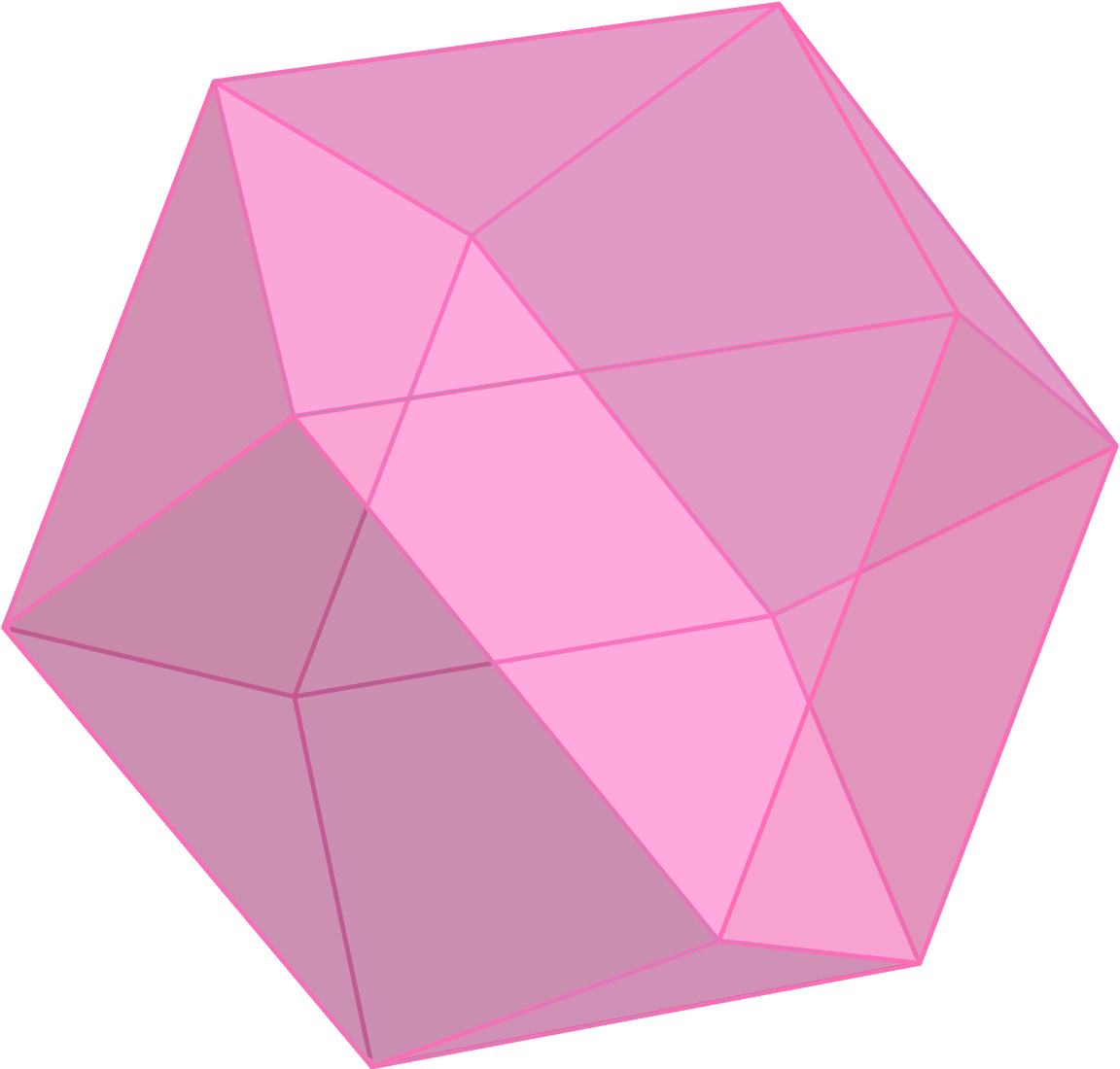} & \includegraphics[height=2cm]{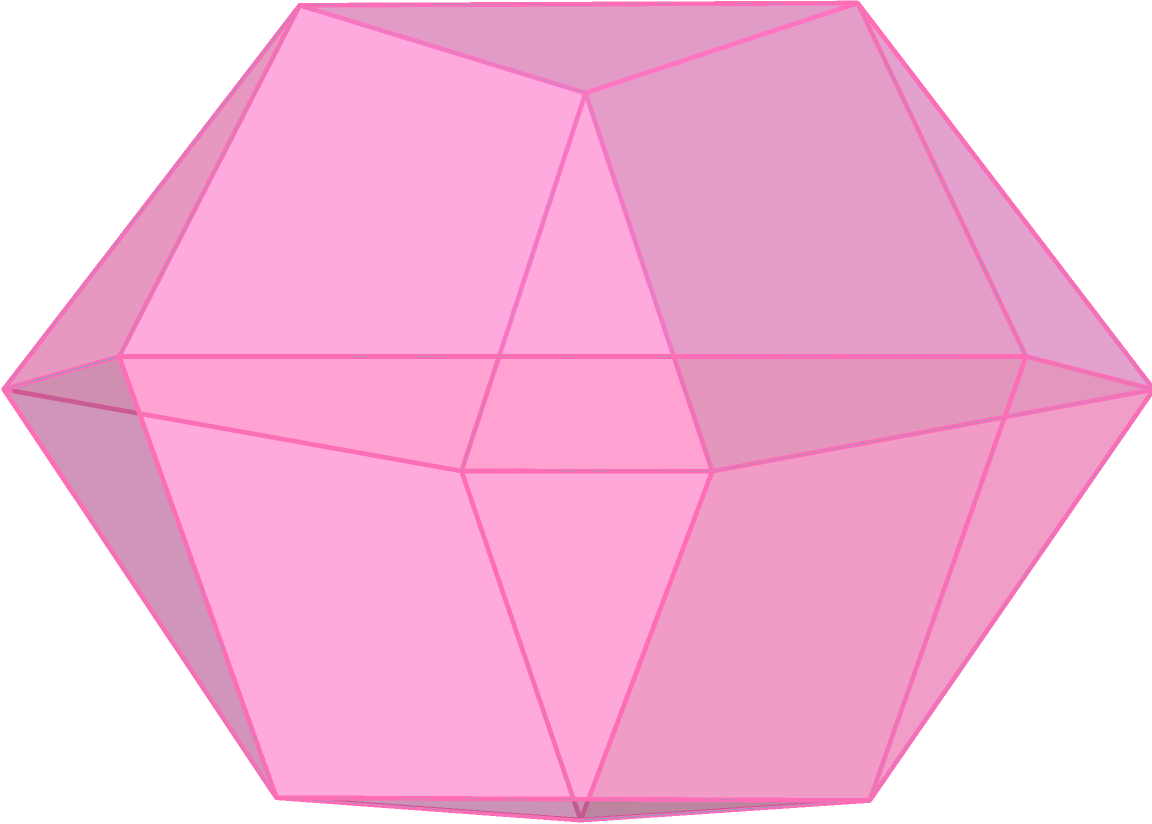} \\
         \hline
        $f$-vector & $(10,21,13)$ & $(12,24,14)$ & $(12,24,14)$ \\
        \hline
        $H$ & \small $2 x_1 + 2 x_2 + 2 x_3 + x_4 = 1$ & \small $x_1+x_2+x_3+x_4=0$ & \small $2 x_1+2 x_2 +2 x_3=1$
    \end{tabular}\\
    }
    \caption{The nine pink polytopes are all possible combinatorial types of hyperplane sections of the $4$-dimensional cross-polytope.}
    \label{tab:sections_4crosspolytope}
\end{table}

\null\vfill

\newpage
\begin{table}[!ht]
    \centering
    \begin{tabular}{>{\centering\arraybackslash}m{1.5cm}|| >{\centering\arraybackslash}m{2.5cm} |>{\centering\arraybackslash}m{2.5cm}|>{\centering\arraybackslash}m{2.5cm}|>{\centering\arraybackslash}m{2.5cm} }
       $P \cap H$  
        & \includegraphics[height=2.2cm]{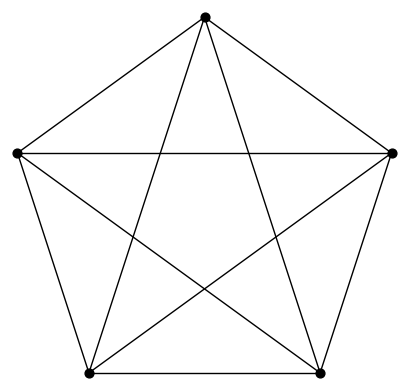} & \includegraphics[height=2.2cm]{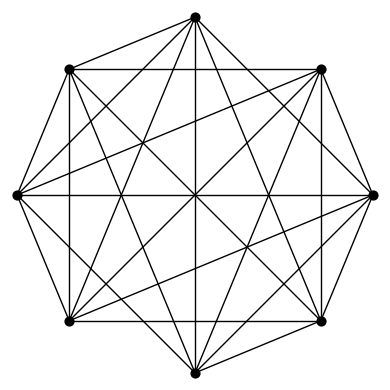}
        & \includegraphics[height=2.2cm]{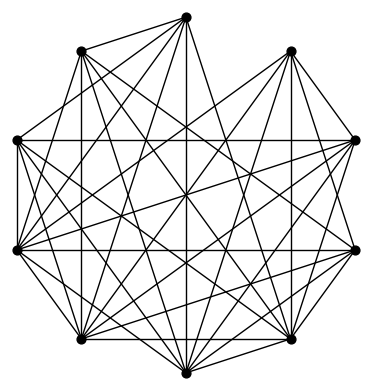}
        & \includegraphics[height=2.2cm]{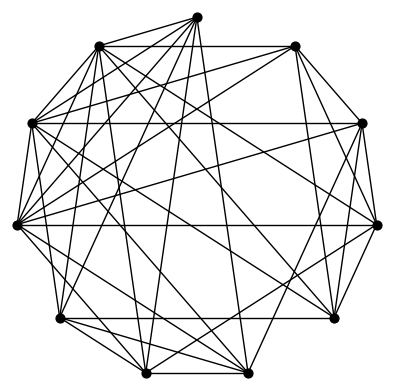}
        \\ \hline
       $f$-vector  
        & $(5,10,10,5)$ 
        & $(8, 24, 32, 16)$
        & $(10, 34, 48, 24)$
        & $(11, 36, 48, 23)$
        \\ \hline
       \begin{multirow}{2}{*}{$H$ } \end{multirow}
        & $x_1 + x_2 + x_3$
        & \begin{multirow}{2}{*}
            {$2x_1 = 1$}
        \end{multirow}
        & $x_1 + x_2$
        & $2x_1 + 2x_2 + x_3$  \\
        & $+ x_4 + x_5 = 1$
        & & $+ x_3  =  0$ &
        $ + x_4 + x_5  =  1$
    \end{tabular}\\
    \medskip
    \begin{tabular}{>{\centering\arraybackslash}m{1.5cm}|| >{\centering\arraybackslash}m{2.5cm} |>{\centering\arraybackslash}m{2.5cm}|>{\centering\arraybackslash}m{2.5cm}|>{\centering\arraybackslash}m{2.5cm} }
       $P \cap H$  
        & \includegraphics[height=2.2cm]{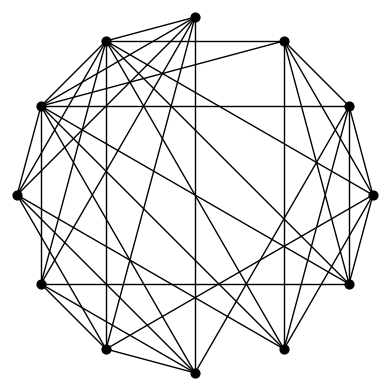} & \includegraphics[height=2.2cm]{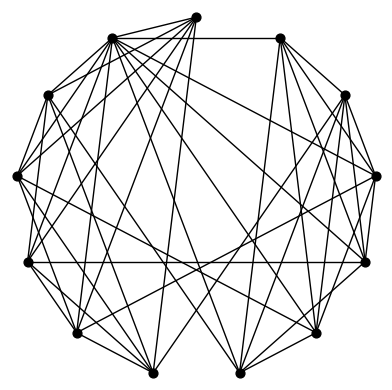}
        & \includegraphics[height=2.2cm]{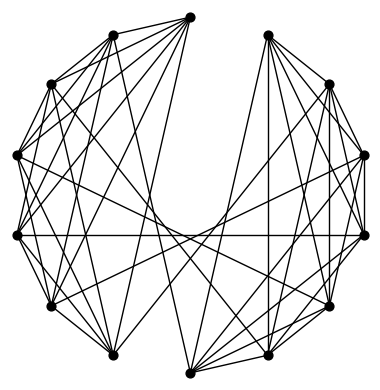}
        & \includegraphics[height=2.2cm]{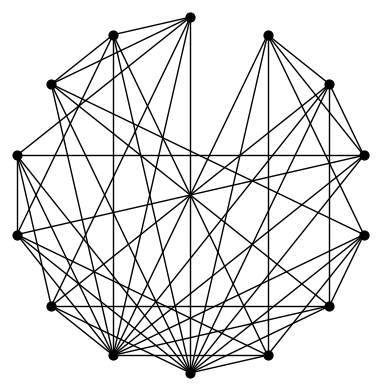}
        \\ \hline
       $f$-vector  
        & $(12, 39, 51, 24)$
        & $(13, 41, 52, 24)$
        & $(14, 42, 52, 24)$
        & $(14, 48, 62, 28)$
        \\ \hline
       \begin{multirow}{2}{*}{$H$} \end{multirow}
        & $2x_1 + 2x_2$
        & $2x_1 + 2x_2$
        & \begin{multirow}{2}{*}
            {$2x_1 + 2x_2  =  1$}
        \end{multirow}
        & $x_1 + x_2$ \\ 
        & $+ x_3 + x_4  =  1$
        & $ + x_3  =  1$
        & & $+ x_3 + x_4  =  0$
    \end{tabular}\\
    \medskip
    \begin{tabular}{>{\centering\arraybackslash}m{1.5cm}|| >{\centering\arraybackslash}m{2.5cm} |>{\centering\arraybackslash}m{2.5cm}|>{\centering\arraybackslash}m{2.5cm}|>{\centering\arraybackslash}m{2.5cm} }
       $P \cap H$  
        & \includegraphics[height=2.2cm]{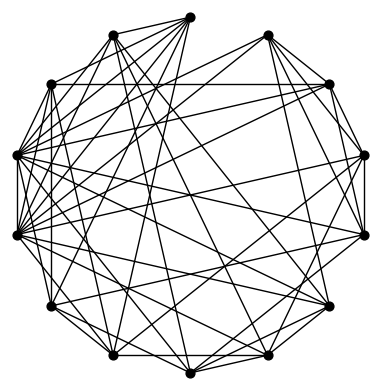} & \includegraphics[height=2.2cm]{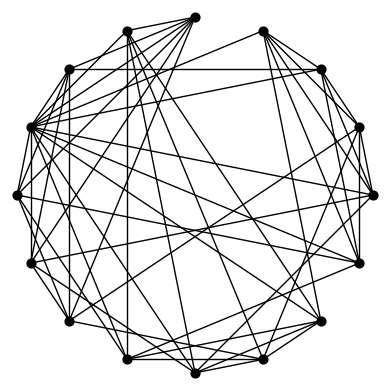}
        & \includegraphics[height=2.2cm]{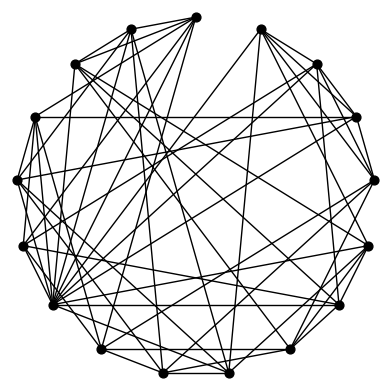}
        & \includegraphics[height=2.2cm]{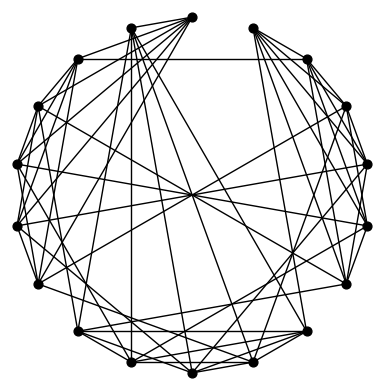}
        \\ \hline
       $f$-vector  
        & $(14, 46, 59, 27)$
        & $(16, 51, 63, 28)$
        & $(17, 54, 66, 29)$
        & $(18, 54, 64, 28)$
        \\ \hline
       \begin{multirow}{2}{*}{$H$}\end{multirow}
        & $2x_1 + 2x_2 + 2x_3$
        & $2x_1 + 2x_2 $
        & $2x_1 + 2x_2 + 2x_3$
        & $2x_1 + 2x_2 $ \\ 
        & $+ x_4 + x_5  =  1$
        & $+ 2x_3 + x_4  =  1$
        & $+ 2x_4 + x_5  =  1$
        & $+ 2x_3  =  1$
    \end{tabular}\\
    \medskip
    \begin{tabular}{>{\centering\arraybackslash}m{1.5cm}|| >{\centering\arraybackslash}m{5.44cm} |>{\centering\arraybackslash}m{5.44cm} }
       $P \cap H$  
        & \includegraphics[height=2.2cm]{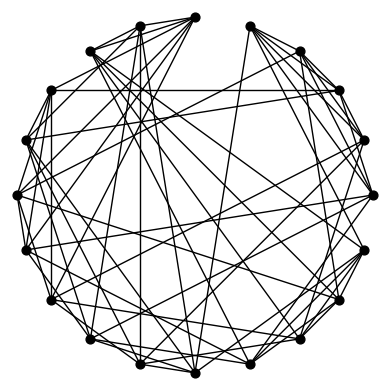} & \includegraphics[height=2.2cm]{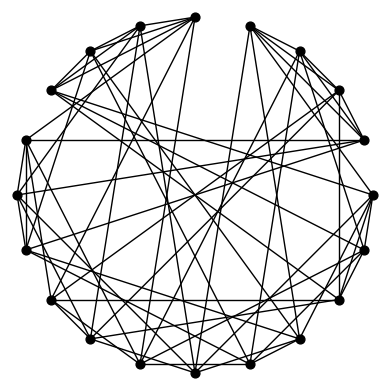}
        \\ \hline
       $f$-vector  
        & $(20, 60, 70, 30)$
        & $(20, 60, 70, 30)$
        \\ \hline
       $H$ & $2x_1 + 2x_2 + 2x_3 + 2x_4  =  1$
        & $x_1 + x_2 + x_3 + x_4 + x_5  =  0$
    \end{tabular}
    \caption{The edge-graphs of the combinatorial types of sections of the $5$-dimensional cross-polytope. The graphs alone distinguish the combinatorial type.}
    \label{tab:5-cross-poly}
\end{table}    
\vspace*{-1em}
\paragraph{Acknowledgements.} This work was started during a visit to the Max Planck Institute in Leipzig and then fully developed at the Institute for Computational and Experimental Research Mathematics of Brown University (an NSF institute supported by the National Science Foundation under Grant No. DMS-1929284), during the semester program ``Discrete Optimization: Mathematics,  Algorithms, and Computation''. Marie-Charlotte Brandenburg was funded by the Deutsche Forschungsgemeinschaft (DFG, German Research Foundation) – SPP 2298.
We are grateful to Francisco Criado Gallart for suggesting the proof of \Cref{lemma:pyramid_large_section} and allowing us to include it in our paper.
We want to thank Amitabh Basu, Saugata Basu, Peter B\"urgisser, Daniel Dadush, Jim Lawrence, Michael Roysdon, Martin Skutella, L\'aszl\'o V\'egh, Stefan Weltge, and other participants for suggestions and comments.

\printbibliography

\vspace*{\fill}

\noindent \textsc{Marie-Charlotte Brandenburg} \\
\textsc{ Max Planck Institute for Mathematics in the Sciences \\
Inselstra{\ss}e 22, 04103 Leipzig, Germany} \\
 \url{marie.brandenburg@mis.mpg.de} \\

\noindent \textsc{Jesús A. De Loera} \\
\textsc{Department of Mathematics\\ One Shields Avenue, Davis CA 95616, USA} \\
\url{deloera@math.ucdavis.edu} \\

\noindent \textsc{Chiara Meroni} \\
\textsc{ Institute for Computational and Experimental Research in Mathematics \\
121 South Main Street, Providence 02903, RI, USA} \\
\url{chiara\_meroni@brown.edu} \\

\end{document}